\documentclass[11pt]{article}
    \usepackage{url}
    \usepackage{verbatim}
    \usepackage[titletoc]{appendix}
    \usepackage{graphicx}
    \usepackage{amssymb}
	\usepackage{amsmath}
	\usepackage{amsthm}
	\usepackage{amsfonts}
	\usepackage{graphicx}
    \usepackage{epstopdf}
    \usepackage{nicefrac}
    \usepackage{longtable}
    \usepackage{color}
    \usepackage{graphicx, amssymb, subcaption,graphics}
    \usepackage[section]{placeins}
     \usepackage{booktabs}
     \usepackage{cases}
     \usepackage{float}
     \usepackage{tabu}

    \newcommand{\be}{\begin{equation}}
    \newcommand{\ee}{\end{equation}}

     \def\0{\mbox{\boldmath $0$}}

	\newtheorem{thm}{Theorem}[section]

	\newtheorem{rem}[thm]{Remark}
	\newtheorem{assumption}{Assumption}

\title{Second order stabilized semi-implicit scheme for the Cahn-Hilliard model with dynamic boundary conditions}

\date{\today}

\begin{document}
\author{
Xiangjun Meng \thanks{School of Mathematical Sciences, Beijing Normal University, Beijing 100875, China
(xjmengbnu@163.com)}
\and
Xuelian Bao \thanks{
School of Mathematical Sciences, Laboratory of Mathematics and Complex Systems, MOE,
Beijing Normal University, Beijing 100875, China.
Research Center for Mathematics and Mathematics Education,
Beijing Normal University at Zhuhai, Zhuhai 519087, China
(xlbao@bnu.edu.cn)}
\and
Zhengru Zhang \thanks{School of Mathematical Sciences, Beijing Normal University, Beijing 100875, China
(Corresponding Author: zrzhang@bnu.edu.cn)
}
}

\maketitle

\vspace{2mm}
\numberwithin{equation}{section}
	\maketitle
	\numberwithin{equation}{section}

\begin{abstract}
We study the numerical algorithm and error analysis for the Cahn-Hilliard equation with dynamic boundary
conditions. A second-order in time, linear and energy stable scheme is proposed, which is an extension of the first-order
stabilized approach. The corresponding energy stability and convergence analysis of the scheme are derived theoretically.
Some numerical experiments are performed to verify the effectiveness and accuracy of the second-order numerical scheme,
including numerical simulations under various initial conditions and energy potential functions, and comparisons with the
literature works.
\end{abstract}
	


\medskip

\section{Introduction}
The Cahn-Hilliard equation was originally introduced by Cahn and Hilliard \cite{cahn1958free} to describe phase
separation and coarsening in heterogeneous systems such as alloys, glass and polymer mixtures. The standard
Cahn-Hilliard equation can be written as follows:
\begin{numcases}{}
\phi_t=\Delta \mu,                   & $(\mathbf{x},t)\in \Omega \times (0,T)$,\nonumber\\
\mu=-\varepsilon \Delta \phi + \frac{1}{\varepsilon}F^\prime(\phi),   & $(\mathbf{x},t)\in \Omega \times (0,T)$,\nonumber
\end{numcases}
where the parameter $\varepsilon>0$ means the thickness of the interface,
$\Omega\subseteq \mathbb{R}^d(d=2,3)$ denotes a bounded domain whose
boundary $\Gamma=\partial \Omega$ with the unit outward normal vector $\bold{n}$. To describe binary alloys, the
function $\phi$ represents the difference between the two local relative concentrations. The area of $\phi=\pm 1$
corresponds to the pure phase of the materials, which are separated by an interfacial region whose thickness is
proportional to $\varepsilon$.

The Cahn-Hilliard equation can be alternatively viewed as the gradient flow of the Ginzburg-Landau type energy
functional
\[
E^{bulk}(\phi)=\int_\Omega\left\{\frac{\varepsilon}{2}|\nabla \phi|^2+\frac{1}{\varepsilon}F(\phi)\right\}dx,
\]
in $H^{-1}$.
$\mu$ denotes the chemical potential in $\Omega$, which can be expressed as the Fr\'{e}chet derivative of the bulk
free energy $E^{bulk}$.  The term $f (x) = F^\prime(x)$ with $F(x)$ being a given double-well potential as
\begin{eqnarray}
	F(x)=\frac{1}{4}(x^2-1)^2, \quad f(x)=x^3-x, \quad x\in \mathbb{R}.
\end{eqnarray}
When the time evolution of  $\phi$ is limited to a bounded region, the appropriate boundary conditions are required.
The classical choice is homogeneous Neumann condition:
\begin{numcases}{}
		\partial_{\bold n} \mu=0,                   & $({\bf x},t)\in \Gamma \times (0,T)$,\nonumber\\
			\partial_{\bold n} \phi=0,                   & $(\mathbf{x},t)\in \Gamma \times (0,T)$,\nonumber
\end{numcases}
where $\partial_{\bold n}$ represents the outward normal derivative on $\Gamma$. The two most important properties of
Cahn-Hilliard equation are the conservation of mass
\[
\int_{\Omega}\phi(t)dx=\int_{\Omega}\phi(0)dx ,\quad\quad\forall t \in [0,T],
\]
and energy decreasing
\[
\frac{d}{dt}E^{bulk}(\phi)=-\|\nabla\mu\|^2_{\Omega}\leq 0.
\]
When considering some special applications (for example, the hydrodynamics applications, such as contact line
problem), it is necessary to describe the short-range interaction between the mixture and the solid wall.
However, the standard homogeneous Neumann condition ignores the influence of boundary on volume dynamics. Therefore,
the researchers added surface energy into the total energy in recent years,
\begin{eqnarray}
	E^{total}(\phi,\psi)=E^{bulk}(\phi)+E^{surf}(\psi),
\end{eqnarray}
with
\[
	E^{surf}(\psi)=\int_\Gamma\left\{\frac{\delta\kappa}{2}|\nabla_\Gamma \psi|^2+\frac{1}{\delta}G(\psi)\right\}dS,
\]
where $\delta$ denotes the thickness of the interface area on the boundary and the parameter $\kappa$ is related to the surface
diffusion. If $\kappa=0$,  it is related to the moving contact line problem. $G$ is the surface potential,
$\nabla_\Gamma$ represents the tangential surface gradient operator and $\Delta_\Gamma$ denotes the Laplace-Beltrami
operator on $\Gamma$. Several dynamic boundary conditions have been proposed and analyzed, for example, \cite{colli2017class,gal2006cahn,knopf2020convergence,knopf2020phase,liu2019energetic,colli2015cahn,
mininni2017higher,racke2003cahn}. By taking the variational derivative of the total energy, Liu and Wu proposed
Cahn-Hilliard model with dynamic boundary conditions called Liu-Wu Model \cite{liu2019energetic}:
\begin{numcases}{}
\phi_t=\Delta \mu,                   & $(\mathbf{x},t)\in \Omega \times (0,T)$,\nonumber\\
\mu=-\varepsilon \Delta \phi + \frac{1}{\varepsilon}F^\prime(\phi),     &$(\mathbf{x},t)\in \Omega \times (0,T)$,\nonumber\\
\partial_{\bold n} \mu=0,                   & $(\mathbf{x},t)\in \Gamma \times (0,T)$,\nonumber\\
\phi|_\Gamma=\psi,     &$(\mathbf{x},t)\in \Gamma \times (0,T)$,\label{LWmodel}\\
\psi_t=\Delta_\Gamma \mu_\Gamma,                   & $(\mathbf{x},t)\in \Gamma \times (0,T)$,\nonumber\\
\mu_\Gamma=-\delta\kappa \Delta_\Gamma \psi + \frac{1}{\delta}G^\prime(\psi)+\varepsilon\partial_{\bold n} \phi,
&$(\mathbf{x},t)\in \Gamma \times (0,T)$.\nonumber
\end{numcases}
Here, $\mu$, $\mu_\Gamma$ denote the chemical potentials in the bulk and on the boundary, respectively. The model
assumes that there is no mass exchange between the bulk and the boundary, namely, $\partial_{\bold n}\mu=0$.
The classical choice of $F,\,G$ is the smooth double-well potential
\begin{equation}
\label{FGdwell}
	F(x)=\frac{1}{4}(x^2-1)^2,\quad G(x)=\frac{1}{4}(x^2-1)^2, \quad x\in \mathbb{R}.
\end{equation}
Moreover, the dynamic boundary conditions ensure the conservation of the total mass
\[
\int_{\Omega}\phi(t)dx+\int_{\Gamma}\psi(t)dS=\int_{\Omega}\phi(0)dx+\int_{\Gamma}\psi(0)dS,\quad\quad\forall t \in [0,T],
\]
especially we have
\begin{equation}
\label{MassCSV}
\int_{\Omega}\phi(t)dx=\int_{\Omega}\phi(0)dx,\quad \int_{\Gamma}\psi(t)dS=\int_{\Gamma}
\psi(0)dS,\quad\quad\forall t \in [0,T],
\end{equation}
indicating that the Liu-Wu model satisfies the mass conservation law in the bulk and on the boundary, respectively.
Moreover, it is easy to find that the system \eqref{LWmodel} satisfies energy dissipation law:
\[
\frac{d}{dt}E^{total}(\phi,\psi)=-\| \nabla\mu\|^2_{\Omega}-\|\nabla_{\Gamma}\mu_\Gamma\|^2_{\Gamma}\leq 0.
\]

The numerical algorithms of Cahn-Hilliard equation have been well studied. There are many effective methods
for time discretization, such as the convex splitting method
\cite{grun2013convergent,shen2012second}, the invariant energy quadratization (IEQ) method
\cite{yang2016linear,yang2018linear,zhao2018general}, the scalar auxiliary variable (SAV) method \cite{shen2018scalar}
and stabilized linearly implicit approach.
Chen and Shen \cite{chen1998applications} and Zhu et al. \cite{zhu1999coarsening} applied the Fourier-spectral method to the stabilized semi-implicit scheme for the Cahn-Hilliard equation. Xu and Tang \cite{xu2006stability} introduced different
stability terms and established a large time-stepping stable semi-implicit method for the two-dimensional epitaxial
growth model. He et al. \cite{he2007large}  proposed a similar method for the Cahn-Hilliard equation, in which
the stability term $B(\phi^{n+1}-\phi^n)$ (or $B(\phi^{n+1}-2\phi^n+\phi^{n-1})$)
is added to the nonlinear volume force of the first-order (second-order) scheme. Shen and Yang \cite{shen2010numerical} applied similar stability terms to Allen-Cahn equation and Cahn-Hilliard equation to design the unconditionally energy stable first-order linear schemes and second-order linear schemes under reasonable stability conditions. This idea has been adopted in \cite{feng2013stabilized}  for the stabilized Crank-Nicolson schemes
for phase field models. Wu et al. \cite{wu2014stabilized} proposed another stabilized  second-order Crank-Nicolson scheme
for tumor-growth system, which involved a new concave-convex energy splitting. These time marching schemes will lead to a
linear system, which is easier to solve than the nonlinear system generated by the traditional convex splitting scheme,
which implicitly deals with the nonlinear convex force. On the other hand, when explicitly dealing with nonlinear
forces, it is necessary to introduce appropriate stability terms and appropriate truncated nonlinear
function $\tilde{f}(\phi)$ instead of $f(\phi)$ to prove the unconditional energy stability property with
reasonable stability constant. It is worth mentioning that Li et al. \cite{li2017second,li2016characterizing}
proved that the energy stability characteristics can also be obtained, however, a larger stability constant is
required  with no truncation made to $f(\phi)$. Stabilization techniques are also used to construct higher order
schemes, such as exponential time difference (ETD) scheme \cite{ju2015fast} and Runge-Kutta scheme \cite{guo2016efficient,shin2017unconditionally}.

Recently, numerical approximations of the Cahn-Hilliard equation with dynamic boundary conditions have been raised (see
\cite{bao2020numerical,cherfils2010numerical,cherfils2014numerical,fukao2017structure,gong2020arbitrarily,trautwein2018finite}
). Specifically, a finite element approach for the Liu-Wu model has been proposed  in
\cite{trautwein2018finite,garcke2020weak}, where the model is simulated by the direct discretization based on
piecewise linear finite element function, and the corresponding nonlinear system is solved by the Newton's method.
However, in the above finite element scheme, the backward implicit Euler method is used for time discretization,
in which the nonlinear system needs to be solved in each time step. Recently, a linear and energy stable numerical scheme for Liu-Wu model has been proposed in \cite{bao2020numerical}, which is an extension of the stable linear implicit method for the classic boundary conditions.

In this paper, the stability and convergence of second-order semi-implicit time marching scheme are studied.
We use the second order backward differentiation formula (BDF2) to discrete time derivative. For the nonlinear force with
second-order stability, the explicit extrapolation method is used and stabilizers are added to ensure energy
dissipation, where the stabilizers are inspired by the work \cite{wang2018efficient} by Wang and Yu. The main
features of our scheme include the following: (1) To the best of our knowledge, this is the first linear,
second-order stabilized semi-implicit scheme for this model; (2) At the discrete level, the constant
coefficient linear system is obtained. We only need to solve the linear equation at each step, which reduces
the computation cost greatly; (3) Discrete energy dissipation is proved. The finite difference method is used for
spatial discretization and satisfies the discretized energy dissipation law; (4) We also give the error
analysis in $l^\infty(0,T;H^{-1})\cap l^2(0,T;H^{1})$ norm in detail.

The rest of this article is organized as follows. We first introduce some defintions and notations in
Section~\ref{sec2}. In Section~\ref{sec3} we present the BDF2-type scheme. A modified energy stability is
established and we prove that the scheme has the property of decreasing energy. Subsequently, the convergence
estimate is provided in Section~\ref{sec4}. In Section~\ref{sec5}, we present some numerical experiments, including
the cases with different initial conditions, cases with different potential functions and the accuracy test. Finally, the concluding remarks are given in Section~\ref{sec6}.
\section{Preliminaries}
\label{sec2}
Before giving the stabilized scheme and corresponding error analysis, we make some definitions in this
section which will be used in the paper.

We consider a finite time interval $[0,T]$ and a domain $\Omega\subseteq \mathbb{R}^d$, which is a
bounded domain with sufficient smooth boundary $\Gamma=\partial \Omega$ and ${\bold n} = {\bold n}(x)$ is the unit outward
normal vector on $\Gamma$.

We use  $\|\cdot\|_{m,p,\Omega}$
to denote the standard norm of the Sobolev space $W^{m,p}(\Omega)$ and  $\|\cdot\|_{m,p,\Gamma}$
to denote the standard norm of the Sobolev space $W^{m,p}(\Gamma)$.
In particular, we use $\|\cdot\|_{L^p(\Omega)}$, $\|\cdot\|_{L^p(\Gamma)}$ to
denote the norm of $W^{0,p}(\Omega) = L^p(\Omega)$ and $W^{0,p}(\Gamma) = L^p(\Gamma)$;
$\|\cdot\|_{m,\Omega}$, $\|\cdot\|_{m,\Gamma}$ to denote the norm
of $W^{m,2}(\Omega)$ = $H^2(\Omega)$ and $W^{m,2}(\Gamma)$ = $H^2(\Gamma)$;
we  also use $\|\cdot\|_{\Omega}$ and $\|\cdot\|_{\Gamma}$ to
denote the norm of $W^{0,2}(\Omega) = L^2(\Omega)$ and $W^{0,2}(\Gamma) = L^2(\Gamma)$.
Let $(\cdot,\cdot)_\Omega$,  $(\cdot,\cdot)_\Gamma$ represent the inner product of $L^2(\Omega)$ and $L^2(\Gamma)$,
respectively. In addition, define for $p \geq 0$
\[
	H^{-p}(\Omega)=(H^{p}(\Omega))^*,\quad	H_0^{-p}(\Omega)=\{u\in H^{-p}(\Omega)|\left\langle u,1\right\rangle_p=0\},
\]
where $\left\langle \cdot,\cdot\right\rangle_p$ stands for the dual product between $H^p(\Omega)$ and
$H^{-p}(\Omega)$. We denote $L_0^2(\Omega):=H_0^0(\Omega)$. For $u \in L_0^2(\Omega)$, let $-\Delta^{-1}u:=u_1
\in H^1(\Omega) \cap L_0^2(\Omega)$, where $u_1$ is the solution to
\[
-\Delta u_1=u\,\,\mathrm{in}\, \, \Omega,\, \frac{\partial u_1}{\partial {\bold n}}=0\,\,\mathrm{on} \,\,\partial\Omega,
\]
and $\| u\|_{-1,\Omega}:=\sqrt{(u,-\Delta^{-1}u)_\Omega}$. Similarly, $H^{-p}(\Gamma)$,
$L_0^2(\Gamma):=H_0^0(\Gamma)$, $\| u\|_{-1,\Gamma}:=\sqrt{(u,-\Delta_\Gamma^{-1}u)_\Gamma}$ are also defined.
Let $\tau$ be the time step size. For a sequence of functions $f^0$, $f^1,\cdots,f^N$ in some Hilbert
space $E$, we denote the sequence by $\{f_\tau \}$ and define the following discrete norm for $\{f_\tau \}$:
\[
\Vert f_\tau\Vert_{l^\infty(E)}=\max_{0\leq n\leq N}\Big (\| f^n\|_E\Big ),
\]

For simplicity, we denote
\[
\delta_t\phi^{n+1}:=\phi^{n+1}-\phi^{n}, \quad\delta_{tt}\phi^{n+1}:=\phi^{n+1}-2\phi^{n}+\phi^{n-1},\quad
\hat{\phi}^{n+1}=2\phi^{n}-\phi^{n-1},
\]
\[\delta_t\psi^{n+1}:=\psi^{n+1}-\psi^{n},\quad \delta_{tt}\psi^{n+1}:=\psi^{n+1}-2\psi^{n}+\psi^{n-1},\quad
\hat{\psi}^{n+1}=2\psi^{n}-\psi^{n-1}.
\]

\section{Second order scheme of the model}
\label{sec3}

We propose a stabilized linear BDF2 scheme for the Liu-Wu model as follows
\begin{numcases}{}
\frac{\frac{3}{2}\phi^{n+1}-2\phi^{n}+\frac{1}{2}\phi^{n-1}}{\tau}=\Delta \mu^{n+1},
 & \hspace{-0.4cm}$\mathbf{x}\in\Omega,$ \label{scheme-1}\\
\mu^{n+1}=-\varepsilon \Delta\phi^{n+1}
+\frac{1}{\varepsilon}f\left(2\phi^{n}-\phi^{n-1}\right)\nonumber\\
\,\,\,\quad\quad\quad-A_1\tau\Delta\left(\phi^{n+1}-\phi^{n}\right)
+B_1\left(\phi^{n+1}-2\phi^n+\phi^{n-1}\right),     &\hspace{-0.4cm}$\mathbf{x}\in \Omega,$ \label{scheme-2} \\
	\partial_{\bold n} \mu^{n+1}=0,                   & \hspace{-0.4cm}$\mathbf{x}\in \Gamma,$  \label{scheme-3}\\
	\phi^{n+1}|_\Gamma=\psi^{n+1},     &\hspace{-0.4cm}$\mathbf{x}\in \Gamma,$  \label{scheme-4}\\
	\frac{\frac{3}{2}\psi^{n+1}-2\psi^{n}+\frac{1}{2}\psi^{n-1}}{\tau}=\Delta_{\Gamma} \mu_\Gamma^{n+1},
& \hspace{-0.4cm}$\mathbf{x}\in \Gamma,$  \label{scheme-5}\\
	\mu_\Gamma ^{n+1}=-\delta\kappa \Delta_\Gamma\psi^{n+1}
	+\frac{1}{\delta}g\left(2\psi^{n}-\psi^{n-1}\right)+\varepsilon\partial_{\bold n} \phi^{n+1} \nonumber\\
	\,\,\,\quad\quad\quad-A_2\tau\Delta_\Gamma\left(\psi^{n+1}-\psi^{n}\right)+B_2\left(\psi^{n+1}-2\psi^n
+\psi^{n-1}\right)+A_1\tau\partial_{\bold n} \left(\phi^{n+1}-\phi^{n}\right) ,&\hspace{-0.4cm}$\mathbf{x}\in \Gamma,$ \label{scheme-6}
\end{numcases}
where  $f=F^\prime, \,g=G^\prime$ are the nonlinear chemical potential. In particular, we notice that a second order
approximation to $f$ and $g$ at time step $t^{n+1}$ are taken as $f(2\phi^n-\phi^{n-1})$
and $g(2\psi^n-\psi^{n-1})$. $T$ is the fixed time, $N$ is the number of time steps and $\tau = T/N$ is
the step size. $A_1,\,B_1,\,A_2$ and $B_2$ are four non-negative constants to be determined, and the stabilization terms
$A_1\tau\Delta\left(\phi^{n+1}-\phi^{n}\right)$, $B_1\left(\phi^{n+1}-2\phi^n+\phi^{n-1}\right)$,
$A_2\tau\Delta_\Gamma\left(\psi^{n+1}-\psi^{n}\right)$ and $B_2\left(\psi^{n+1}-2\psi^n+\psi^{n-1}\right)$ are
added to the bulk equation and boundary equation to enhance stability, respectively.
Before proving the stability, we first give some Assumptions.
\begin{assumption}
\label{assumption1}
	Assume that the Lipschitz properties hold for the second derivative of $F$ with respect to $\phi$ and
the second derivative of $G$ with  respect to $\psi$ (namely, $f^\prime$ and $g^\prime$). $f^\prime$ and
$g^\prime$ are bounded. Precisely, there exists positive constants $K_1$, $K_2$, $L_1$ and $L_2$ such that
\[
|f^\prime(\phi_1)-f^\prime(\phi_2)| \leq K_1	| \phi_1-\phi_2|,\quad
| g^\prime(\psi_1)-g^\prime(\psi_2)|\leq K_2	| \psi_1-\psi_2| ,
\quad \forall\phi_1,\phi_2,\psi_1,\psi_2 \in \mathbb{R},
\]
and
\[
\max_{\phi \in \mathbb{R}}| f^\prime(\phi)| \leq L_1 ,\quad
\max_{\psi \in \mathbb{R}}| g^\prime(\psi)| \leq L_2.
\]
\end{assumption}
\begin{assumption}
\label{assumption2}
Assume that the mass conservative property is available for the two
initial values of interior and boundary respectively:
\begin{eqnarray*}
 \frac{1}{|\Omega|} \int_\Omega{\phi^1dx}&=&\frac{1}{|\Omega|} \int_\Omega{\phi^0dx}=m_0,\\
	\frac{1}{|\Gamma|} \int_\Gamma{\psi^1dx}&=&\frac{1}{|\Gamma|} \int_\Gamma{\psi^0dx}=m_1.
	\end{eqnarray*}
\end{assumption}
We have the energy stability as follows.
\begin{thm}
\label{thm1}
	Assume that Assumption \ref{assumption1} and Assumption \ref{assumption2}
	hold. 	Then under the conditions
	\begin{eqnarray}
	A_1 \geq \frac{1}{\alpha_2} \frac{{L_1}^2}{16\varepsilon^2}-\alpha_1\frac{\varepsilon}{2\tau},\quad
B_1 \geq \frac{L_1}{\varepsilon},\label{StaCon1}\\
	A_2 \geq \frac{1}{\alpha_2} \frac{{L_2}^2}{16\delta^2}-\alpha_1\frac{\delta\kappa}{2\tau}, \quad B_2
\geq \frac{L_2}{\delta},\label{StaCon2}\\
	0 \leq \alpha_1 \leq 1,\quad 0 < \alpha_2 \leq 1,\label{StaCon3}
   \end{eqnarray}
	we have
	\begin{eqnarray}
	&&\tilde{E}(\phi^{n+1},\psi^{n+1})\nonumber\\
	&&\leq \tilde{E}(\phi^{n},\psi^{n})-\frac{1}{4\tau}\|
\delta_{tt}\phi^{n+1}\|^2_{-1,\Omega}-\frac{1}{4\tau}\| \delta_{tt}\psi^{n+1}\|_{-1,\Gamma}\nonumber\\
	&&-(1-\alpha_1)(\frac{\varepsilon}{2}\| \nabla\delta_{t}\phi^{n+1}\|_\Omega^2+\frac{\delta\kappa}{2}\|\nabla_{\Gamma}
\delta_{t}\psi^{n+1}\|_{\Gamma}^2)-(1-\alpha_2)\frac{1}{\tau}(\|  \delta_{t}\phi^{n+1}\|_{-1,\Omega}^2
	+\|\delta_{t}\psi^{n+1}\|_{-1,\Gamma}^2)\nonumber\\
	&&-\left(2\sqrt{\alpha_2\left(A_1+\frac{\alpha_1\varepsilon}{2\tau}\right)}
-\frac{L_1}{2\varepsilon}\right)\| \delta_{t}\phi^{n+1}\|_\Omega^2
	-\left(2\sqrt{\alpha_2\left(A_2+\frac{\alpha_1\delta\kappa }{2\tau}\right)}-\frac{L_2}{2\delta}\right)
\| \delta_{t}\psi^{n+1}\|_{\Gamma}^2\nonumber\\
&&-\left(\frac{B_1}{2}-\frac{L_1}{2\varepsilon}\right)\|\delta_{tt}\phi^{n+1}\|_\Omega^2-
\left(\frac{B_2}{2}-\frac{L_2}{2\delta}\right)\|\delta_{tt}\psi^{n+1}\|_{\Gamma}^2,
\label{energy}
	\end{eqnarray}
	where
	\begin{eqnarray}
\tilde{E}(\phi^{n+1},\psi^{n+1})&=&E^{total}(\phi^{n+1},\psi^{n+1})+\frac{1}{4\tau}
\left(\|\delta_{t}\phi^{n+1}\|^2_{-1,\Omega}+\|\delta_{t} \psi^{n+1}\|^2_{-1,\Gamma}\right)\nonumber\\
&&+\left(\frac{L_1}{2\varepsilon}+\frac{B_1}{2}\right)\|\delta_{t}\phi^{n+1}\|_\Omega^2
+\left(\frac{L_2}{2\delta}+\frac{B_2}{2}\right)\|\delta_{t}\psi^{n+1}\|_{\Gamma}^2.
\label{EngMod}
	\end{eqnarray}
\end{thm}
\begin{proof}
	Integrating both sides of equation \eqref{scheme-1}, we have
	\begin{eqnarray*}
	\frac{1}{|\Omega|}\int_{\Omega}{\phi^{n+1}dx}=m_0,\,\,n=1,...N.\nonumber
	\end{eqnarray*}
	Thus $\delta _t \phi^{n+1} \in L^2_0(\Omega)$ for $n=0,...N$.
	Pairing \eqref{scheme-1} with $(-\Delta)^{-1}\delta _t\phi^{n+1}$ and adding to \eqref{scheme-2} paired
with $-\delta _t\phi^{n+1}$, we have
	\begin{eqnarray}
	&&\left (\frac{3\phi^{n+1}-4\phi^n+\phi^{n-1}}{2\tau},(-\Delta)^{-1}\delta _t\phi^{n+1} \right
)_{\Omega}\nonumber\\
&=&\varepsilon(\Delta\phi^{n+1},\delta _t \phi^{n+1})_{\Omega}-\frac{1}{\varepsilon}
(f(\hat{{\phi}}^{n+1}), \delta _t\phi^{n+1})_{\Omega}
	+A_1\tau\left(\partial_{\bold n}(\phi^{n+1}-\phi^{n}),\delta _t \phi^{n+1}\right)_{\Gamma}\nonumber\\
	&&-A_1\tau\| \nabla\delta _t\phi^{n+1}\|_{\Omega}^2	-B_1(\delta _{tt}\phi^{n+1},\delta _t\phi^{n+1})_{\Omega}.\nonumber
	\end{eqnarray}
With the fact that
\[
2(h^{n+1}-h^{n},h^{n+1})=\| h^{n+1}\|^2-\| h^{n}\|^2+\| h^{n+1}-h^{n}\|^2,
\]
and
\begin{eqnarray*}
&&\left (\frac{3h^{n+1}-4h^n+h^{n-1}}{2\tau},h^{n+1}\right )\\
&=&\frac{1}{4\tau}(\| h^{n+1}\|^2+\| 2h^{n+1}-h^{n}\|^2-\| h^{n}\|^2-\| 2h^{n}-h^{n-1}\|^2+\| \delta_{tt}h^{n+1}\|^2),
\end{eqnarray*}
we have
\begin{eqnarray}
&-&\left(\frac{3\phi^{n+1}-4\phi^n+\phi^{n-1}}{2\tau},(-\Delta)^{-1}\delta_t\phi^{n+1}\right)_{\Omega}\nonumber\\
&=&-\frac{1}{\tau}\|\delta_t\phi^{n+1}\|^2_{-1,\Omega}	-\frac{1}{4\tau}(\|\delta_t\phi^{n+1}\|^2_{-1,\Omega}
-\|\delta_t\phi^{n}\|^2_{-1,\Omega}+\|\delta_{tt}\phi^{n+1}\|^2_{-1,\Omega}),
\label{energyproof1}
	\end{eqnarray}
	\begin{eqnarray}
	\varepsilon(\Delta\phi^{n+1},\delta_t\phi^{n+1})_{\Omega}=\varepsilon(\partial_{\bold n}\phi^{n+1},
\delta_t\phi^{n+1})_{\Gamma}-\frac{\varepsilon}{2}(\|\nabla\phi^{n+1}\|_{\Omega}^2
-\|\nabla\phi^{n}\|_{\Omega}^2+\|\nabla\delta_{t}\phi^{n+1}\|_{\Omega}^2),
\label{energyproof2}
	\end{eqnarray}
and
	\begin{eqnarray}
-B_1(\delta_{tt}\phi^{n+1},\delta_{t}\phi^{n+1})_{\Omega}=-\frac{B_1}{2}\|\delta_{t}
\phi^{n+1}\|_{\Omega}^2+\frac{B_1}{2}\|\delta_{t}\phi^{n}\|_{\Omega}^2-\frac{B_1}{2}\|\delta_{tt}\phi^{n+1}\|_{\Omega}^2.
\label{energyproof3}
	\end{eqnarray}
Expanding  $F(\phi^{n+1})$ and $F(\phi^{n})$ at $\hat{\phi}^{n+1}=2\phi^{n}-\phi^{n-1}$ yields
\[
	F(\phi^{n+1})=  F(\hat{\phi}^{n+1})+f(\hat{\phi}^{n+1})(\phi^{n+1}-\hat{\phi}^{n+1})+\frac12f^\prime(\xi_1^{n})
(\phi^{n+1}-\hat{\phi}^{n+1})^2,
\]
and
\[	F(\phi^{n})=F(\hat{\phi}^{n+1})+f(\hat{\phi}^{n+1})(\phi^{n}-\hat{\phi}^{n+1})+\frac12f^\prime(\xi_2^{n})
(\phi^{n}-\hat{\phi}^{n+1})^2,
\]
where $\xi_1^{n}$ is between $\phi^{n+1}$ and $\hat{\phi}^{n+1}$, $\xi_2^{n}$ is between $\phi^{n}$ and
$\hat{\phi}^{n+1}$. Substracting the above two equations and using the facts that $\phi^{n+1}-\hat{\phi}^{n+1}=\delta_{tt}\phi^{n+1}$ and $\phi^{n}-\hat{\phi}^{n+1}=-\delta_{t}\phi^{n}$, we obtain
	\begin{eqnarray}
	F(\phi^{n+1})-F(\phi^{n})-f(\hat{\phi}^{n+1})\delta_{t}\phi^{n+1}&=&\frac12f^\prime(\xi_1^{n})
(\delta_{tt}\phi^{n+1})^2-\frac12f^\prime(\xi_2^{n})(\delta_{t}\phi^{n})^2\nonumber\\
	&\leq& \frac{L_1}{2}|\delta_{tt}\phi^{n+1}|^2+\frac{L_1}{2}|\delta_{t}\phi^{n}|^2.
\label{energyproof4}
	\end{eqnarray}
Combining the result with \eqref{energyproof1}-\eqref{energyproof4}, we obtain
	\begin{eqnarray*}
	\uppercase\expandafter{\romannumeral1}&:=&-\frac{1}{\varepsilon}\left(f(\hat{\phi}^{n+1}),
\delta_{t}\phi^{n+1}\right)_{\Omega}\\
	&=&\frac{1}{\tau}\| \delta_{t}\phi^{n+1}\|^2_{-1,\Omega}+\frac{1}{4\tau}\left(\|
\delta_{t}\phi^{n+1}\|^2_{-1,\Omega}-\|\delta_{t}\phi^{n}\|^2_{-1,\Omega}+\|\delta_{tt}\phi^{n+1}\|^2_{-1,\Omega}\right)\\
	&-&\varepsilon(\partial_{\bold{n}}\phi^{n+1},\delta_{t}\phi^{n+1})_{\Gamma}+\frac{\varepsilon}{2}\left(\|
\nabla\phi^{n+1}\|_\Omega^2-\|\nabla\phi^{n}\|_\Omega^2+\|\nabla\delta_{t}\phi^{n+1}\|_\Omega^2\right)\\
&-&A_1\tau(\partial_{\bold {n}}(\phi^{n+1}-\phi^{n}),\delta_{t}\phi^{n+1})_{\Gamma}+A_1\tau\|\nabla\delta_{t}\phi^{n+1}\|_\Omega^2 \\
	&+&\frac{B_1}{2}\| \delta_{t}\phi^{n+1}\|_\Omega^2-\frac{B_1}{2}\|\delta_{t}\phi^{n}\|_\Omega^2
+\frac{B_1}{2}\|\delta_{tt}\phi^{n+1}\|_\Omega^2\\
&\leq&-\frac{1}{\varepsilon}\left(F(\phi^{n+1})-F(\phi^{n}),1\right)_{\Omega}
+\frac{L_1}{2\varepsilon}\|\delta_{tt}\phi^{n+1}\|_\Omega^2+\frac{L_1}{2\varepsilon}\|\delta_{t}\phi^{n}\|_\Omega^2.
	\end{eqnarray*}
 Rewriting $\uppercase\expandafter{\romannumeral1}$ gives
 \begin{eqnarray}
&&\frac1{\varepsilon}\left(F(\phi^{n+1})-F(\phi^{n}),1\right)_{\Omega}+\frac{\varepsilon}{2}\left(\|
\nabla\phi^{n+1}\|_\Omega^2-\|\nabla\phi^{n}\|_\Omega^2\right)+\frac{1}{4\tau}\left(\|
 \delta_{t}\phi^{n+1}\|^2_{-1,\Omega}-\| \delta_{t}\phi^{n}\|^2_{-1,\Omega}\right)\nonumber\\
&-&\varepsilon\left(\partial_{\bold n}\phi^n,\delta_t\phi^{n+1}\right)_{\Gamma}-A_1\tau\left(\partial_{\bold n}(\phi^{n+1}
-\phi^{n}),\delta_{t}\phi^{n}\right)_{\Gamma}\nonumber\\
&+&\frac{L_1}{2\varepsilon}\left(\| \delta_{t}\phi^{n+1}\|_\Omega^2-\|
\delta_{t}\phi^{n}\|_\Omega^2\right)+\frac{B_1}{2}\left(\| \delta_{t}\phi^{n+1}
\|_\Omega^2-\| \delta_{t}\phi^{n}\|_\Omega^2\right)\nonumber\\
&\leq&-\frac{1}{4\tau}\| \delta_{tt}\phi^{n+1}\|^2_{-1,\Omega}-\frac{1}{\tau}\|
\delta_{t}\phi^{n+1}\|_{-1,\Omega}^2-\frac{\varepsilon}{2}\|\nabla\delta_{t}\phi^{n+1}
\|_\Omega^2-A_1\tau\|\nabla\delta_{t}\phi^{n+1}\|_\Omega^2\nonumber\\
&+&\frac{L_1}{2\varepsilon}\| \delta_{t}\phi^{n+1}\|_\Omega^2-\frac{B_1}{2}\|
\delta_{tt}\phi^{n+1}\|_\Omega^2+\frac{L_1}{2\varepsilon}\| \delta_{tt}\phi^{n+1}\|_\Omega^2.
\end{eqnarray}
	Similarly, integrating both sides of equation \eqref{scheme-5}, we get
\[
	\frac{1}{|\Gamma|}\int_{\Gamma}{\psi^{n+1}dx}=m_1,\,\,n=1,...N.\nonumber
\]
	Thus $\delta _t \psi^{n+1} \in L^2_0(\Gamma)$ for $n=0,\,1,\,\cdots N$.
	Pairing \eqref{scheme-5} with $(-\Delta_{\Gamma})^{-1}\delta _t\psi^{n+1}$ and adding to \eqref{scheme-6}
paired with $\delta _t\psi^{n+1}$, we have
	\begin{eqnarray}
&&	\left (\frac{3\psi^{n+1}-4\psi^n+\psi^{n-1}}{2\tau},(-\Delta_{\Gamma})^{-1}\delta _t\psi^{n+1}\right
)_{\Gamma}\nonumber\\
&=&\delta\kappa(\Delta_{\Gamma}\psi^{n+1},\delta _t \psi^{n+1})_{\Gamma}-\frac{1}{\delta}
(g(\hat{\psi}^{n+1}), \delta _t\psi^{n+1})_{\Gamma}
-\varepsilon(\partial_{\bold n}\phi^{n+1},\delta_t \psi^{n+1})_{\Gamma}-A_1\tau(\partial_{\bold n}(\phi^{n+1}-\phi^{n}),\delta _t \psi^{n+1})_{\Gamma}\nonumber\\
&&-A_2\tau\| \nabla_{\Gamma} \delta _t\psi^{n+1}\|^2 _{\Gamma}-B_2(\delta_{tt}\psi^{n+1},\delta_t\psi^{n+1})_{\Gamma}.\nonumber
\end{eqnarray}
	The integral of each part reads as follows:
	\begin{eqnarray}
&-&\left (\frac{3\psi^{n+1}-4\psi^n+\psi^{n-1}}{2\tau},(-\Delta_{\Gamma})^{-1}\delta _t\psi^{n+1}\right)_{\Gamma}\nonumber\\
&=&-\frac{1}{\tau}\|\delta_t\psi^{n+1}\|^2_{-1,\Gamma}-\frac{1}{4\tau}(\|\delta_t\psi^{n+1}\|^2_{-1,\Gamma}
-\|\delta_t\psi^{n}\|^2_{-1,\Gamma}+\|\delta_{tt}\psi^{n+1}\|^2_{-1,\Gamma}),
\label{energyproof5}
	\end{eqnarray}

	\begin{eqnarray}
\delta\kappa(\Delta_{\Gamma}\psi^{n+1},\delta_t\psi^{n+1})_{\Gamma}
&=&-\delta\kappa\left(\nabla_{\Gamma}\psi^{n+1},\nabla_\Gamma\delta_t\psi^{n+1}\right)_{\Gamma}\nonumber\\
&=&-\frac{\delta\kappa}{2}(\|\nabla_{\Gamma}\psi^{n+1}\|_{\Gamma}^2
-\|\nabla_{\Gamma}\psi^{n}\|_{\Gamma}^2+\|\nabla_{\Gamma}\delta_{t}\psi^{n+1}\|_{\Gamma}^2),
\label{energyproof6}
	\end{eqnarray}
	\begin{eqnarray}
	-\varepsilon(\partial_{\bold n}\phi^{n+1},\delta t \psi^{n+1})_{\Gamma}=-\varepsilon(\partial_{\bold n}\phi^{n+1},\delta _t
\phi^{n+1})_{\Gamma},\label{energyproof7}
	\end{eqnarray}
	\begin{eqnarray}
	-B_2(\delta_{tt}\psi^{n+1},\delta_{t}\psi^{n+1})_{\Gamma}=-\frac{B_2}{2}\|\delta_{t}\psi^{n+1}
\|_{\Gamma}^2+\frac{B_2}{2}\|\delta_{t}\psi^{n}\|_{\Gamma}^2-\frac{B_2}{2}\|\delta_{tt}\psi^{n+1}\|_{\Gamma}^2.
\label{energyproof8}
	\end{eqnarray}
	Expanding $G(\psi^{n+1})$ and $G(\psi^{n})$ at $\hat{\psi}^{n+1}=2\psi^{n}-\psi^{n-1}$ leads to
\[
G(\psi^{n+1})= G(\hat{\psi}^{n+1})+g(\hat{\psi}^{n+1})(\psi^{n+1}-\hat{\psi}^{n+1})
+\frac12g^\prime(\zeta_1^{n})(\psi^{n+1}-\hat{\psi}^{n+1})^2,
\]
and
\[
G(\psi^{n})=
G(\hat{\psi}^{n+1})+g(\hat{\psi}^{n+1})(\psi^{n}-\hat{\psi}^{n+1})+\frac12g^\prime(\zeta_2^{n})(\psi^{n}-\hat{\psi}^{n+1})^2,
\]
where $\zeta_1^{n}$ is between $\psi^{n+1}$ and $\hat{\psi}^{n+1}$, $\zeta_2^{n}$ is
between $\psi^{n}$ and $\hat{\psi}^{n+1}$. Substracting the above two equations and using the fact that
$\psi^{n+1}-\hat{\psi}^{n+1}=\delta_{tt}\psi^{n+1}$ and $\psi^{n}-\hat{\psi}^{n+1}=-\delta_{t}\psi^{n}$, we get
\begin{eqnarray}
G(\psi^{n+1})-G(\psi^{n})-g(\hat{\psi}^{n+1})\delta_{t}\psi^{n+1}&=&\frac12g'(\zeta_1^{n})
(\delta_{tt}\psi^{n+1})^2-\frac12g'(\zeta_2^{n})(\delta_{t}\psi^{n})^2 \nonumber\\
	&\leq& \frac{L_2}{2}| \delta_{tt}\psi^{n+1}|^2+\frac{L_2}{2}| \delta_{t}\psi^{n}|^2,
\label{energyproof9}
\end{eqnarray}
and
\begin{eqnarray}
-A_1\tau\left(\partial_{\bold n}(\phi^{n+1}-\phi^{n}),\delta_t\psi^{n+1}\right)_{\Gamma}
=-A_1\tau\left(\partial_{\bold n}(\phi^{n+1}-\phi^{n}),\delta_t\phi^{n+1}\right)_{\Gamma}.
\label{energyproof10}
\end{eqnarray}
Combining the result with \eqref{energyproof5}-\eqref{energyproof10}, we obtain
	\begin{eqnarray*}
	\uppercase\expandafter{\romannumeral2}&:=&-\frac{1}{\delta}\left(g(\hat{\psi}^{n+1}),\delta_{t}
\psi^{n+1}\right)_{\Gamma}\\
&=&\frac{1}{\tau}\| \delta_{t}\psi^{n+1}\|^2_{-1,\Gamma}+\frac{1}{4\tau}\left(\|
\delta_{t}\psi^{n+1}\|^2_{-1,\Gamma}-\|\delta_{t}\psi^{n}\|^2_{-1,\Gamma}+\|\delta_{tt}\psi^{n+1}\|^2_{-1,\Gamma}\right) \\
&+&\varepsilon(\partial_{\bold n}\phi^{n+1},\delta_{t}\phi^{n+1})_{\Gamma}+\frac{\delta\kappa}{2}\left(\|
\nabla_{\Gamma}\psi^{n+1}\|_{\Gamma}^2-\|\nabla_{\Gamma}\psi^{n}\|_{\Gamma}^2
+\|\nabla_{\Gamma}\delta_{t}\psi^{n+1}\|_{\Gamma}^2\right) \\
&+&A_1\tau\left(\partial_{\bold n}(\phi^{n+1}-\phi^{n}),\delta_t\phi^{n+1}\right)_{\Gamma}
+A_2\tau\|\nabla_{\Gamma}\delta_{t}\psi^{n+1}\|_{\Gamma}^2 \nonumber\\
	&+&\frac{B_2}{2}\|\delta_{t}\psi^{n+1}\|_{\Gamma}^2-\frac{B_2}{2}\|\delta_{t}\psi^{n}
\|_{\Gamma}^2+\frac{B_2}{2}\|\delta_{tt}\psi^{n+1}\|_{\Gamma}^2	 \\
&\leq&-\frac{1}{\delta}\left(G(\psi^{n+1})-G(\psi^{n}),1\right)_{\Gamma}+\frac{L_2}{2\delta}
\|\delta_{tt}\psi^{n+1}\|_{\Gamma}^2+\frac{L_2}{2\delta}\|\delta_{t}\psi^{n}\|_{\Gamma}^2.
	\end{eqnarray*}
Rewriting $\uppercase\expandafter{\romannumeral2}$ yields
\begin{eqnarray}	
&&\frac1{\delta}\left(G(\psi^{n+1})-G(\psi^{n}),1\right)_{\Gamma}+\frac{\delta\kappa}{2}\left(\|
\nabla_{\Gamma}\psi^{n+1}\|_{\Gamma}^2-\|\nabla_{\Gamma}\psi^{n}\|_{\Gamma}^2\right)
+\frac{1}{4\tau}\left(\| \delta_{t}\psi^{n+1}\|^2_{-1,\Gamma}-\|
\delta_{t}\psi^{n}\|^2_{-1,\Gamma}\right)\nonumber\\
&+&\varepsilon(\partial_{\bold n}\phi^{n+1},\delta_t\phi^{n+1})_{\Gamma}+A_1\tau\left(\partial_{\bold n}(\phi^{n+1}-\phi^{n}),
\delta_{t}\phi^{n}\right)_{\Gamma}\nonumber\\
&+&\frac{L_2}{2\delta}\left(\| \delta_{t}\psi^{n+1}\|_{\Gamma}^2-\|
\delta_{t}\psi^{n}\|_{\Gamma}^2\right)+\frac{B_2}{2}\left(\| \delta_{t}\psi^{n+1}
\|_{\Gamma}^2-\| \delta_{t}\psi^{n}\|_{\Gamma}^2\right)\nonumber\\
&\leq&-\frac{1}{4\tau}\| \delta_{tt}\psi^{n+1}\|^2_{-1,\Gamma}-\frac{1}{\tau}\|
\delta_{t}\psi^{n+1}\|_{-1,\Gamma}^2-\frac{\delta\kappa}{2}\|\nabla_{\Gamma}\delta_{t}\psi^{n+1}
\|_{\Gamma}^2-A_2\tau\|\nabla_{\Gamma}\delta_{t}\psi^{n+1}\|_{\Gamma}^2\nonumber\\
&+&\frac{L_2}{2\delta}\| \delta_{t}\psi^{n+1}\|_{\Gamma}^2-\frac{B_2}{2}\| \delta_{tt}\psi^{n+1}\|_{\Gamma}^2+\frac{L_2}{2\delta}\| \delta_{tt}\psi^{n+1}\|_{\Gamma}^2.
\end{eqnarray}
Noticing the facts that
\[
\chi_1 \|\nabla\delta_t\phi^{n+1}\|_{\Omega}^2+\frac{\alpha_2}{\tau}\|\delta_t\phi^{n+1}
\|^2_{-1,\Omega}\geq2\sqrt{\frac{\chi_1\alpha_2}{\tau}}\|\delta_t\phi^{n+1}\|_{\Omega}^2,	
\]
and
\[
\chi_2 \|\nabla_\Gamma\delta_t\psi^{n+1}\|_{\Gamma}^2+\frac{\tilde{\alpha}_2}{\tau}\|\delta_t\psi^{n+1}
\|^2_{-1,\Gamma}\geq2\sqrt{\frac{\chi_2\tilde{\alpha}_2}{\tau}}\|\delta_t\psi^{n+1}\|_{\Gamma}^2,	
\]
with $\displaystyle\chi_1=A_1\tau+\frac{\alpha_1\varepsilon}{2}$,\,$\displaystyle\chi_2=A_2\tau+\frac{\tilde{\alpha}_1\delta\kappa}{2}$,\, $0\leq\alpha_1,\tilde{\alpha}_1\leq 1$,\, $0<\alpha_2,\tilde{\alpha}_2\leq 1$, we have
\[
-\left(\frac{\alpha_1\epsilon}{2}+A_1\tau\right)\|\nabla\delta_t\phi^{n+1}\|_{\Omega}^2
-\frac{\alpha_2}{\tau}\|\delta_t\phi^{n+1}\|^2_{-1,\Omega}
\leq -2\sqrt{\left(\frac{\alpha_1\varepsilon}{2\tau}+A_1\right)\alpha_2}\|\delta_t\psi^{n+1}\|_{\Gamma}^2,	
\]
and
\[
-\left(\frac{\tilde{\alpha}_1\delta\kappa}{2}+A_2\tau\right)\|\nabla_{\Gamma}\delta_t\psi^{n+1}\|_{\Gamma}^2
-\frac{\tilde{\alpha}_2}{\tau}\|\delta_t\psi^{n+1}\|^2_{-1,\Gamma}
\leq -2\sqrt{\left(\frac{\tilde{\alpha}_1\delta\kappa}{2\tau}+A_2\right)\tilde{\alpha}_2}\|\delta_t\phi^{n+1}\|_{\Gamma}^2.
\]	
For simplicity, let $\alpha_1=\tilde{\alpha}_1$,\, $\alpha_2=\tilde{\alpha}_2$ in the following.

Combining the above equations and the inequalities we get
\begin{eqnarray}
&&\frac1{\varepsilon}\left(F(\phi^{n+1})-F(\phi^{n}),1\right)_{\Omega}+\frac{\varepsilon}{2}\left(\|
 \nabla\phi^{n+1}\|_{\Omega}^2-\|\nabla\phi^{n}\|_{\Omega}^2\right)+\frac{1}{4\tau}\left(\|
 \delta_{t}\phi^{n+1}\|^2_{-1,\Omega}-\| \delta_{t}\phi^{n}\|^2_{-1,\Omega}\right)\nonumber\\
&+&\frac{L_1}{2\varepsilon}\left(\| \delta_{t}\phi^{n+1}\|_{\Omega}^2-\|
\delta_{t}\phi^{n}\|_{\Omega}^2\right)+\frac{B_1}{2}\left(\| \delta_{t}\phi^{n+1}
\|_{\Omega}^2-\| \delta_{t}\phi^{n}\|_{\Omega}^2\right)\nonumber\\
&+&\frac1{\delta}\left(G(\psi^{n+1})-G(\psi^{n}),1\right)_{\Gamma}+\frac{\delta\kappa}{2}\left(\|
\nabla_{\Gamma}\psi^{n+1}\|_{\Gamma}^2-\|\nabla_{\Gamma}\psi^{n}\|_{\Gamma}^2\right)
+\frac{1}{4\tau}\left(\| \delta_{t}\psi^{n+1}\|^2_{-1,\Gamma}-\|
\delta_{t}\psi^{n}\|^2_{-1,\Gamma}\right)\nonumber\\
&+&\frac{L_2}{2\delta}\left(\| \delta_{t}\psi^{n+1}\|_{\Gamma}^2-\|
\delta_{t}\psi^{n}\|_{\Gamma}^2\right)+\frac{B_2}{2}\left(\| \delta_{t}\psi^{n+1}
\|_{\Gamma}^2-\| \delta_{t}\psi^{n}\|_{\Gamma}^2\right)\nonumber\\
&\leq&-\frac{1}{4\tau}\| \delta_{tt}\phi^{n+1}\|^2_{-1,\Omega}-\frac{1}{4\tau}\|
\delta_{tt}\psi^{n+1}\|^2_{-1,\Gamma}\nonumber\\
&-&(1-\alpha_1)(\frac{\varepsilon}{2}\| \nabla\delta_{t}\phi^{n+1}\|_\Omega^2+\frac{\delta\kappa}{2}\|\nabla_{\Gamma}
\delta_{t}\psi^{n+1}\|_{\Gamma}^2)-(1-\alpha_2)\frac{1}{\tau}(\|  \delta_{t}\phi^{n+1}\|_{-1,\Omega}^2
	+\|\delta_{t}\psi^{n+1}\|_{-1,\Gamma}^2)\nonumber\\
&+&\left(-2\sqrt{(\frac{\alpha_1\varepsilon}{2\tau}+A_1)\alpha_2}+\frac{L_1}{2\varepsilon}
\right)\|\delta_t\phi^{n+1}\|_\Omega^2+\left(-\frac{B_1}{2}+\frac{L_1}{2\varepsilon}\right)
\|\delta_{tt}\phi^{n+1}\|_\Omega^2\nonumber\\
&+&\left(-2\sqrt{(\frac{\alpha_1\delta\kappa}{2\tau}+A_1)\alpha_2}+\frac{L_2}{2\delta}\right)
\|\delta_t\psi^{n+1}\|_{\Gamma}^2+\left(-\frac{B_2}{2}+\frac{L_2}{2\delta}\right)
\|\delta_{tt}\psi^{n+1}\|_{\Gamma}^2.\nonumber
\end{eqnarray}
Then under the conditions \eqref{StaCon1}-\eqref{StaCon3}, for the modified energy \eqref{EngMod}, the estimate \eqref{energy} holds.
\end{proof}

\begin{rem}
We can see that the BDF2 scheme \eqref{scheme-1}-\eqref{scheme-6} is conditionally stable if we take $\alpha_1=\alpha_2=1$ and the 
constraint on time step is 
\[
\tau \leq \min\left\{\frac{8\epsilon^3}{L_1^2},\, \frac{8\delta^3\kappa}{L_2^2}\right\}.
\]
If we set the artificial parameters as \eqref{StaCon1}-\eqref{StaCon3}, then the scheme is unconditionally stable, which implies the 
stabilizers $A_1$ and $A_2$ play an important role in order to obtain an unconditionally energy stable scheme.
\end{rem}

\section{Convergence analysis}\label{sec4}
We will establish the error estimate of the semi-discretized BDF2 scheme for the Cahn-Hilliard model with
dynamic boundary conditions in the norm of $l^\infty(0,T;H^{-1})\cap l^2(0,T;H^{1})$.
Let $\phi(t^n)$, $\psi(t^n)$ be the exact solution at time $t=t^n$ to equation \eqref{LWmodel} and $\phi^n$, $\psi^n$ be the solution
at time $t=t^n$ to the numerical scheme \eqref{scheme-1}-\eqref{scheme-6}.
Define the error functions  $e_{\phi}^n=\phi^n-\phi(t^n)$, $e_{\psi}^n=\psi^n-\psi(t^n)$, $e_{\mu}^n=\mu^n-\mu(t^n)$,
$e_{\Gamma}^n=\mu_\Gamma^n-\mu_\Gamma(t^n)$.
Because the integrals of $\phi^n$ and $\psi^n$ are conserved, $\delta_t\phi^n$ belongs to $L_0^2(\Omega)$
and $\delta_t\psi^n$ belongs to $L_0^2(\Gamma)$. This fact makes the $H^{-1}$ norms of $e^n_\phi$ and $e^n_\psi$ are well-defined in $H^{-1}(\Omega)$ and  $H^{-1}(\Gamma)$, respectively.
Before presenting the detailed error analysis, we need to give an Assumption.
\begin{assumption}
\label{assumption3}
	Assume that there exist two constants $C_0$ and $C_1$ independent of $\tau$, such that
\[
\|e_{\phi}^{1}\|^2_{-1}+\varepsilon\tau\|\nabla e_{\phi}^{1}\|^2\leq C_0\tau^4,
\]
and
\[
\|e_{\psi}^{1}\|^2_{-1}+\delta\kappa\tau\|\nabla_{\Gamma} e_{\psi}^{1}\|^2\leq C_1\tau^4.
\]
\end{assumption}
We have the error estimate as follows.

\begin{thm}
\label{theorem2}
	Suppose that the exact solutions $(\phi,\psi,\mu,\mu_{\Gamma})$ are sufficiently smooth and
Assumption \ref{assumption1}, \ref{assumption2} and \ref{assumption3} hold. Then $\forall \tau \leq1$, we have the following error estimate for
the BDF2 scheme \eqref{scheme-1}-\eqref{scheme-6}:
	\begin{eqnarray*}	
		&&\max_{1\leq n\leq N}\{\| e_{\phi}^{n+1}\|^2_{-1,\Omega}+\|
2e_{\phi}^{n+1}-e_{\phi}^{n}\|^2_{-1,\Omega}+2A_1\tau^2\| \nabla e_{\phi}^{n+1}
\|_\Omega^2\nonumber\\
		&&+\| e_{\psi}^{n+1}\|^2_{-1,\Gamma}+\|
2e_{\psi}^{n+1}-e_{\psi}^{n}\|^2_{-1,\Gamma}+2A_2\tau^2\| \nabla_{\Gamma}
e_{\psi}^{n+1}\|_{\Gamma}^2\}\nonumber\\
		&+&\sum^N_{n=1}(2A_1\tau^2\| \delta_t \nabla e_{\phi}^{n+1}\|_\Omega^2
+\tau\varepsilon\|\nabla e_{\phi}^{n+1}\|_\Omega^2+\| \delta_{tt}
e_{\phi}^{n+1}\|^2_{-1,\Omega}+4B_1\tau\| e_{\phi}^{n+1}\|_\Omega^2\nonumber\\
		&+&2A_2\tau^2\| \delta_t \nabla_{\Gamma} e_{\psi}^{n+1}\|_{\Gamma}^2
+\tau\delta\kappa\|\nabla_{\Gamma} e_{\psi}^{n+1}\|_{\Gamma}^2+\| \delta_{tt}
e_{\psi}^{n+1}\|^2_{-1,\Gamma}+4B_2\tau\| e_{\psi}^{n+1}\|_{\Gamma}^2)\nonumber\\
		&\leq& \exp\left((C_8\varepsilon^{-3}+C_9\delta^{-3})T\right)\left(C_{10}\varepsilon^{-1}+C_{11}\delta^{-1}
		+C_0(5+2A_1\varepsilon^{-1}\tau)+C_1(5+2A_2\delta^{-1}\kappa^{-1}\tau)\right)\tau^4.\label{convergence}
	\end{eqnarray*}
	where $C_8$, $C_9$, $C_{10}$, $C_{11}$ are four constants that can be uniformly bounded independent
of $\varepsilon$, $\delta$, $\kappa$ and $\tau$.
\end{thm}
\begin{proof}
A careful consistency analysis implies that
\begin{numcases}{}
\frac{\frac{3}{2}\phi(t^{n+1})-2\phi(t^{n})+\frac{1}{2}\phi(t^{n-1})}{\tau}=\Delta \mu(t^{n+1})+R_{\phi}^{n+1}, &
$\mathbf{x}\in\Omega,$  \label{consistent-1}\\
\mu^{n+1}=-\varepsilon \Delta \phi(t^{n+1})
+\frac{1}{\varepsilon}f\left(\phi(t^{n+1})\right)-A_1\tau\Delta\left(\phi(t^{n+1})-\phi(t^{n})\right)\nonumber\\
\quad\quad\quad+B_1\left(\phi(t^{n+1})-2\phi(t^{n})+\phi(t^{n-1})\right)+R_\mu^{n+1},     &$\mathbf{x}\in
 \Omega,$\label{consistent-2}\\
\partial_{\bold n} \mu(t^{n+1})=0,                   & $\mathbf{x}\in \Gamma,$ \label{consistent-3}\\
\phi(t^{n+1})|_\Gamma=\psi(t^{n+1}),     &$\mathbf{x}\in \Gamma,$  \label{consistent-4}\\
\frac{\frac{3}{2}\psi(t^{n+1})-2\psi(t^{n})+\frac{1}{2}\psi(t^{n-1})}{\tau}=\Delta_{\Gamma}
\mu_\Gamma(t^{n+1})+R_\psi^{n+1},    & $\mathbf{x}\in \Gamma,$  \label{consistent-5}\\
\mu_\Gamma(t^{n+1})=-\delta\kappa \Delta_\Gamma \psi(t^{n+1})
+\frac{1}{\delta}g\left(\psi(t^{n+1})\right)+\varepsilon\partial_{\bold n}
\phi(t^{n+1})-A_2\tau\Delta_\Gamma\left(\psi(t^{n+1})-\psi(t^{n})\right) \nonumber\\
\,\,\,\quad\quad\quad+B_2\left(\psi(t^{n+1})-2\psi(t^{n})+\psi(t^{n-1})\right)
+A_1\tau\partial_{\bold n}\left(\phi(t^{n+1})-\phi(t^{n})\right)+R_\Gamma^{n+1}, &$\mathbf{x}\in \Gamma,$ \label{consistent-6}
\end{numcases}
where the residual terms are
\begin{eqnarray*}
&&R_\phi^{n+1}=\frac{\frac{3}{2}\phi(t^{n+1})-2\phi(t^{n})+\phi(t^{n-1})}{\tau}-\phi_t(t^{n+1}),\nonumber\\
&&R_\psi^{n+1}=\frac{\frac{3}{2}\psi(t^{n+1})-2\psi(t^{n})+\psi(t^{n-1})}{\tau}-\psi_t(t^{n+1}),\\
&&R_\mu^{n+1}=A_1\tau\Delta\left(\phi(t^{n+1})-\phi(t^{n})\right)-B_1\left(\phi(t^{n+1})-2\phi(t^{n})+\phi(t^{n-1})\right),\\
&&R_\Gamma^{n+1}=A_2\tau\Delta\left(\psi(t^{n+1})-\psi(t^{n})\right)-B_2\left(\psi(t^{n+1})-2\psi(t^{n})+\psi(t^{n-1})\right)
-A_1\tau\partial_{\bold n}\left(\phi(t^{n+1})-\phi(t^{n})\right).
\end{eqnarray*}
For simplicity, we define
\begin{eqnarray*}	
&&R_1^{n+1}=\phi(t^{n+1})-2\phi(t^{n})+\phi(t^{n-1}),\\
&&R_2^{n+1}=\tau\left(\phi(t^{n+1})-\phi(t^{n})\right),\\
&&R_3^{n+1}=\psi(t^{n+1})-2\psi(t^{n})+\psi(t^{n-1}),\\
&&R_4^{n+1}=\tau\left(\psi(t^{n+1})-\psi(t^{n})\right).
\end{eqnarray*}	
By substracting \eqref{consistent-1}-\eqref{consistent-6} from the corresponding scheme \eqref{scheme-1}-\eqref{scheme-6},
we derive the error equations as follows,	
	\begin{numcases}{}
	\frac{\frac{3}{2}e_{\phi}^{n+1}-2e_{\phi}^{n}+\frac{1}{2}e_{\phi}^{n-1}}{\tau}=\Delta e_\mu^{n+1}-R_{\phi}^{n+1},        & $\mathbf{x}\in\Omega,$ \label{error-1}\\
	e_\mu^{n+1}=-\varepsilon \Delta e_\phi^{n+1}
+\frac{1}{\varepsilon}\left(f(2\phi^{n}-\phi^{n-1})-f(\phi(t^{n+1}))\right)-A_1\tau\Delta\delta_ te_\phi^{n+1}\nonumber\\
\quad\quad\quad+B_1\delta_ {tt}e_\phi^{n+1}-A_1\Delta R_2^{n+1}+B_1R_1^{n+1},&$\mathbf{x}\in \Omega,$\label{error-2}\\
	\partial_{\bold n} e_\mu^{n+1}=0,                   & $\mathbf{x}\in \Gamma,$  \label{error-3}\\
	e_\phi^{n+1}|_\Gamma=e_\psi^{n+1},     &$\mathbf{x}\in \Gamma,$ \label{error-4}\\
	\frac{\frac{3}{2}e_\psi^{n+1}-2e_\psi^{n}+\frac{1}{2}e_\psi^{n-1}}{\tau}=\Delta_{\Gamma}
 e_\Gamma^{n+1}-R_\psi^{n+1},    & $\mathbf{x}\in \Gamma,$  \label{error-5}\\
	e_\Gamma ^{n+1}=-\delta\kappa \Delta_\Gamma e_\psi^{n+1}
	+\frac{1}{\delta}\left(g(2\psi^{n}-\psi^{n-1})-g(\psi(t^{n+1}))\right)+\varepsilon\partial_{\bold n}
e_\phi^{n+1}-A_2\tau\Delta_\Gamma\delta_t e_\psi^{n+1}  \nonumber\\
	\,\,\,\quad\quad\quad+B_2\delta_{tt} e_\psi^{n+1}+A_1\tau\partial_{\bold n}\delta_t e_\phi^{n+1}
-A_2\Delta_\Gamma R_4^{n+1}+B_2 R_3^{n+1}+A_1\partial_{\bold n} R_2^{n+1} ,&$\mathbf{x}\in \Gamma$. \label{error-6}
\end{numcases}
Pairing \eqref{error-1} with $(-\Delta)^{-1}e_\phi^{n+1}$ and adding to \eqref{error-2} paired with $-e_\phi^{n+1}$, we have
\begin{eqnarray}
&&\left(\frac{\frac{3}{2}e_{\phi}^{n+1}-2e_{\phi}^{n}+\frac{1}{2}e_{\phi}^{n-1}}{\tau},
(-\Delta)^{-1}e_\phi^{n+1}\right)_{\Omega}-\varepsilon(\Delta e_\phi^{n+1},e_\phi^{n+1})_{\Omega}
-A_1\tau(\Delta \delta_t e_\phi^{n+1},e_\phi^{n+1})_{\Omega}\nonumber\\
&=&(R_\phi^{n+1},(-\Delta)^{-1}\delta _te_\phi^{n+1})_{\Omega}-B_1(R_1^{n+1},e_\phi^{n+1})_{\Omega}+A_1(\Delta
R_2^{n+1},e_\phi^{n+1})_{\Omega}\nonumber\\
&-&B_1(\delta_{tt}e_\phi^{n+1},e_\phi^{n+1})_{\Omega}-\frac{1}{\varepsilon}
\left(f(2\phi^n-\phi^{n-1})-f(\phi(t^{n+1})),e_\phi^{n+1}\right)_{\Omega}\nonumber\\
&:=&J_1+J_2+J_3+J_4+J_5.
\label{convergenceproof1}
\end{eqnarray}
The left hand side of \eqref{convergenceproof1} can be estimated term by term as below:
\begin{eqnarray}
&&\left(\frac{\frac{3}{2}e_{\phi}^{n+1}-2e_{\phi}^{n}+\frac{1}{2}e_{\phi}^{n-1}}{\tau},
(-\Delta)^{-1}e_\phi^{n+1}\right)_{\Omega}\nonumber\\
&=&\frac{1}{4\tau}(\| e_\phi^{n+1}\|^2_{-1,\Omega}+\|
2e_\phi^{n+1}-e_\phi^{n}\|^2_{-1,\Omega})\nonumber\\
&-&\frac{1}{4\tau}(\| e_\phi^{n}\|^2_{-1,\Omega}+\|
2e_\phi^{n}-e_\phi^{n-1}\|^2_{-1,\Omega})+\frac{1}{4\tau}\| \delta_{tt}
e_\phi^{n+1}\|^2_{-1,\Omega},\label{convergenceproof2}
\end{eqnarray}
\begin{eqnarray}
-\varepsilon(\Delta e_\phi^{n+1},e_\phi^{n+1})_{\Omega}=-\varepsilon(\partial_{\bold n}
e_\phi^{n+1},e_\phi^{n+1})_{\Gamma}+\varepsilon\|\nabla e_\phi^{n+1}
\|_\Omega^2,\label{convergenceproof3}
\end{eqnarray}
\begin{eqnarray}
&&-A_1\tau(\Delta\delta_te_\phi^{n+1},e_\phi^{n+1})_{\Omega}=-A_1\tau(\partial_{\bold n}\delta_te_\phi^{n+1},
e_\phi^{n+1})_{\Gamma}+A_1\tau(\delta_t\nabla e_\phi^{n+1},\nabla e_\phi^{n+1})_{\Omega}\nonumber\\
&=&-A_1\tau(\partial_{\bold n}\delta_te_\phi^{n+1},e_\phi^{n+1})_{\Gamma}+\frac12A_1\tau(\|\nabla
e_\phi^{n+1}\|_\Omega^2-\|\nabla e_\phi^{n}\|_\Omega^2+\|\delta_t\nabla
e_\phi^{n+1}\|_\Omega^2).\label{convergenceproof4}
\end{eqnarray}
Next, we estimate the terms on the right hand side of \eqref{convergenceproof1}
\begin{eqnarray}
J_1=-(R_\phi^{n+1},(-\Delta)^{-1} e_\phi^{n+1})_{\Omega}\leq
\frac{1}{\eta_1}\|\Delta^{-1}R_\phi^{n+1}\|^2_{-1,\Omega}+\frac{\eta_1}{4}\| \nabla
e_\phi^{n+1}\|_\Omega^2,\label{convergenceproof5}
\end{eqnarray}
\begin{eqnarray}
J_2=-B_1(R_1^{n+1},e_\phi^{n+1})_{\Omega}\leq \frac{B_1^2}{\eta_1}\|
R_1^{n+1}\|^2_{-1,\Omega}+\frac{\eta_1}{4}\| \nabla e_\phi^{n+1}
\|_\Omega^2,\label{convergenceproof6}
\end{eqnarray}
\begin{eqnarray}
J_3&=&A_1(\Delta R_2^{n+1},e_\phi^{n+1})_{\Omega}= A_1(\partial_{\bold n}R_2^{n+1},e_\phi^{n+1})_{\Gamma}
-A_1(\nabla R_2^{n+1},\nabla e_\phi^{n+1})_{\Omega}\nonumber\\
&\leq& \frac{A_1^2}{\eta_1}\| \nabla R_2^{n+1}\|_\Omega^2+\frac{\eta_1}{4}\| \nabla
e_\phi^{n+1}\|_\Omega^2+A_1(\partial_{\bold n}R_2^{n+1},e_\phi^{n+1})_{\Gamma},\label{convergenceproof7}
\end{eqnarray}

\begin{eqnarray}
J_4&=&-B_1(\delta_{tt}e_\phi^{n+1},e_\phi^{n+1})_{\Omega}=
-B_1\left(e_\phi^{n+1}-(2e_\phi^{n}-e_\phi^{n-1}),e_\phi^{n+1}\right)_{\Omega}\nonumber\\
&\leq& -B_1 \| e_\phi^{n+1}\|_\Omega^2+\frac{B_1^2}{\eta_1}\|
 2e_\phi^{n}-e_\phi^{n-1}\|^2_{-1,\Omega}+\frac{\eta_1}{4}\| \nabla
 e_\phi^{n+1}\|_\Omega^2,\label{convergenceproof8}
\end{eqnarray}
\begin{eqnarray}
	J_5&=&-\frac{1}{\varepsilon}\left(f(2\phi^n-\phi^{n-1})-f(\phi(t^{n+1})),e_\phi^{n+1}\right)_{\Omega}\nonumber\\
	&\leq&\frac{K_1}{\varepsilon}(| 2\phi^n-\phi^{n-1}-\phi(t^{n+1})|,| e_\phi^{n+1}|)_{\Omega}\nonumber\\
	&=&\frac{K_1}{\varepsilon}(| 2e_\phi^n-e_\phi^{n-1}
-\delta_{tt}\phi(t^{n+1})|,| e_\phi^{n+1}|)_{\Omega}\nonumber\\
	&\leq& \frac{K_1^2}{\varepsilon^2\eta_1}\|
2e_\phi^{n}-e_\phi^{n-1}\|^2_{-1,\Omega}+\frac{K_1^2}{\varepsilon^2\eta_1}\|
R_1^{n+1}\|^2_{-1,\Omega}+\frac{\eta_1}{2}\| \nabla e_\phi^{n+1}\|_\Omega^2,\label{convergenceproof9}
\end{eqnarray}
where $\eta_1$ is a positive constant.

Similarly, pairing \eqref{error-5} with $(-\Delta_{\Gamma})^{-1}e_\psi^{n+1}$ and adding to \eqref{error-6} paired
with $-e_\psi^{n+1}$, we have
\begin{eqnarray}
&&\left(\frac{\frac{3}{2}e_{\psi}^{n+1}-2e_{\psi}^{n}+\frac{1}{2}e_{\psi}^{n-1}}
{\tau},(-\Delta_{\Gamma})^{-1}e_\psi^{n+1}\right)_{\Gamma}-\delta \kappa(\Delta_{\Gamma}
e_\psi^{n+1},e_\psi^{n+1})_{\Gamma}-A_2\tau(\Delta_{\Gamma} \delta_t e_\psi^{n+1},e_\psi^{n+1})_{\Gamma}\nonumber\\
&=&(R_\psi^{n+1},(-\Delta_{\Gamma})^{-1}e_\psi^{n+1})_{\Gamma}-B_2(R_3^{n+1},e_\psi^{n+1})_{\Gamma}
+A_2(\Delta_{\Gamma} R_4^{n+1},e_\psi^{n+1})_{\Gamma}\nonumber\nonumber\\
&-&B_2(\delta_{tt}e_\psi^{n+1},e_\psi^{n+1})_{\Gamma}-\frac{1}{\delta}\left(g(2\psi^n-\psi^{n-1})
-g(\psi(t^{n+1})),e_\psi^{n+1}\right)_{\Gamma}\nonumber\\
&-&A_1(\partial_{\bold n}R_2^{n+1},e_\psi^{n+1})_{\Gamma}-\varepsilon(\partial_{\bold n}e_\phi^{n+1},e_\psi^{n+1})_{\Gamma}
-A_1\tau(\partial_{\bold n}\delta_te_\phi^{n+1},e_\psi^{n+1})_{\Gamma}\nonumber\\
&:&=J_6+J_7+J_8+J_9+J_{10}+J_{11}+J_{12}+J_{13}.\label{convergenceproof10}
\end{eqnarray}
The left hand side of \eqref{convergenceproof10} can be estimated as follows,
\begin{eqnarray}
&&\left(\frac{\frac{3}{2}e_{\psi}^{n+1}-2e_{\psi}^{n}+\frac{1}{2}e_{\psi}^{n-1}}{\tau},
(-\Delta_{\Gamma})^{-1}e_\psi^{n+1}\right)_{\Gamma}\nonumber\nonumber\\
&=&\frac{1}{4\tau}(\| e_\psi^{n+1}\|^2_{-1,{\Gamma}}+\|
2e_\psi^{n+1}-e_\psi^{n}\|^2_{-1,{\Gamma}})\nonumber\\
&-&\frac{1}{4\tau}(\| e_\psi^{n}\|^2_{-1,{\Gamma}}+\|
2e_\psi^{n}-e_\psi^{n-1}\|^2_{-1,{\Gamma}})+\frac{1}{4\tau}\| \delta_{tt}
e_\psi^{n+1}\|^2_{-1,{\Gamma}},\label{convergenceproof11}
\end{eqnarray}
\begin{eqnarray}
-\delta\kappa(\Delta_{\Gamma} e_{\Gamma}^{n+1},e_\psi^{n+1})_{\Gamma}=\delta\kappa\|\nabla_{\Gamma}
e_\psi^{n+1}\|_{\Gamma}^2,\label{convergenceproof12}
\end{eqnarray}
\begin{eqnarray}
-A_2\tau(\Delta_{\Gamma}\delta_te_\psi^{n+1},e_\psi^{n+1})_{\Gamma}
&=&A_2\tau(\delta_t\nabla_{\Gamma} e_\psi^{n+1},\nabla_{\Gamma} e_\psi^{n+1})_{\Gamma}\nonumber\\
&=&\frac12A_2\tau(\|\nabla_{\Gamma} e_\psi^{n+1}\|_{\Gamma}^2-\|\nabla_{\Gamma}
e_\psi^{n}\|_{\Gamma}^2+\|\delta_t\nabla_{\Gamma} e_\psi^{n+1}\|_{\Gamma}^2).
\label{convergenceproof13}
\end{eqnarray}
Also, we estimate the terms on the right hand side of \eqref{convergenceproof10},
\begin{eqnarray}
J_6=-(R_\psi^{n+1},(-\Delta_{\Gamma})^{-1}e_\phi^{n+1})_{\Gamma}\leq
\frac{1}{\eta_2}\|\Delta_{\Gamma}^{-1}R_\psi^{n+1}\|^2_{-1,{\Gamma}}
+\frac{\eta_2}{4}\| \nabla_{\Gamma} e_\psi^{n+1}\|_{\Gamma}^2,\label{convergenceproof14}
\end{eqnarray}
\begin{eqnarray}
J_7=-B_2(R_3^{n+1},e_\psi^{n+1})_{\Gamma}\leq \frac{B_2^2}{\eta_2}\|
 R_3^{n+1}\|^2_{-1,{\Gamma}}+\frac{\eta_2}{4}\| \nabla_{\Gamma}
 e_\psi^{n+1}\|_{\Gamma}^2,\label{convergenceproof15}
\end{eqnarray}
\begin{eqnarray}
J_8=A_2(\Delta_{\Gamma} R_4^{n+1},e_\psi^{n+1})_{\Gamma}\leq \frac{A_2^2}{\eta_2}
\| \nabla_{\Gamma} R_4^{n+1}\|_{\Gamma}^2+\frac{\eta_2}{4}\| \nabla_{\Gamma}
e_\psi^{n+1}\|_{\Gamma}^2,\label{convergenceproof16}
\end{eqnarray}

\begin{eqnarray}
J_9&=&-B_2(\delta_{tt}e_\psi^{n+1},e_\psi^{n+1})_{\Gamma}=
-B_2\left(e_\psi^{n+1}-(2e_\psi^{n}-e_\psi^{n-1}),e_\psi^{n+1}\right)_{\Gamma}\nonumber\\
&\leq& -B_2 \| e_\psi^{n+1}\|_{\Gamma}^2+\frac{B_2^2}{\eta_2}\|
2e_\psi^{n}-e_\psi^{n-1}\|^2_{-1,{\Gamma}}+\frac{\eta_2}{4}\| \nabla_{\Gamma}
e_\psi^{n+1}\|_{\Gamma}^2,\label{convergenceproof17}
\end{eqnarray}
\begin{eqnarray}
	J_{10}&=&-\frac{1}{\delta}\left(g(2\psi^n-\psi^{n-1})-g(\psi(t^{n+1})),e_\psi^{n+1}\right)_{\Gamma}\nonumber\\
	&\leq&\frac{L_2}{\delta}(| 2\psi^n-\psi^{n-1}-\psi(t^{n+1})|,| e_\psi^{n+1}|)_{\Gamma}\nonumber\\
	&=&\frac{K_2}{\delta}(| 2e_\psi^n-e_\psi^{n-1}-\delta_{tt}\psi(t^{n+1})|,
| e_\psi^{n+1}|)_{\Gamma}\nonumber\\
	&\leq& \frac{K_2^2}{\delta^2\eta_2}\|
2e_\psi^{n}-e_\psi^{n-1}\|^2_{-1,{\Gamma}}+\frac{K_2^2}{\delta^2\eta_2}\|
R_3^{n+1}\|^2_{-1,{\Gamma}}+\frac{\eta_2}{2}\| \nabla_{\Gamma}
e_\psi^{n+1}\|_{\Gamma}^2,\label{convergenceproof18}
\end{eqnarray}
where $\eta_2$ is a positive constant.
\begin{eqnarray}
J_{11}=-A_1(\partial_{\bold n}R_2^{n+1},e_\psi^{n+1})_{\Gamma},\label{convergenceproof19}
\end{eqnarray}
\begin{eqnarray}
J_{12}=-\varepsilon(\partial_{\bold n}e_\phi^{n+1},e_\psi^{n+1})_{\Gamma},\label{convergenceproof20}
\end{eqnarray}
\begin{eqnarray}
J_{13}=-A_1\tau(\partial_{\bold n}\delta_te_\phi^{n+1},e_\psi^{n+1})_{\Gamma}.\label{convergenceproof21}
\end{eqnarray}
Combining \eqref{convergenceproof1}-\eqref{convergenceproof21} leads to
\begin{eqnarray}
	&&\frac{1}{4\tau}(\| e_\phi^{n+1}\|^2_{-1,\Omega}+\|
2e_\phi^{n+1}-e_\phi^{n}\|^2_{-1,\Omega})+\frac12A_1\tau\|\nabla e_\phi^{n+1}
\|_\Omega^2\nonumber\\
	&+&\frac12A_1\tau
	\|\delta_t\nabla e_\phi^{n+1}\|_\Omega^2+\varepsilon\|\nabla
e_\phi^{n+1}\|_\Omega^2+\frac{1}{4\tau}\| \delta_{tt} e_\phi^{n+1}\|^2_{-1,\Omega}\nonumber\\
	&+&B_1\| e_\phi^{n+1}\|_\Omega^2+\frac{1}{4\tau}(\| e_\psi^{n+1}
\|^2_{-1,\Omega}+\| 2e_\psi^{n+1}-e_\psi^{n}\|^2_{-1,\Omega})\nonumber\\
	&+&\frac12A_2\tau\|\nabla_{\Gamma} e_\psi^{n+1}\|_{\Gamma}^2+\frac12A_2\tau
	\|\delta_t\nabla_{\Gamma} e_\psi^{n+1}\|_{\Gamma}^2+\delta\kappa \|\nabla_{\Gamma}
e_\psi^{n+1}\|_{\Gamma}^2
	+\frac{1}{4\tau}\| \delta_{tt} e_\psi^{n+1}\|^2_{-1,\Gamma}+B_2\|
e_\psi^{n+1}\|_{\Gamma}^2\nonumber\\
	&\leq&
	\frac{1}{4\tau}(\| e_\phi^{n}\|^2_{-1,\Omega}+\|
2e_\phi^{n}-e_\phi^{n-1}\|^2_{-1,\Omega})+\frac12A_1\tau\|\nabla
e_\phi^{n}\|_\Omega^2+\frac{1}{\eta_1}\|\Delta^{-1}R_\phi^{n+1}\|_{-1,\Omega}^2\nonumber\\
	&+&\frac32\eta_1\|\nabla e_\phi^{n+1}\|_\Omega^2+\frac{1}{\eta_1}\left(B_1^2+\frac{K_1^2}
{\varepsilon^2}\right)\| R_1^{n+1}\|_{-1,\Omega}^2+\frac{A_1^2}{\eta_1}
\|\nabla R_2^{n+1}\|_\Omega^2
	\nonumber\\
	&+&\frac{1}{\eta_1}\left(B_1^2+\frac{K_1^2}{\varepsilon^2}\right)\| 2e_\phi^{n}-e_\phi^{n-1}\|^2_{-1,\Omega}+
	\frac{1}{4\tau}(\| e_\psi^{n}\|^2_{-1,\Gamma}+\| 2e_\psi^{n}-e_\psi^{n-1}
\|^2_{-1,\Gamma})\nonumber\\
	&+&\frac12A_2\tau\|\nabla_{\Gamma}
e_\psi^{n}\|_{\Gamma}^2+\frac{1}{\eta_2}\|\Delta_{\Gamma}^{-1}R_\psi^{n+1}
\|_{-1,\Gamma}^2+\frac32\eta_2\|\nabla_{\Gamma} e_\psi^{n+1}\|_{\Gamma}^2\nonumber\\
	&+&\frac{1}{\eta_2}\left(B_2^2+\frac{K_2^2}{\delta^2}\right)\|
 R_3^{n+1}\|_{-1,\Gamma}^2+\frac{A_2^2}{\eta_2}\|\nabla_{\Gamma} R_4^{n+1}\|_{\Gamma}^2
	+\frac{1}{\eta_2}\left(B_2^2+\frac{K_2^2}{\delta^2}\right)\|
2e_\psi^{n}-e_\psi^{n-1}\|^2_{-1,\Gamma}.\label{convergenceproof22}
\end{eqnarray}	
Using Taylor expansions in integral form, we can get estimate for the residuals:
\[
\|\Delta^{-1} R_\phi^{n+1}\|^2_{-1,\Omega}\leq c_1\tau^3\int_{t_{n-1}}^{t_{n+1}}
\|\partial_{ttt}\Delta^{-1}\phi(t)\|^2_{-1,\Omega}dt
\leq C_2\tau^3,
\]
\[
\|\Delta^{-1} R_\psi^{n+1}\|^2_{-1,\Gamma}\leq c_2\tau^3\int_{t_{n-1}}^{t_{n+1}}
\|\partial_{ttt}\Delta_{\Gamma}^{-1}\psi(t)\|^2_{-1,\Gamma}dt\leq C_3\tau^3,
\]
\[
\| R_1^{n+1} \|^2_{-1,\Omega}\leq c_3\tau^3\int_{t_{n-1}}^{t_{n+1}}
\|\partial_{tt}\phi(t)\|^2_{-1,\Omega}dt\leq C_4\tau^3,
\]
\[
\| R_3^{n+1} \|^2_{-1,\Gamma}\leq c_4\tau^3\int_{t_{n-1}}^{t_{n+1}}
\|\partial_{tt}\psi(t)\|^2_{-1,\Gamma}dt
\leq C_5\tau^3,
\]
\[
\| \nabla R_2^{n+1} \|^2_{\Omega}\leq c_5\tau^3\int_{t_{n}}^{t_{n+1}}
\|\partial_{t}\nabla\phi(t)\|_\Omega^2dt\leq C_6\tau^3,
\]
\[
\| \nabla R_4^{n+1} \|^2_{\Gamma}\leq c_6\tau^3\int_{t_{n}}^{t_{n+1}}
\|\partial_{t}\nabla_{\Gamma}\psi(t)\|_{\Gamma}^2dt\leq C_7\tau^3.
\]
Taking $\displaystyle\eta_1=\frac\varepsilon2$, $\displaystyle\eta_2=\frac{\delta}{2}$ in \eqref{convergenceproof22}, we get
\begin{eqnarray}
	&&(\| e_\phi^{n+1}\|^2_{-1,\Omega}+\|
2e_\phi^{n+1}-e_\phi^{n}\|^2_{-1,\Omega})+2A_1\tau^2\|\nabla e_\phi^{n+1}
\|_\Omega^2+2A_1\tau^2	\|\delta_t\nabla e_\phi^{n+1}\|_\Omega^2\nonumber\\
	&+&\varepsilon\tau\|\nabla e_\phi^{n+1}\|_\Omega^2+\| \delta_{tt}
 e_\phi^{n+1}\|^2_{-1,\Omega}+4B_1\tau\| e_\phi^{n+1}\|_\Omega^2\nonumber\\
	&+&(\| e_\psi^{n+1}\|^2_{-1,\Gamma}+\|
 2e_\psi^{n+1}-e_\psi^{n}\|^2_{-1,\Gamma})+2A_2\tau^2\|\nabla_{\Gamma}
 e_\psi^{n+1}\|_{\Gamma}^2+2A_2\tau^2
	\|\delta_t\nabla_{\Gamma} e_\psi^{n+1}\|_{\Gamma}^2\nonumber\\
	&+&\delta\kappa\tau \|\nabla_{\Gamma} e_\psi^{n+1}\|_{\Gamma}^2+\| \delta_{tt}
 e_\psi^{n+1}\|^2_{-1,\Gamma}+4B_2\tau\| e_\psi^{n+1}\|_{\Gamma}^2\nonumber\\
	&\leq&
	\| e_\phi^{n}\|^2_{-1,\Omega}+\| 2e_\phi^{n}-e_\phi^{n-1}
\|^2_{-1,\Omega}+2A_1\tau^2\|\nabla
e_\phi^{n}\|_\Omega^2+C_8\tau\varepsilon^{-3}\|2e_\phi^{n}-e_\phi^{n-1}
\|^2_{-1,\Omega}\nonumber\\
	&+&	\| e_\psi^{n}\|^2_{-1,\Gamma}+\|
2e_\psi^{n}-e_\psi^{n-1}\|^2_{-1,\Gamma}+2A_2\tau\|\nabla_{\Gamma}e_\psi\|_{\Gamma}^2
+C_9\tau\delta^{-3}\|2e_\psi^{n}-e_\psi^{n-1}\|^2_{-1,\Gamma}\nonumber\\
	&+&C_{10}\varepsilon^{-1}\tau^4 +C_{11}\delta^{-1}\tau^4,
\end{eqnarray}
where
\[
C_8=8K_1^2+8B_1^2\varepsilon^2,\quad C_9=8K_2^2+8B_2^2\delta^2,
\]
\[
C_{10}=8C_2+8C_4(\frac{K_1^2}{\varepsilon^2}+B_1^2)+8C_6A_1^2,\quad C_{11}=8C_3+8C_5(\frac{K_2^2}{\delta^2}+B_2^2)+8C_7A_2^2.
\]
By using the discrete Gronwall inequality, we obtain
\begin{eqnarray*}	
	&&\max_{1\leq n\leq N}\{\| e_{\phi}^{n+1}\|^2_{-1,\Omega}+\|
2e_{\phi}^{n+1}-e_{\phi}^{n}\|^2_{-1,\Omega}+2A_1\tau^2\| \nabla e_{\phi}^{n+1}
\|_\Omega^2\nonumber\\
	&&+\| e_{\psi}^{n+1}\|^2_{-1,\Gamma}+\|
2e_{\psi}^{n+1}-e_{\psi}^{n}\|^2_{-1,\Gamma}+2A_2\tau^2\| \nabla_{\Gamma}
e_{\psi}^{n+1}\|_{\Gamma}^2\}\nonumber\\
	&+&\sum^N_{n=1}(2A_1\tau^2\| \delta_t \nabla e_{\phi}^{n+1}\|_\Omega^2+\tau\varepsilon\|\nabla
e_{\phi}^{n+1}\|_\Omega^2+\| \delta_{tt} e_{\phi}^{n+1}\|^2_{-1,\Omega}+4B_1\tau\|
e_{\phi}^{n+1}\|_\Omega^2\nonumber\\
	&+&2A_2\tau^2\| \delta_t \nabla_{\Gamma} e_{\psi}^{n+1}\|_{\Gamma}^2
+\tau\delta\kappa\|\nabla_{\Gamma} e_{\psi}^{n+1}\|_{\Gamma}^2+\| \delta_{tt}
e_{\psi}^{n+1}\|^2_{-1,\Gamma}+4B_2\tau\| e_{\psi}^{n+1}\|_{\Gamma}^2)\nonumber\\
	&\leq& \exp\left((C_8\varepsilon^{-3}+C_9\delta^{-3})T\right)\left(C_{10}\varepsilon^{-1}+C_{11}\delta^{-1}
	+C_0(5+2A_1\varepsilon^{-1}\tau)+C_1(5+2A_2\delta^{-1}\kappa^{-1}\tau)\right)\tau^4.
\end{eqnarray*}
This completes the proof.
\end{proof}

\section{Numerical experiments}
\label{sec5}
In this section, we present some numerical experiments of the Liu-Wu model by scheme \eqref{scheme-1}-\eqref{scheme-6}
in two dimensions. For time discretization, we use the BDF2 scheme. For spatial operators, we use the
second-order central finite difference method to discretize them on a uniform spatial grid. For such a linear
scheme, we use the generalized minimum residual method as the linear solver. We conduct the experiments on the
rectangular domain $[0,1]^2$.

\subsection{Accuracy test}
\label{SubSecAT}
In this section, numerical accuracy tests using the scheme \eqref{scheme-1}-\eqref{scheme-6} are presented to
support our error analysis. Let $\Omega$ to be the unit square, the spatial step size $h = 1/256$ and the time
step $\tau=0.08, \,0.04, \,0.025, \,0.0125,\,0.01, \,0.005$. The parameters
are chosen as $\varepsilon=\delta=0.02,\, \kappa=0.02,\, A_1 = 68,\,A_2 = 150, \,B_1 = 120$ and $B_2 = 120$. The initial
data is taken as the piecewise constant setting: 
\begin{equation}
\label{condition3}
\phi_0(x,y)=\left\{
\begin{array}{ll}
0,\quad &  x \in \Omega,\\
1,\quad  &  x \in \Gamma.
\end{array}\right.
\end{equation}
 We choose $F$ and $G$ to be the modified double-well potential as
\[
F(x)=G(x)=\left\{
\begin{array}{ll}
        (x-1)^2,              & x > 1,\\
		\frac{1}{4}(x^2-1)^2,     & -1 \leq x \leq  1,\\
		(x+1)^2,              & x < -1.
\end{array}
\right.
\]
Therefore, the second derivative of $F$ with respect to $\phi$ and the second derivative of $G$ with
respect to $\psi$ are bounded
\[
\max_{\phi \in \mathbb{R}}| F^{\prime\prime}(\phi)|
 =\max_{\psi \in \mathbb{R}}| G^{\prime\prime}(\psi)| \leq 2.
\]
The errors are calculated as the difference between the solution of the coarse time step and that of the
reference time step $\tau=2.5\times 10^{-4}$. In Figure \ref{fig17}, we plot the sum of $L^2$ errors of $\phi$
and $\psi$ between the numerical solution and the reference solution at $T = 4$ with different time step
sizes. The result shows clearly that the slope of fitting line is $2.0653$, which
in turn verifies the convergence rate of the numerical scheme is asymptotically at least second-order
temporally for $\phi$ and $\psi$, which is consistent with our numerical analysis in Section~\ref{sec4}.
\begin{figure}[h]
	\centering
	\includegraphics[height=0.28\textwidth,width=0.28\textwidth]{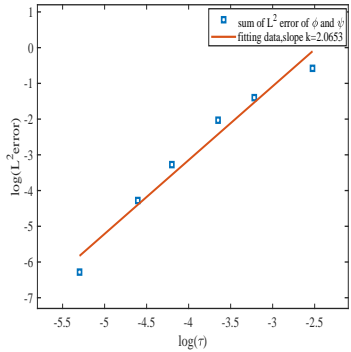}
	\caption{The numerical errors $\|e_\phi\|_\Omega+\|e_\psi\|_\Gamma$ at $T = 4$.}
	\label{fig17}
\end{figure}

\subsection{Cases with different initial conditions}
\label{SubSecIC}

We consider the numerical approximations for the Liu-Wu model with different initial conditions.

\medskip

\noindent \textbf{Case\, 1.}\quad The initial condition is set as piecewise constants:
\begin{equation}
\label{condition1}
\phi_0(x,y)=\left\{
\begin{array}{ll}
1 &  x>1/2,\\
-1  &   x\leq 1/2.
\end{array}
\right.
\end{equation}
\begin{figure}[h]
	\centering
	\includegraphics[height=0.28\textwidth,width=0.28\textwidth]{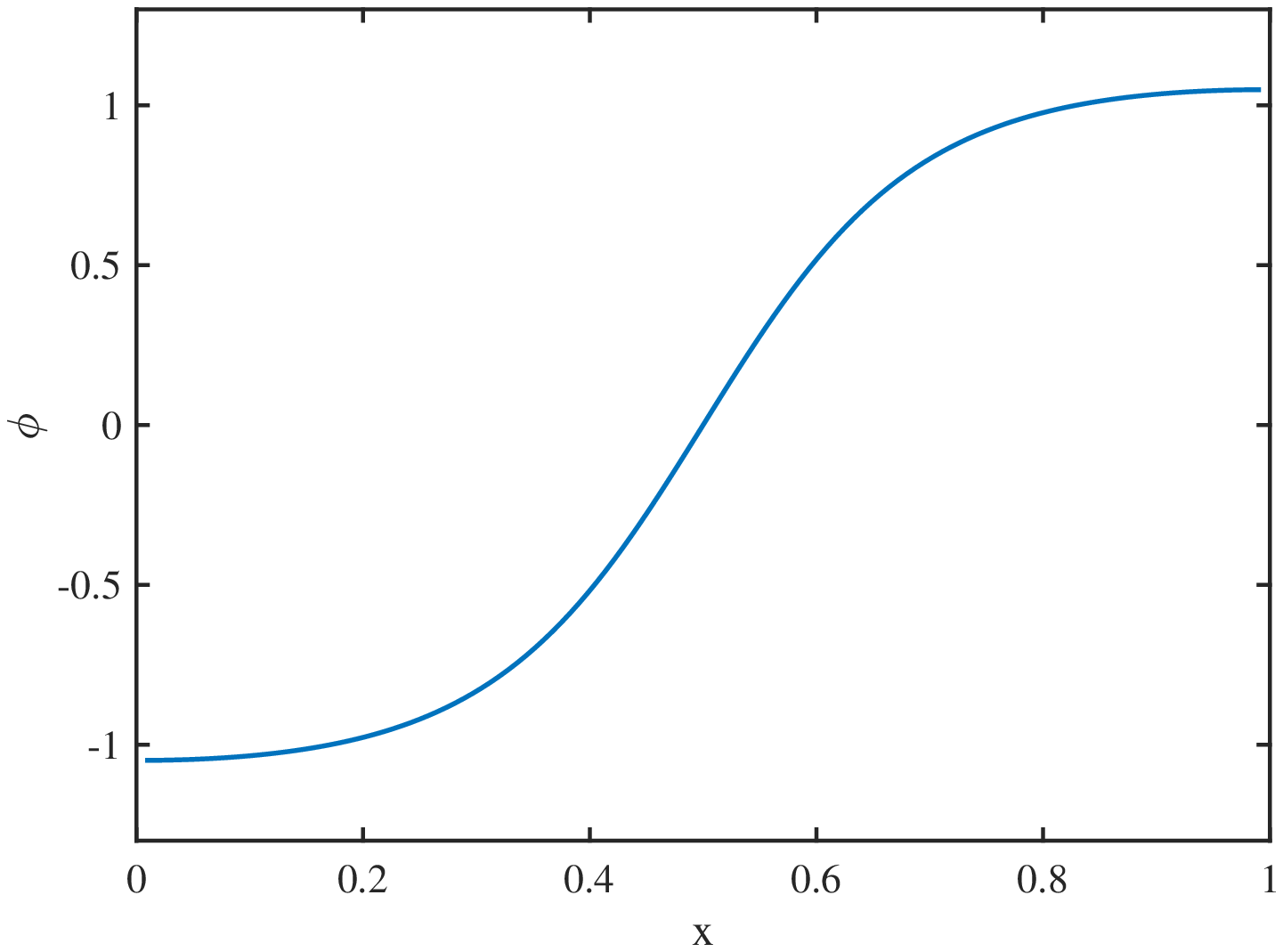}
	\includegraphics[height=0.28\textwidth,width=0.28\textwidth]{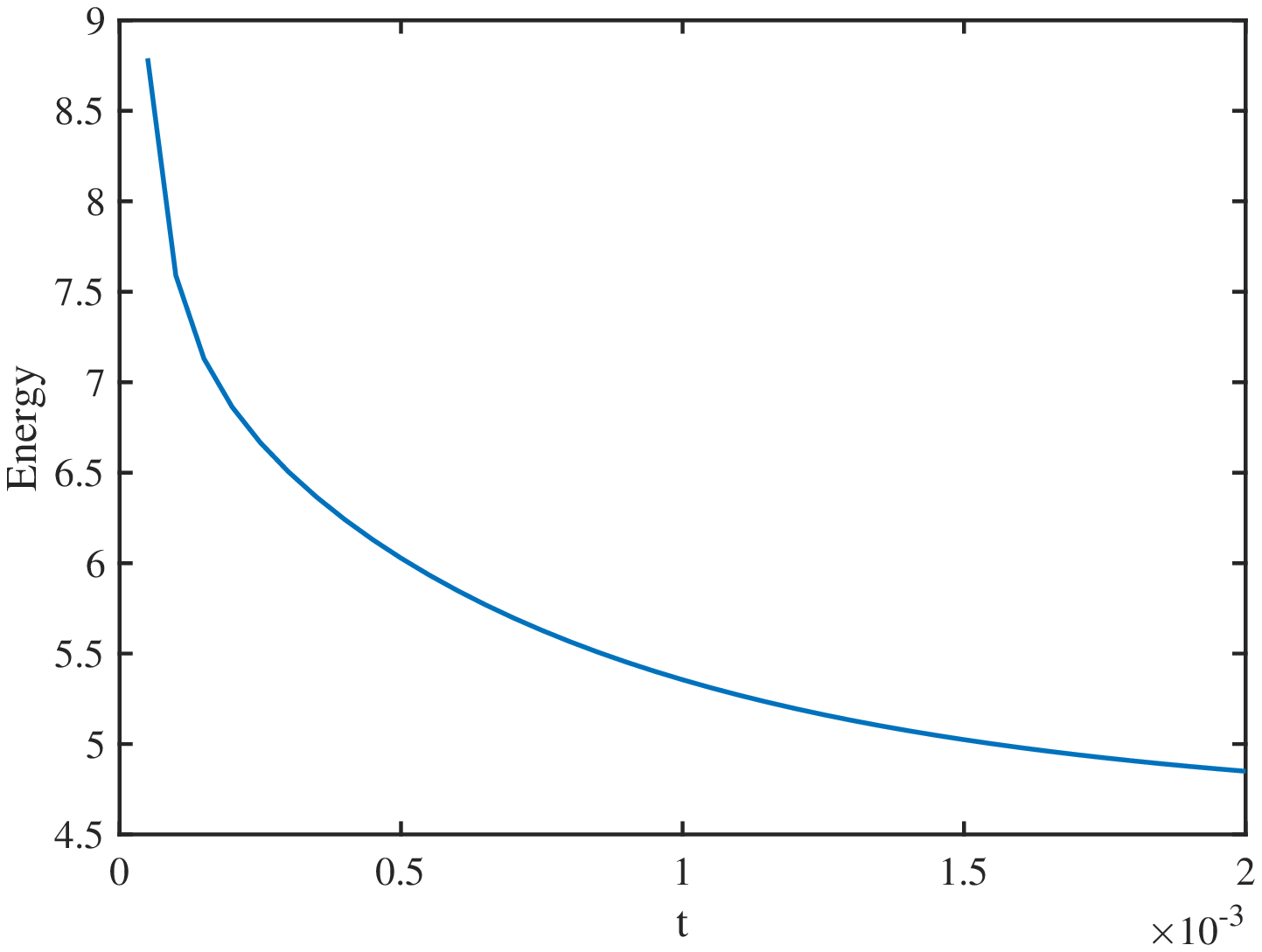}
	\includegraphics[height=0.28\textwidth,width=0.28\textwidth]{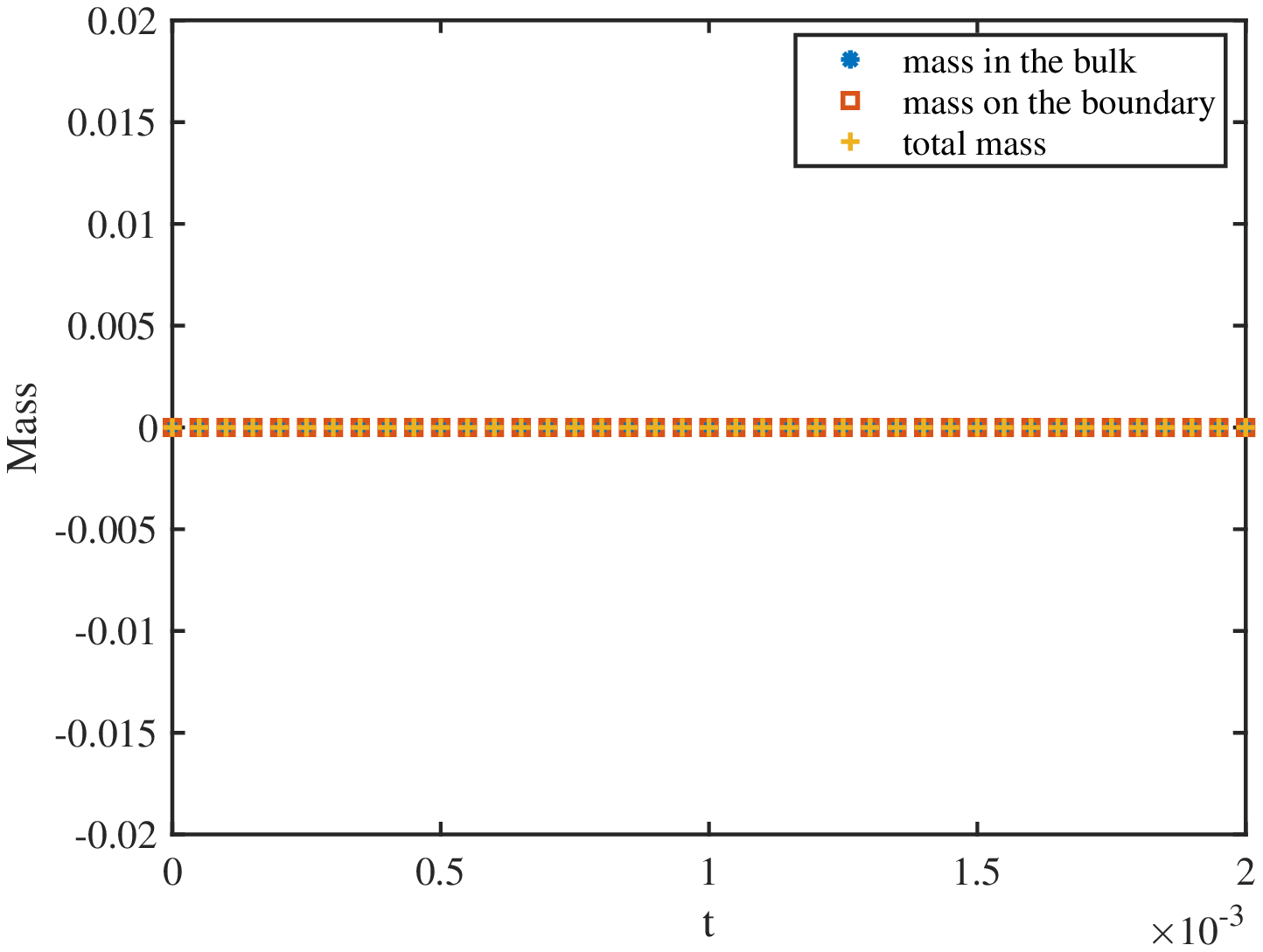}
	\caption{Projection of numerical solution on $y=\frac12$ at $t=0.002$ (left); Energy evolution (middle);
Mass evolution (right) for $0\leq t\leq0.002$ with initial condition \eqref{condition1}.}		
	\label{fig2}
\end{figure}
In this example, the time step $\tau=10^{-5}$ and the spacial size $h=0.01$.
The parameters are set as $\varepsilon=1,\, \delta=0.1, \,\kappa=1, \,A_1=A_2=1,\,B_1=1$ and $B_2=10$. We take the classical double
well potential function (\ref{FGdwell}).
We only plot the cutline of solution on $y=\frac12$ at $t=0.002$ in Figure \ref{fig2}, since the
the numerical result is almost a constant in the vertical direction. It is consistent with the literature works.
 The evolution of energy and mass with time are also shown in the Figure \ref{fig2}, which reveals the
energy stability and the conservation of mass in the region and the boundary.

\medskip

\noindent \textbf{Case\, 2.}\quad Consider the initial condition
\begin{equation}
\label{condition2}
\phi_0(x,y)=\sin(4\pi x)\cos(4\pi y).
\end{equation}
\begin{figure}[h]
\centering
\includegraphics[height=0.28\textwidth,width=0.28\textwidth]{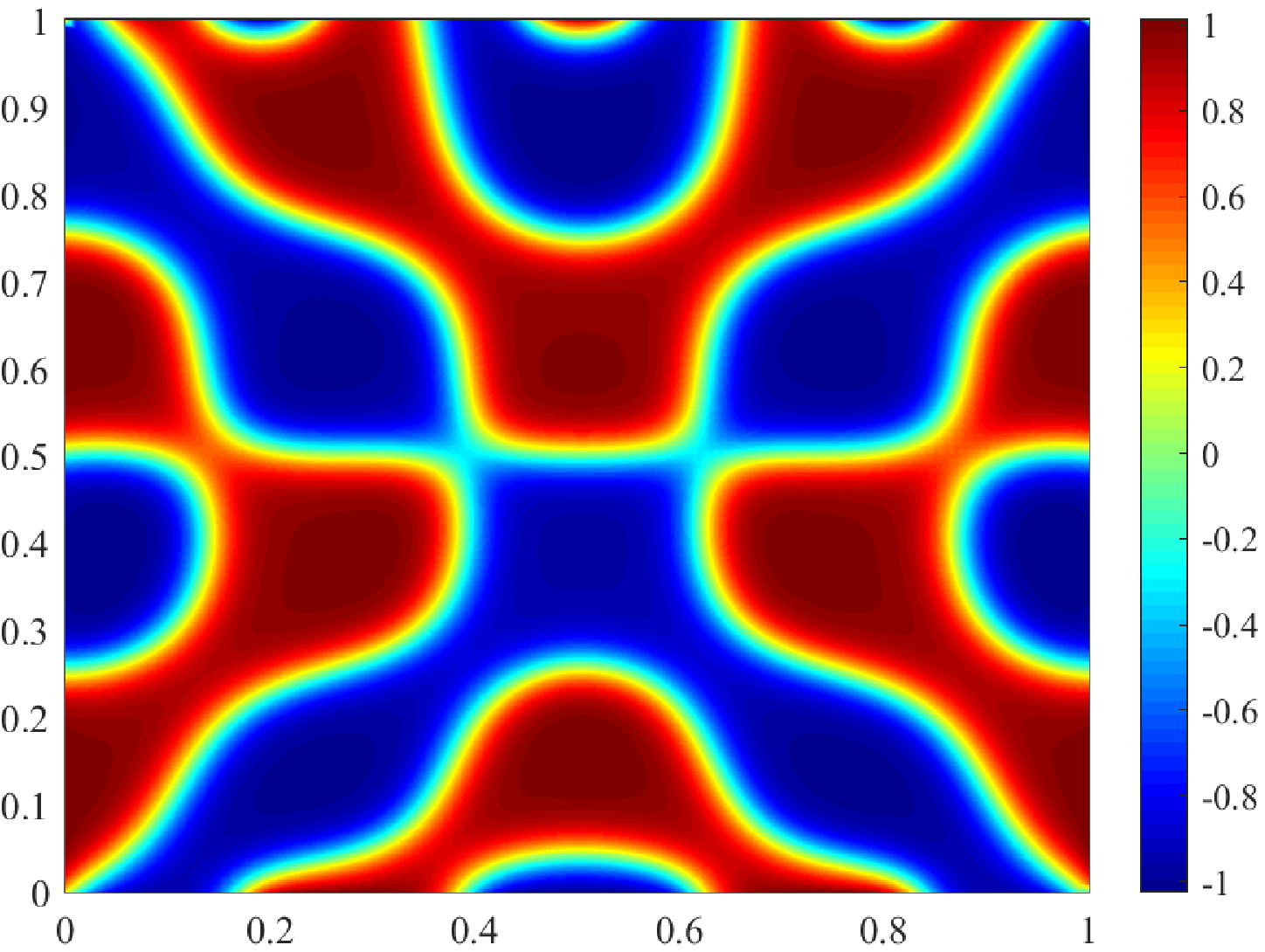}
\includegraphics[height=0.28\textwidth,width=0.28\textwidth]{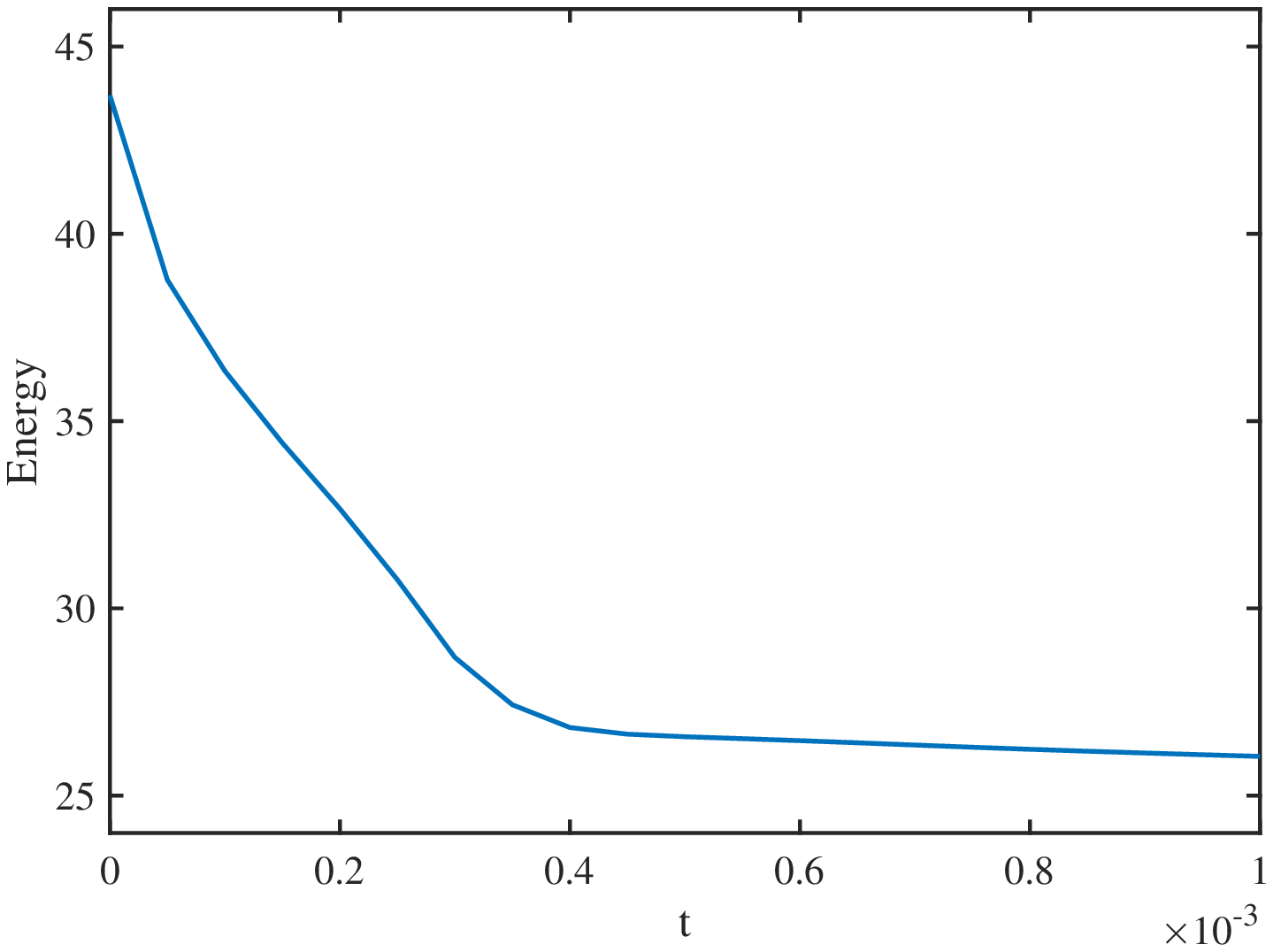}
\includegraphics[height=0.28\textwidth,width=0.28\textwidth]{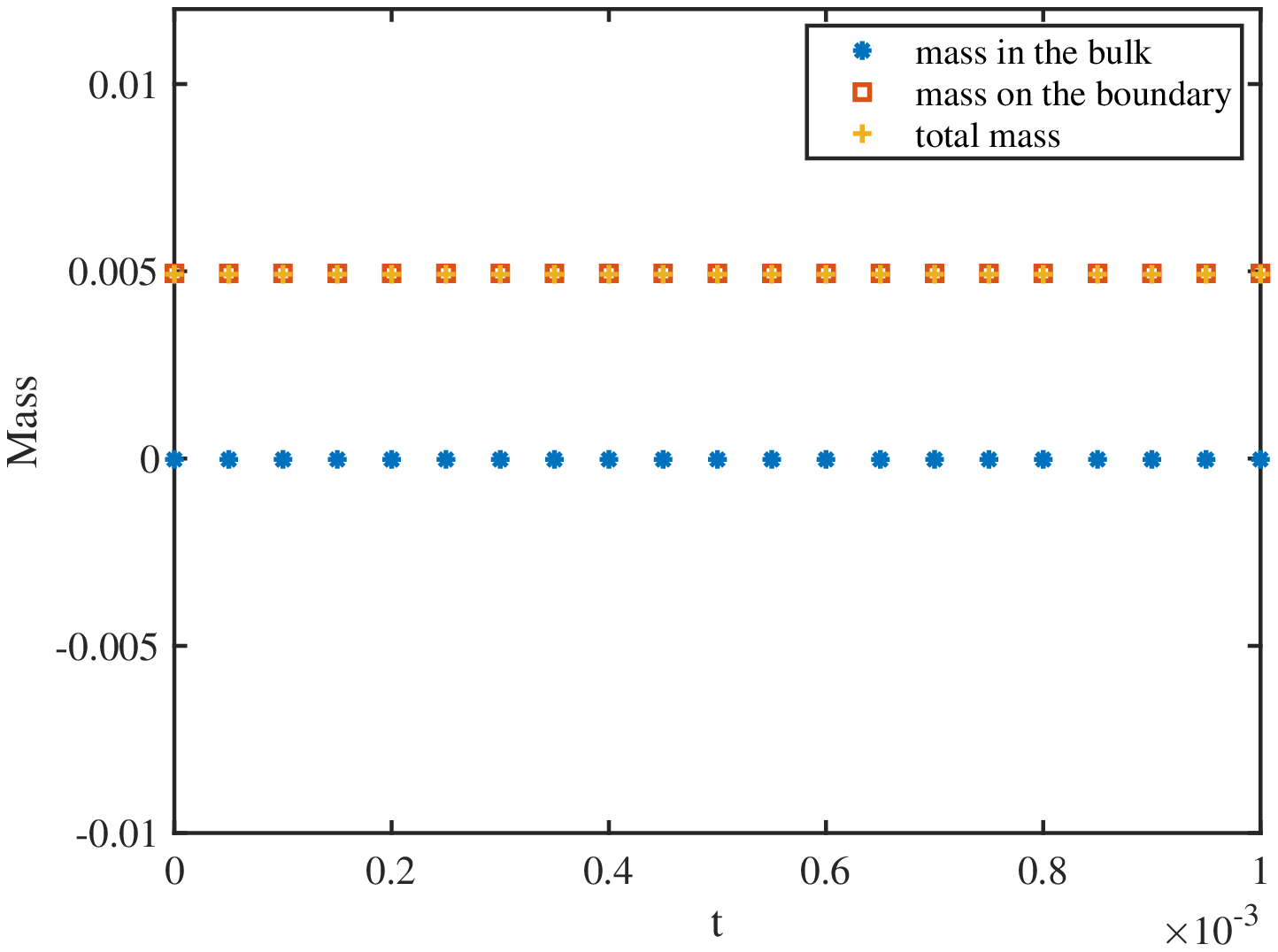}
\caption{Numerical solution $\phi$ and $\psi$ at $t=0.001$ (left); Energy evolution (middle) for $0\leq t\leq 0.001$;
Mass evolution (right) for $0\leq t\leq 0.001$ with initial condition \eqref{condition2}.}		
\label{fig3}
	\end{figure}
Here, the time step $\tau=10^{-5}$ and the spacial
size $h=0.01$. The parameters are set as $\varepsilon=\delta=0.02,\, \kappa=1,\, A_1=A_2=1,\, B_1=B_2=50$ to ensure
that the scheme is stable. The numerical solution at $t=0.001$ is displayed in Figure \ref{fig3}. The development
of energy and total mass for $0\leq t\leq 0.001$
 is also shown in the Figure \ref{fig3}, which reveals the energy stability and the conservation
of mass in the region and the boundary, respectively. It is seen that the total energy has a quick decay in the
early stage until $t=0.0004$, and then the energy decreases lightly.

\medskip

\noindent \textbf{Case\, 3.}\quad We reproduce the numerical experiment in Section \ref{SubSecAT} ever studied by
Garcke and Knopf \cite{garcke2020weak}. The initial data is set to $0$ at interior points and  $1$ on the boundary
points. The time step is $\tau=8\times 10^{-6}$ and the spacial step is $h=0.01$.
The parameters are set as $\varepsilon=\delta=0.02$ and $\kappa=0.02$. The stability parameters are  $A_1=A_2=5, \,B_1=B_2=100$.
\begin{figure}[h]
	\centering
   \includegraphics[height=0.28\textwidth,width=0.28\textwidth]{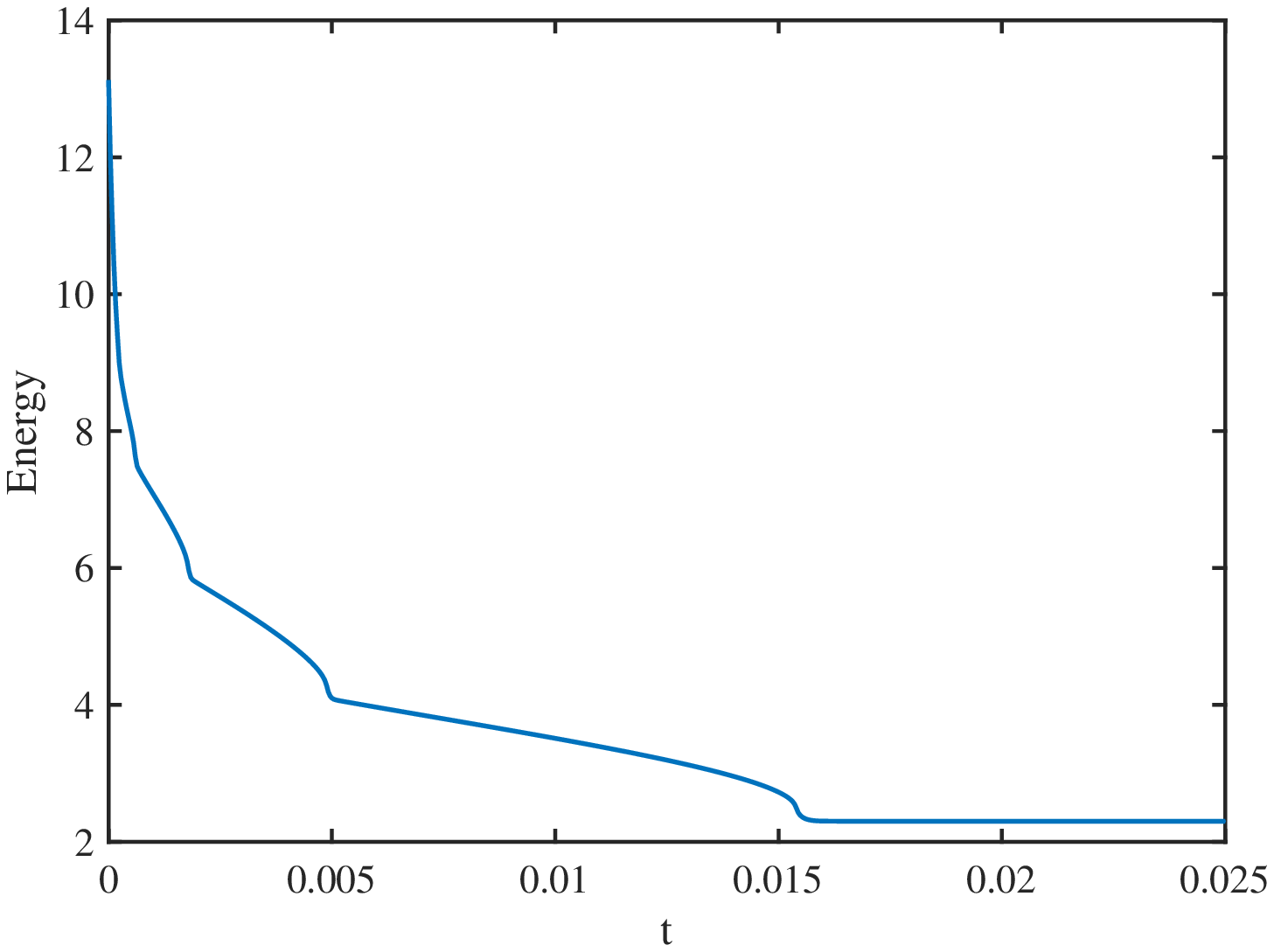}
	\includegraphics[height=0.28\textwidth,width=0.28\textwidth]{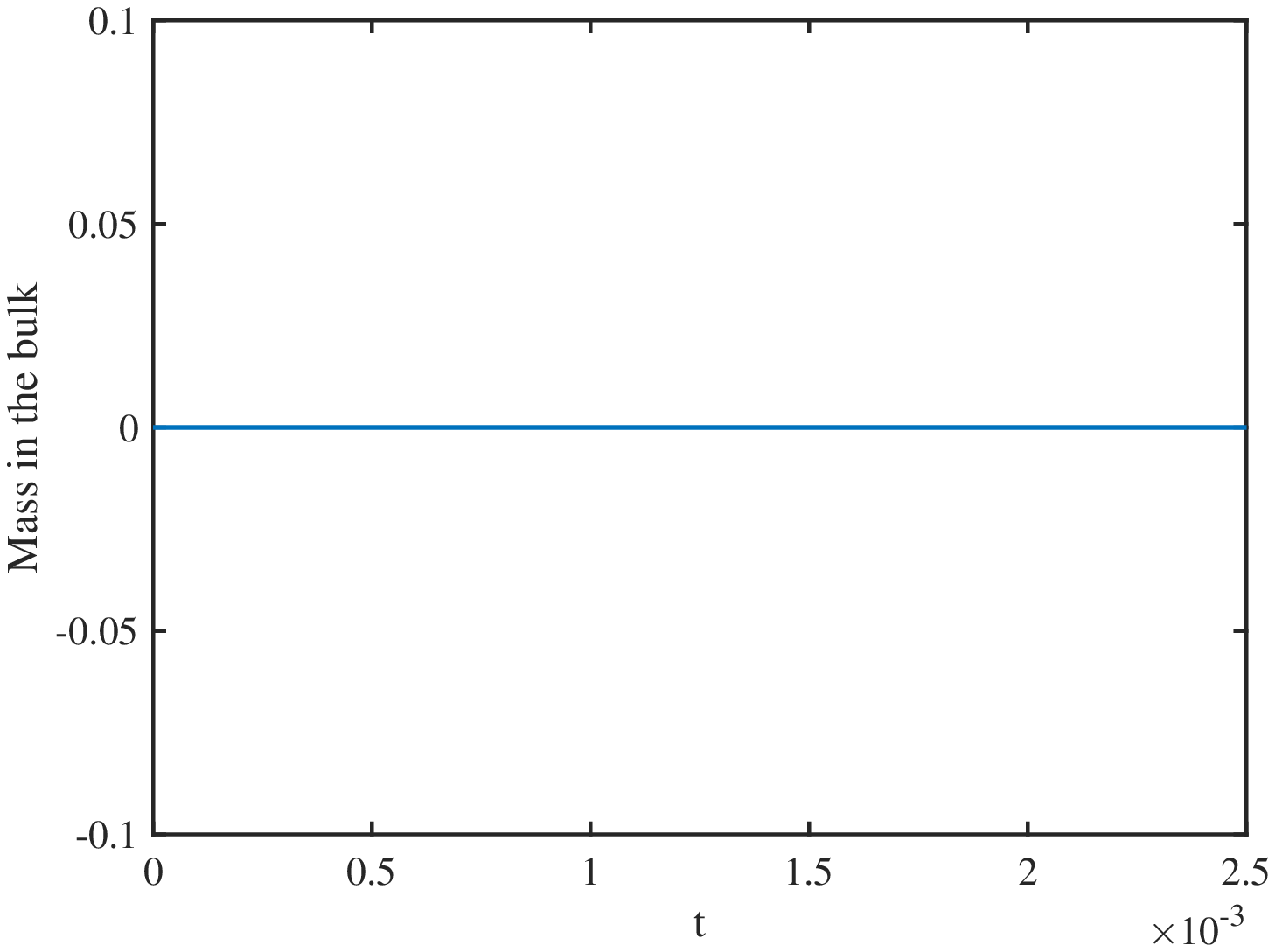}
	\includegraphics[height=0.28\textwidth,width=0.28\textwidth]{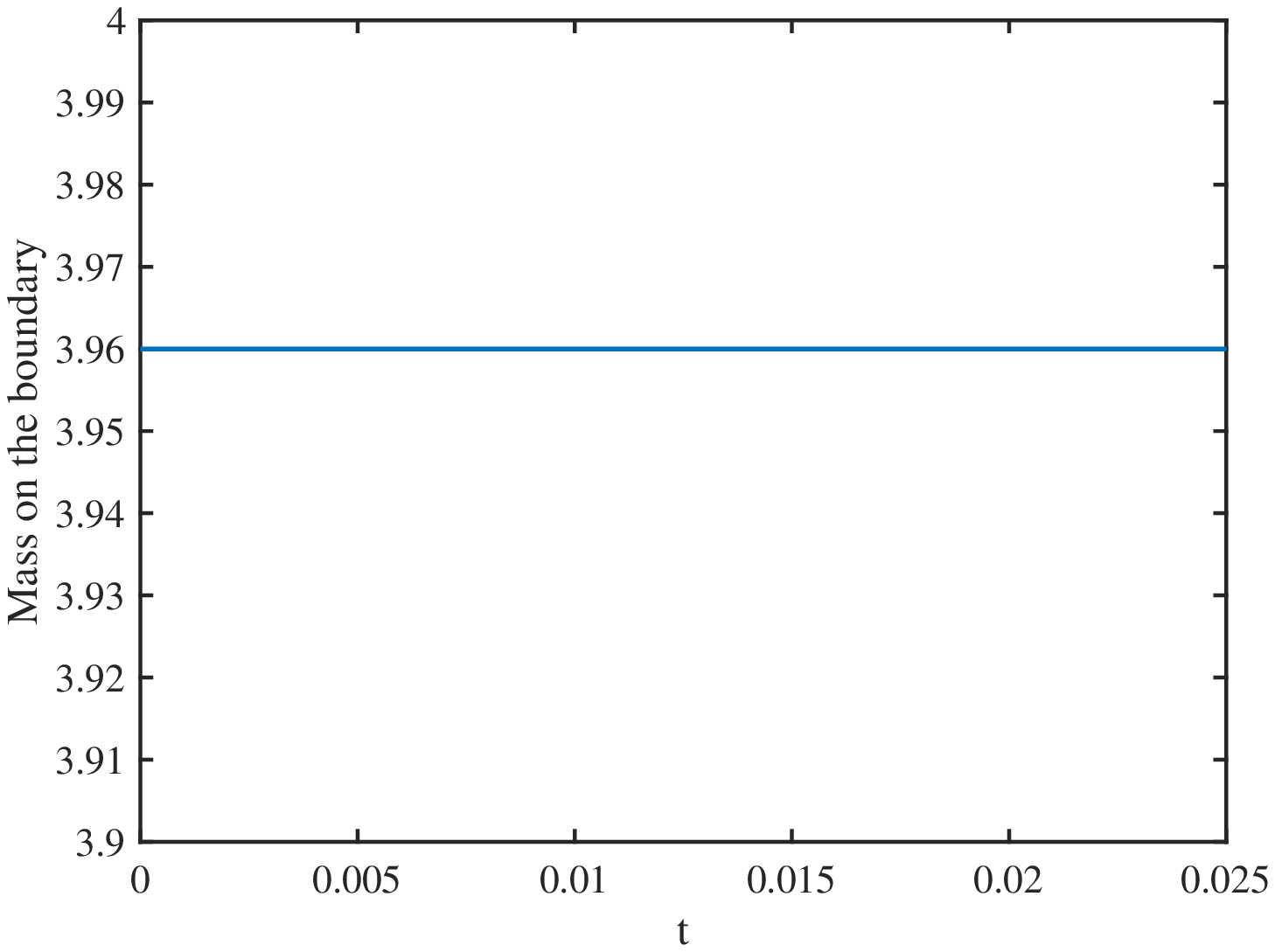}
\caption{Total energy development for $0\leq t\leq 0.025$ (left); Mass in the bulk (middle); Mass on the boundary (Right).}			
\label{fig5}
\end{figure}

\begin{figure}[h]
	\centering
	\includegraphics[height=0.28\textwidth,width=0.28\textwidth]{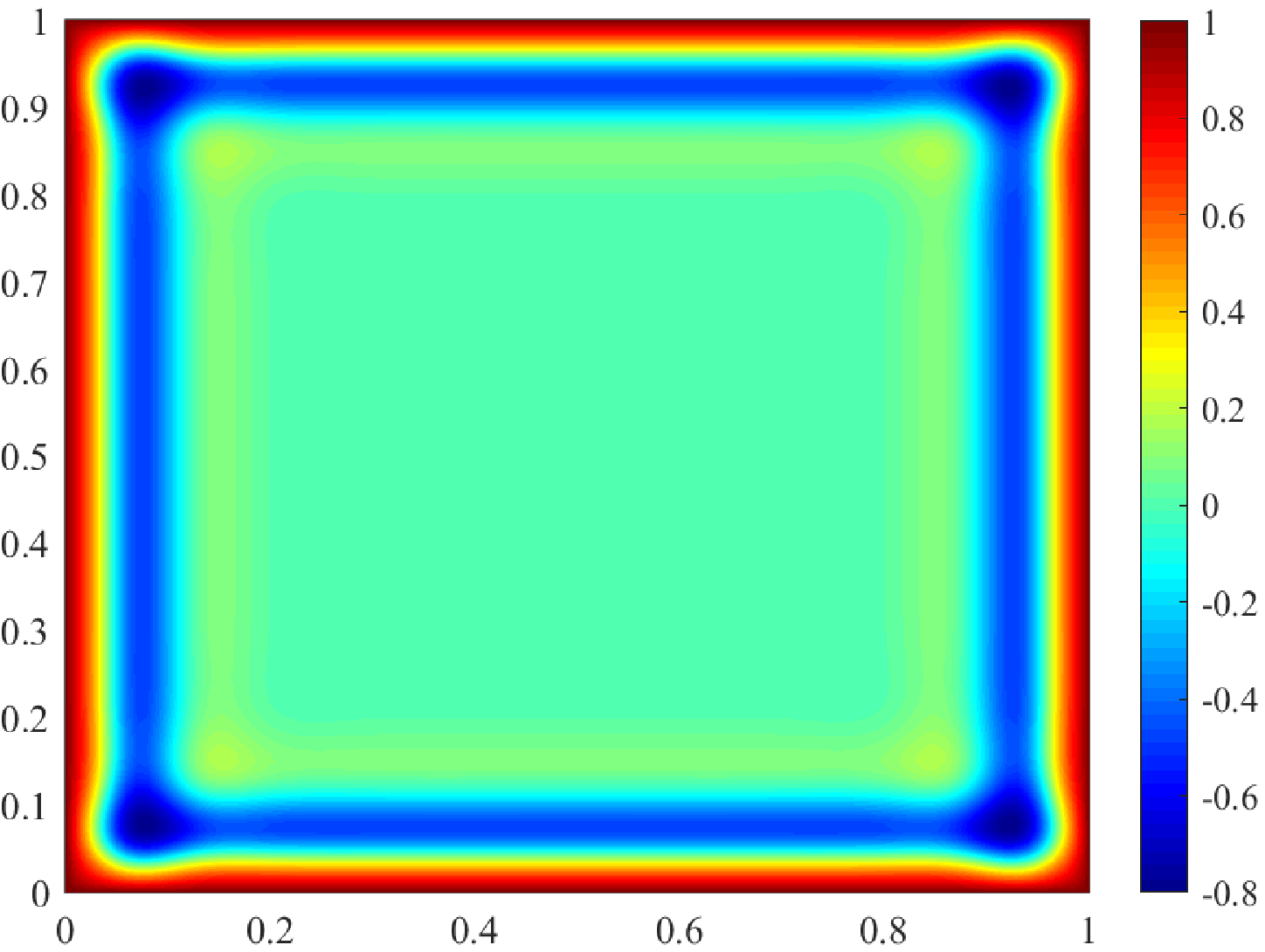}
	\includegraphics[height=0.28\textwidth,width=0.28\textwidth]{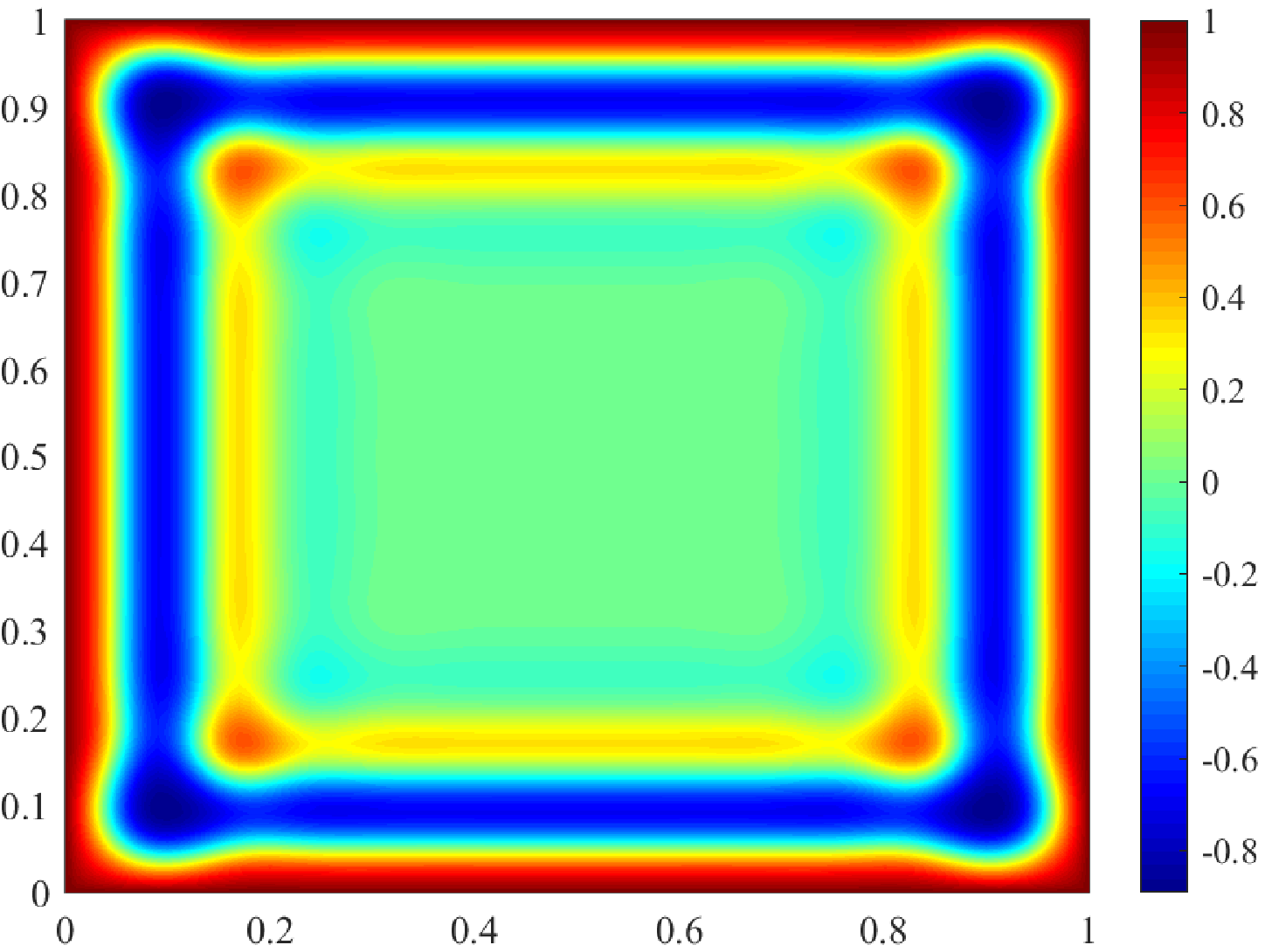}
	\includegraphics[height=0.28\textwidth,width=0.28\textwidth]{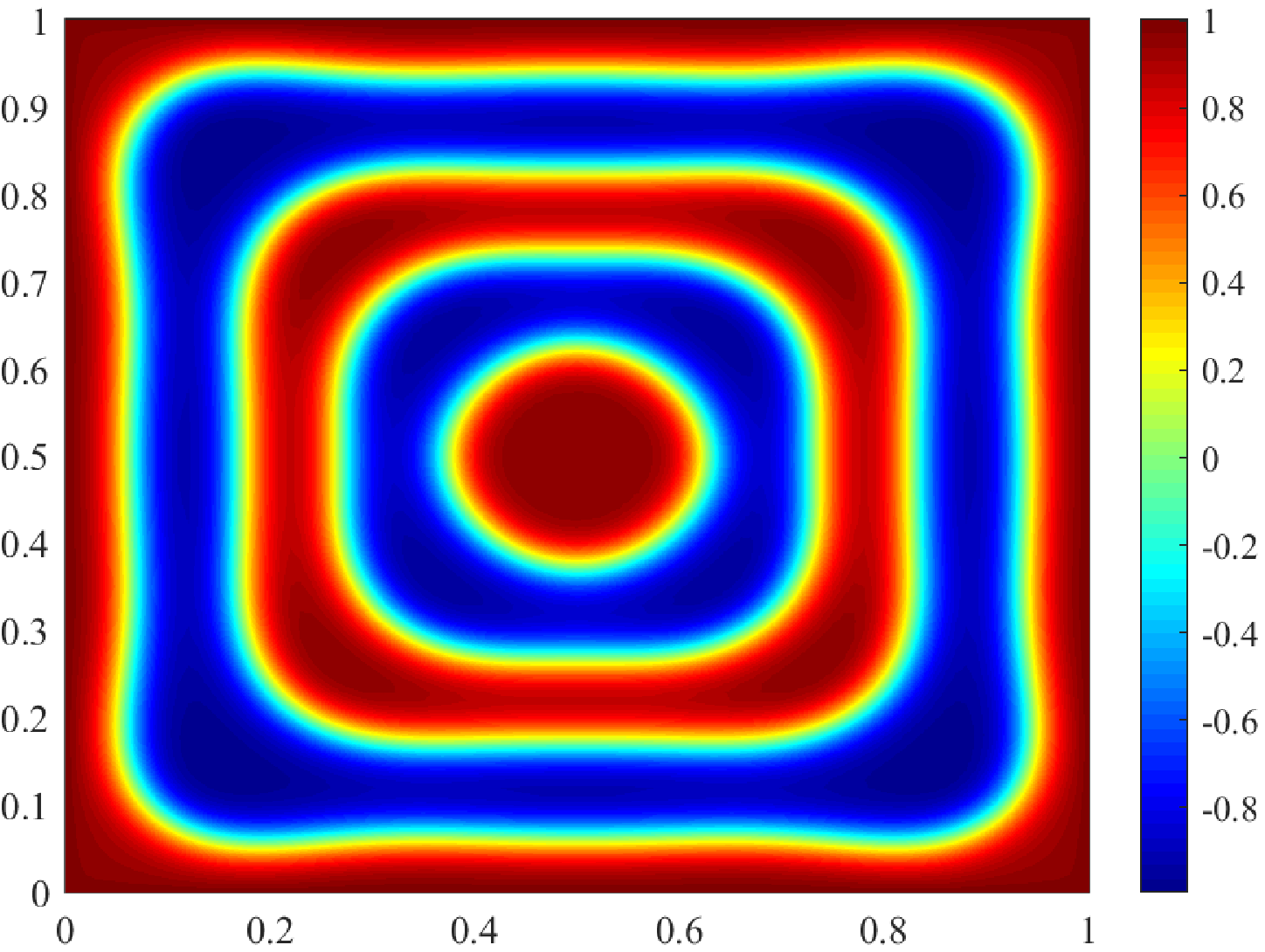}
	\includegraphics[height=0.28\textwidth,width=0.28\textwidth]{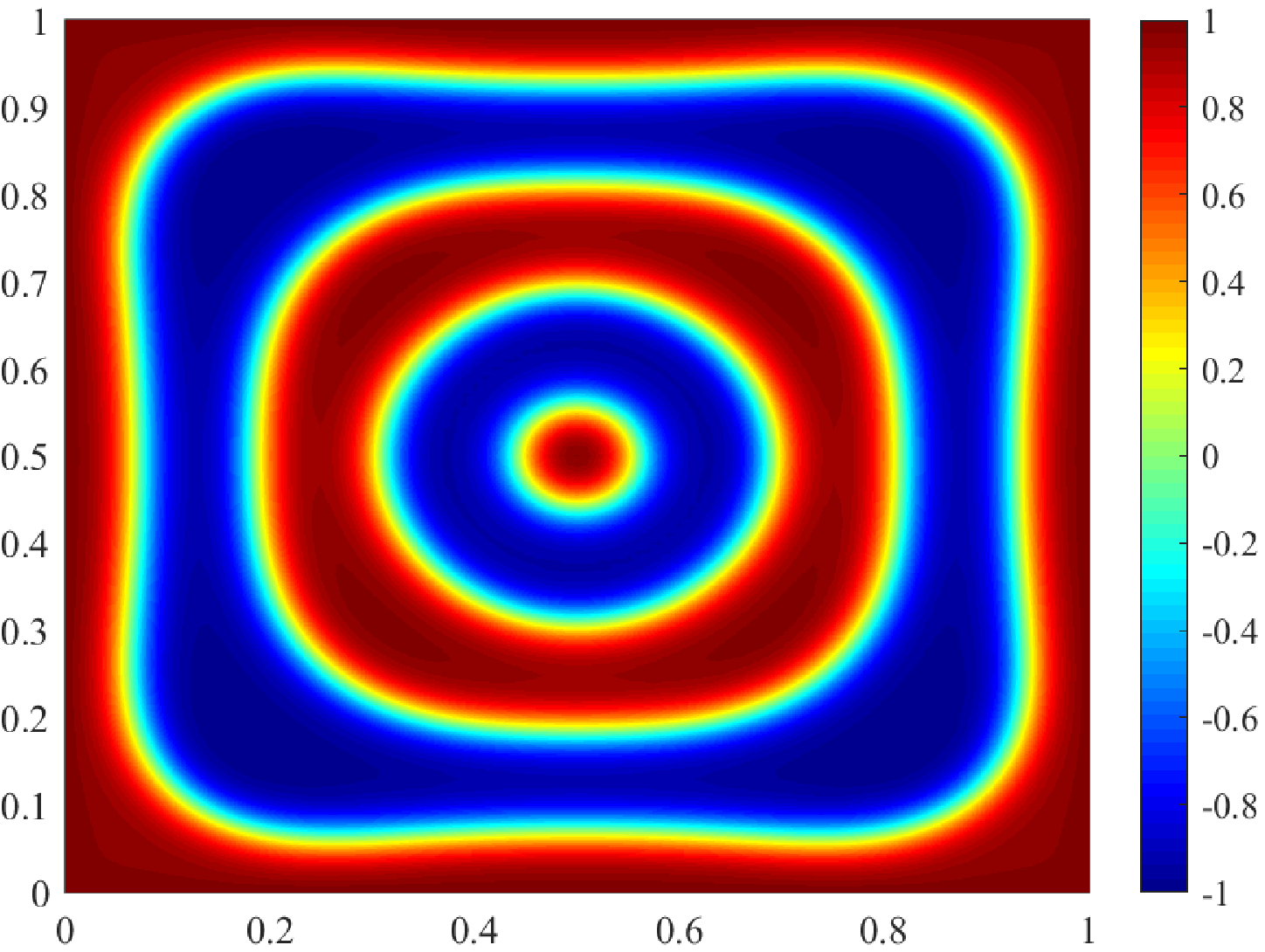}
	\includegraphics[height=0.28\textwidth,width=0.28\textwidth]{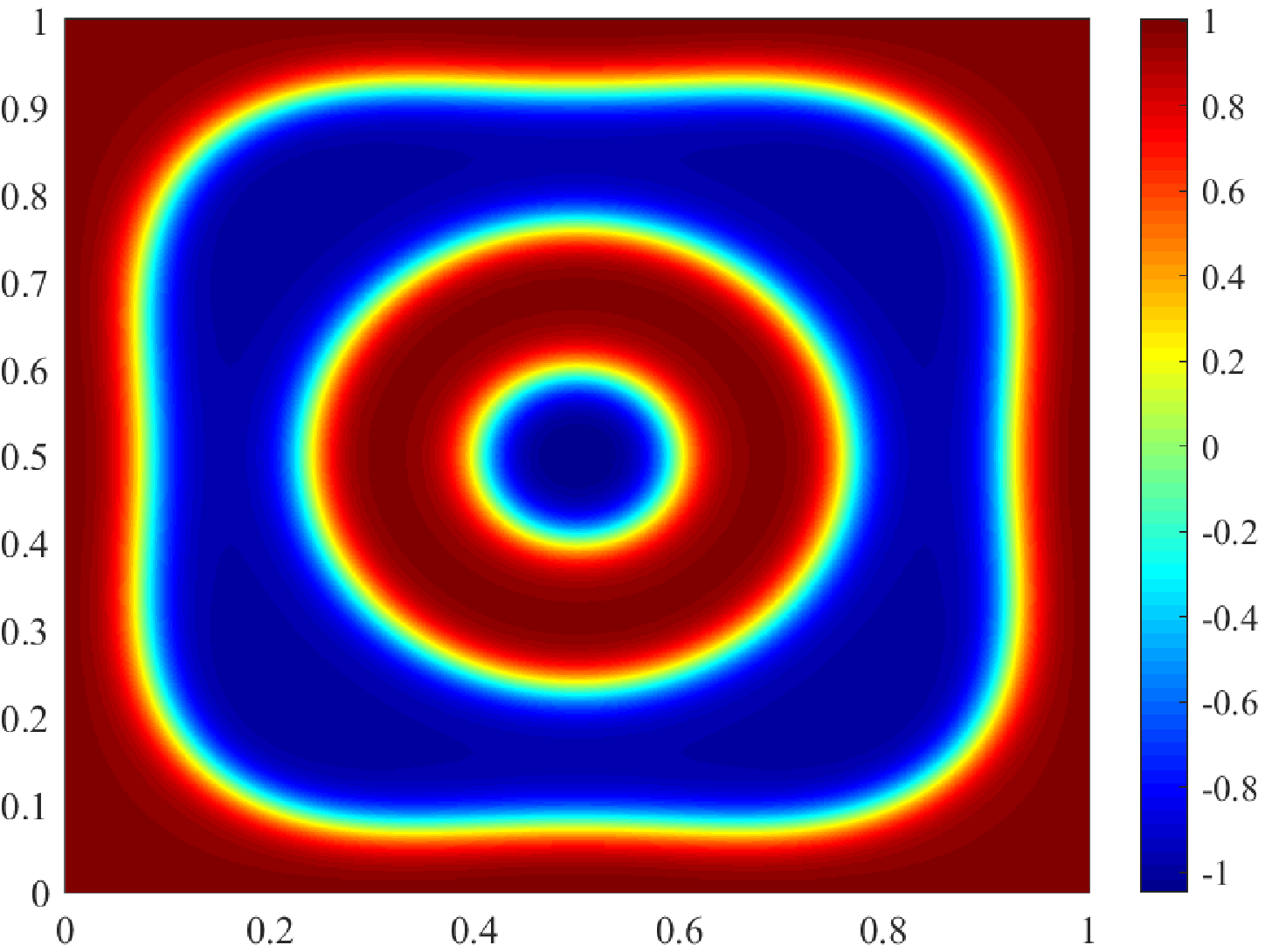}
	\includegraphics[height=0.28\textwidth,width=0.28\textwidth]{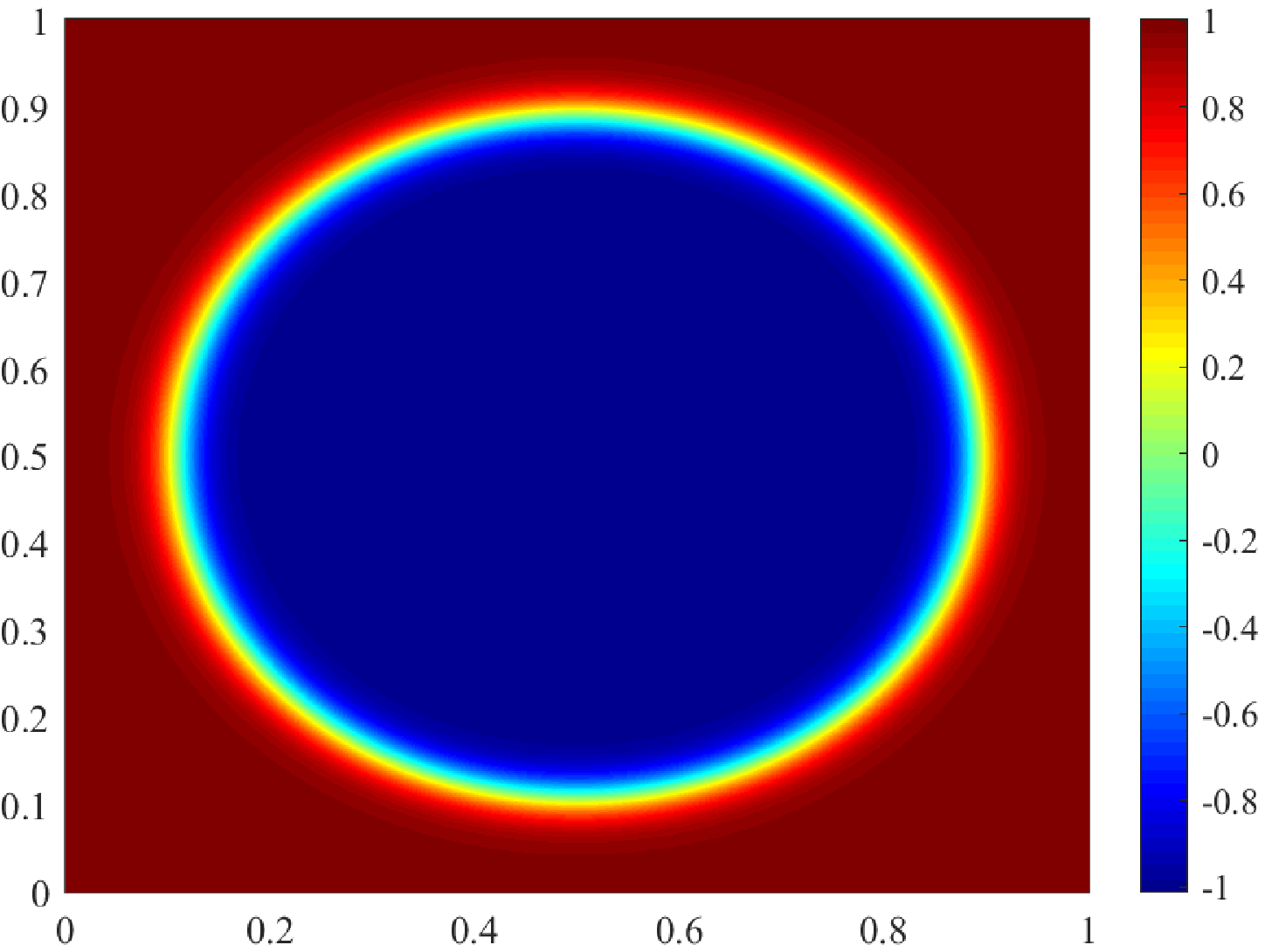}
	\caption{Snapshots of the phase variable $\phi$ at $t=0.00004,\,0.00008,\,0.00064,\,0.0016,\,0.004,\,0.02$.}		
	\label{fig6}
\end{figure}
The evolution of energy is presented in Figure \ref{fig5}. It is observed that the energy decays quickly initially until
about $t=0.017$, and then the energy curve trends to become flat, which implies the system reaches a steady state. We also show
the curves for the development of the mass in the bulk (total $\phi$) and on the boundary (total $\psi$) in Figure \ref{fig5}.
Obviously, the both kinds of  mass are conserved respectively, which is consistent with theoretical result (\ref{MassCSV}).

The numerical solutions at $t=0.00004,\,0.00008,\,0.00064,\,0.0016,\,0.004$ and $0.02$ are displayed in Figure \ref{fig6}. Due to the
conservation of mass on the boundary, the numerical solution remains $1$ throughout the computation. A wavy
structure begins to form starting from the initial time, and then multi-layered wavy structure is evolved gradually.
Next, the multi-layered structure may be developed to the steady state: a circle centered in the region with $-1$
inside and $1$ outside the circle. These numerical results are consistent with the reference works in the literatures.

\medskip

\noindent \textbf{Case\, 4.}\quad We simulate a phase separation process in the case of
vanishing adsorption rates. The initial configuration is
\[
\phi_0(x,y)=\max\{0.1\sin(\pi x),\,0.1\sin(\pi y)\}.
\]
Here, we take the classical double well potential function (\ref{FGdwell}). The time step $\tau=8\times 10^{-5}$ and the spacial step is $h=0.01$.
The parameters are set as $\varepsilon=\delta=0.02$ and $\kappa=1$. The stability parameters are  $A_1=A_2=5, \,B_1=B_2=100$, which is compared with those listed in Section 5.1 in \cite{metzger2022convergent}.


\begin{figure}[h]
	\centering
	\includegraphics[height=0.23\textwidth,width=0.23\textwidth]{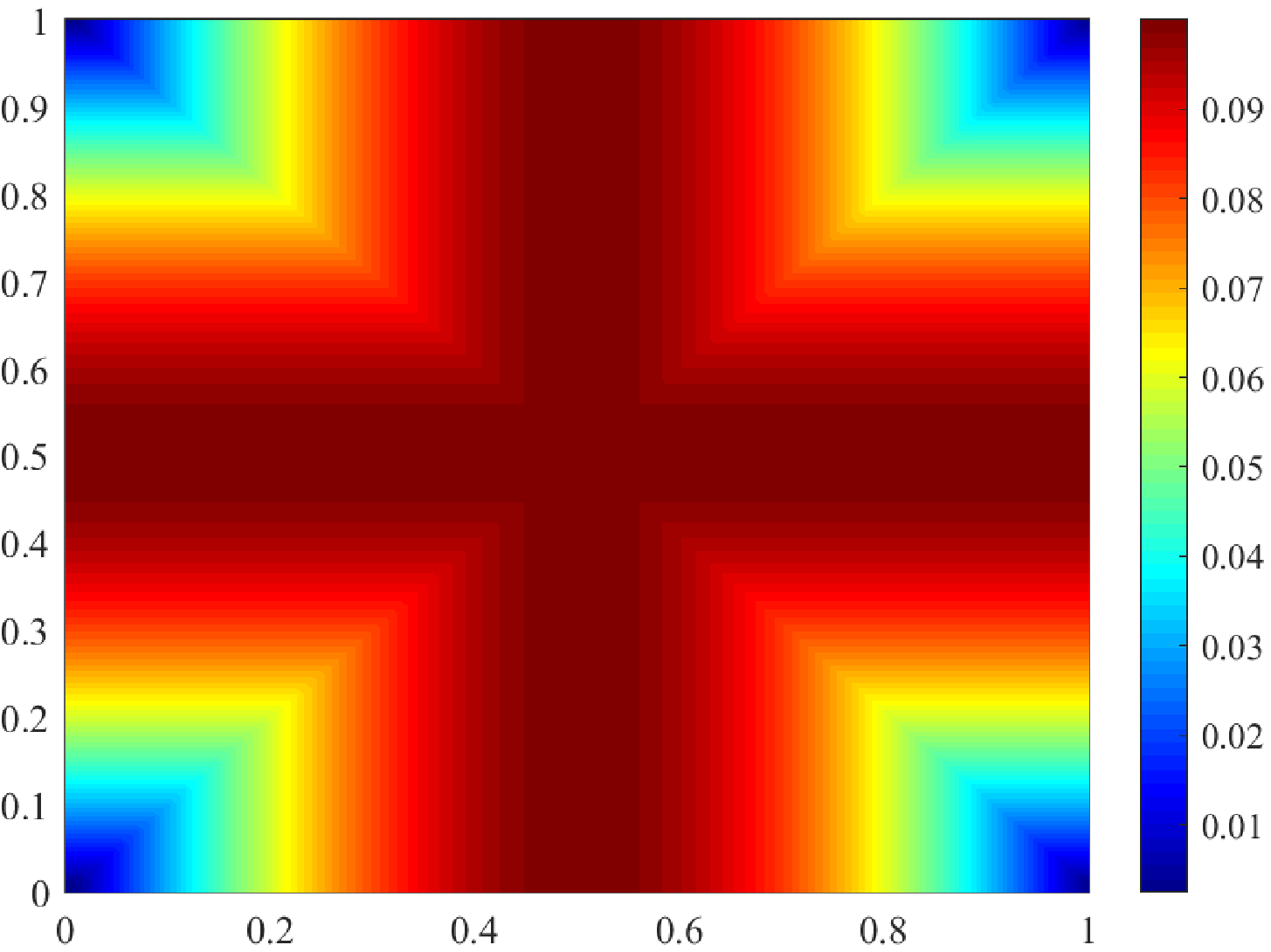}
	\includegraphics[height=0.23\textwidth,width=0.23\textwidth]{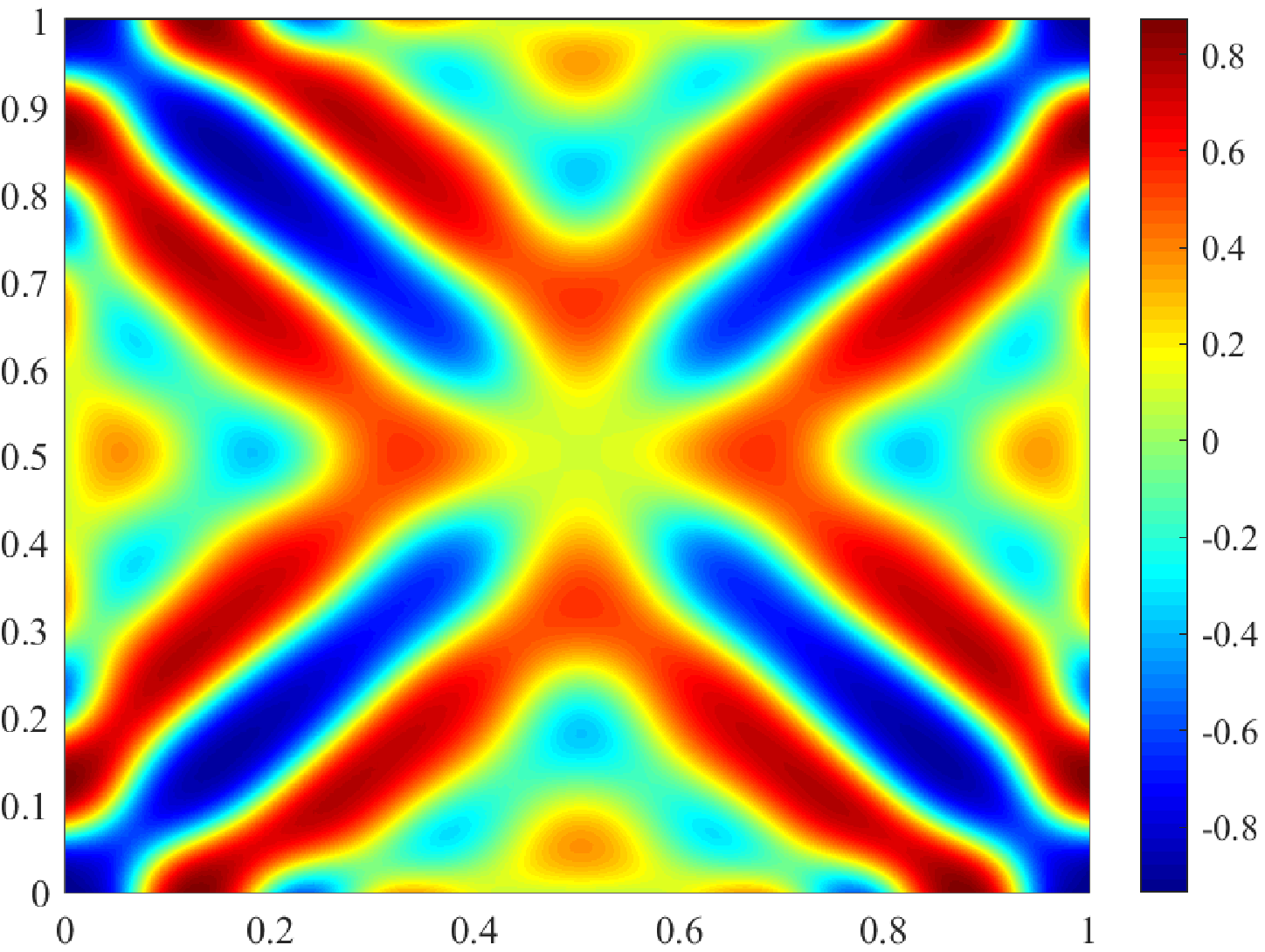}
	\includegraphics[height=0.23\textwidth,width=0.23\textwidth]{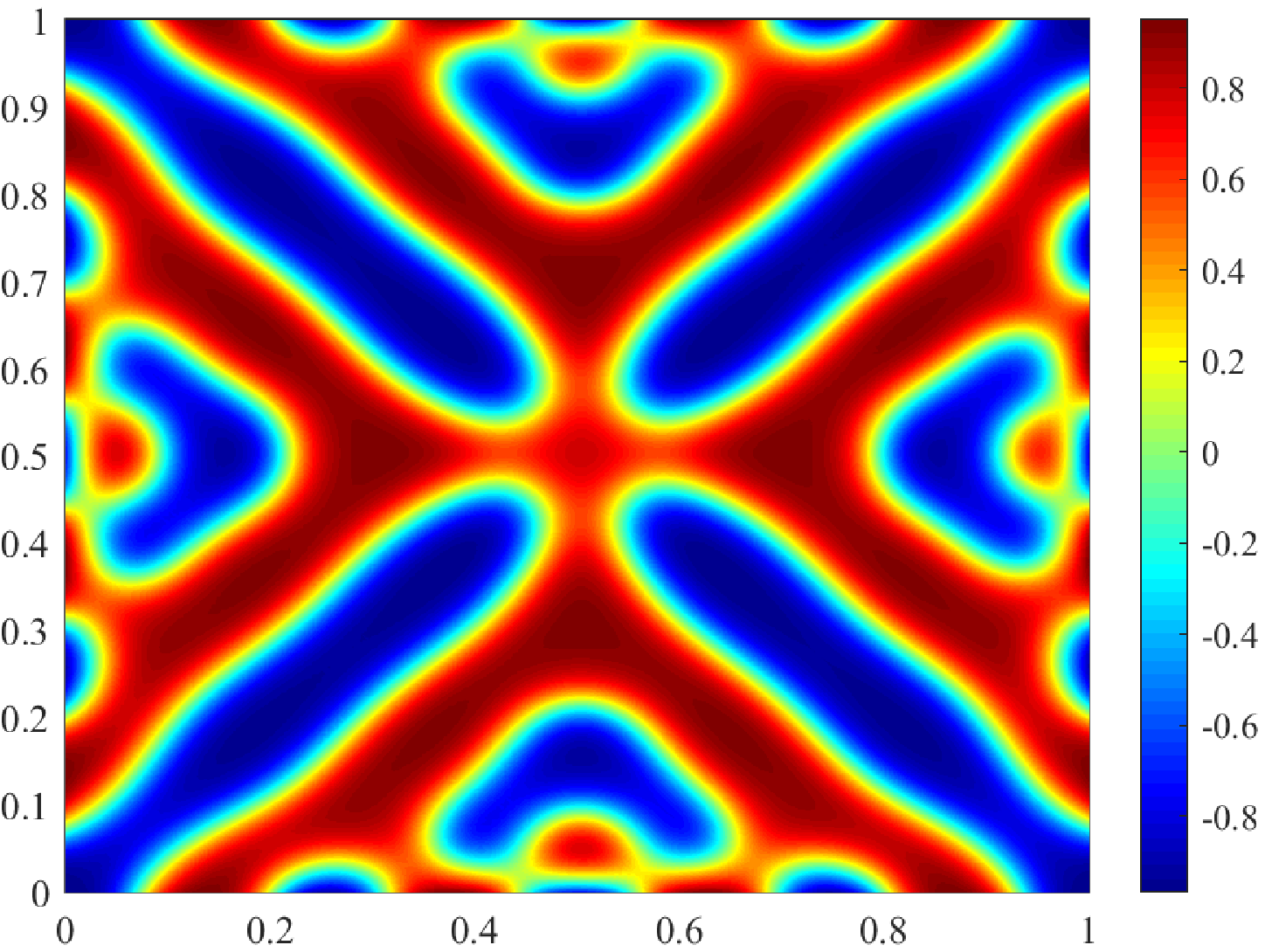}
	\includegraphics[height=0.23\textwidth,width=0.23\textwidth]{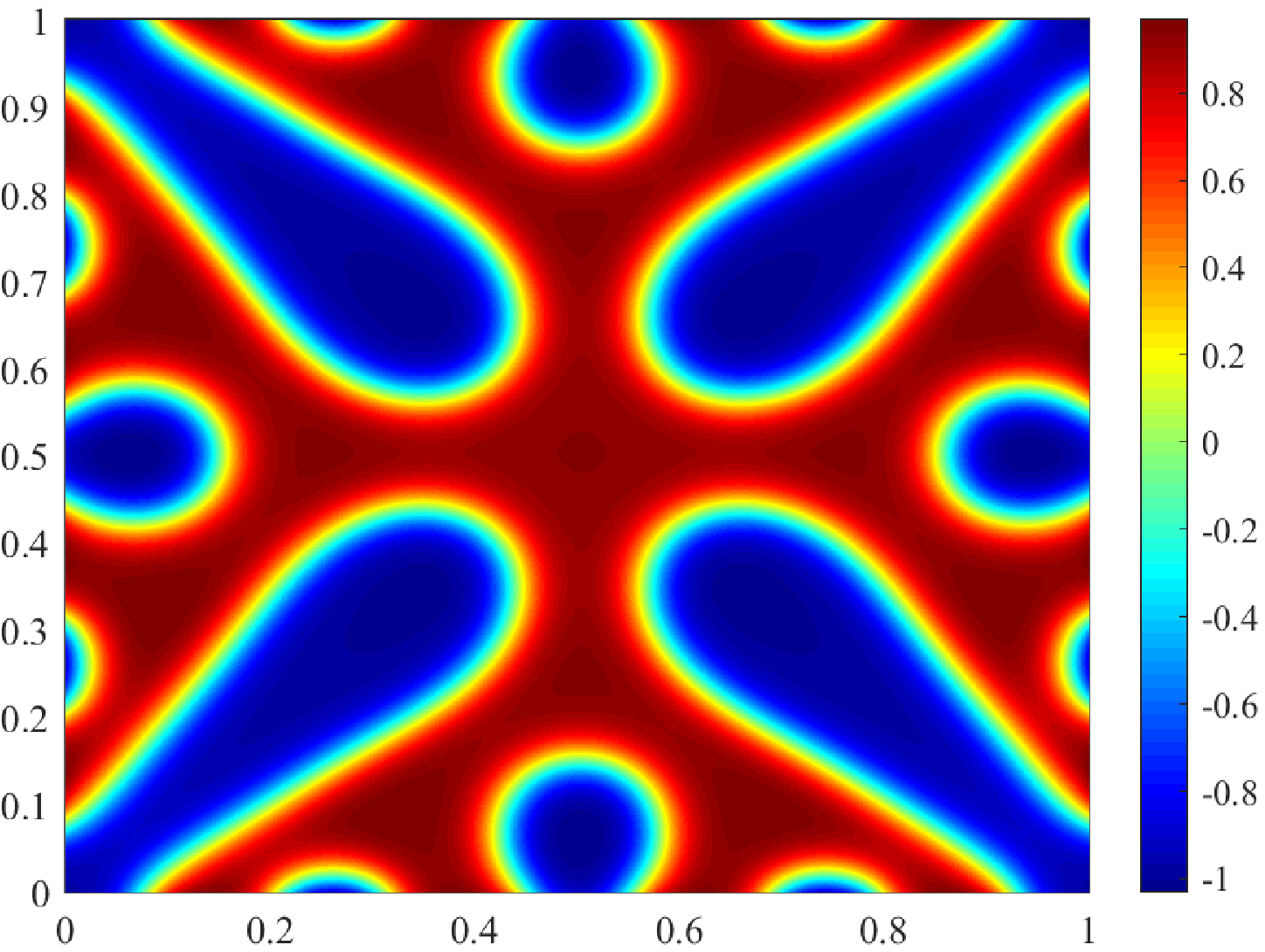}
	\includegraphics[height=0.23\textwidth,width=0.23\textwidth]{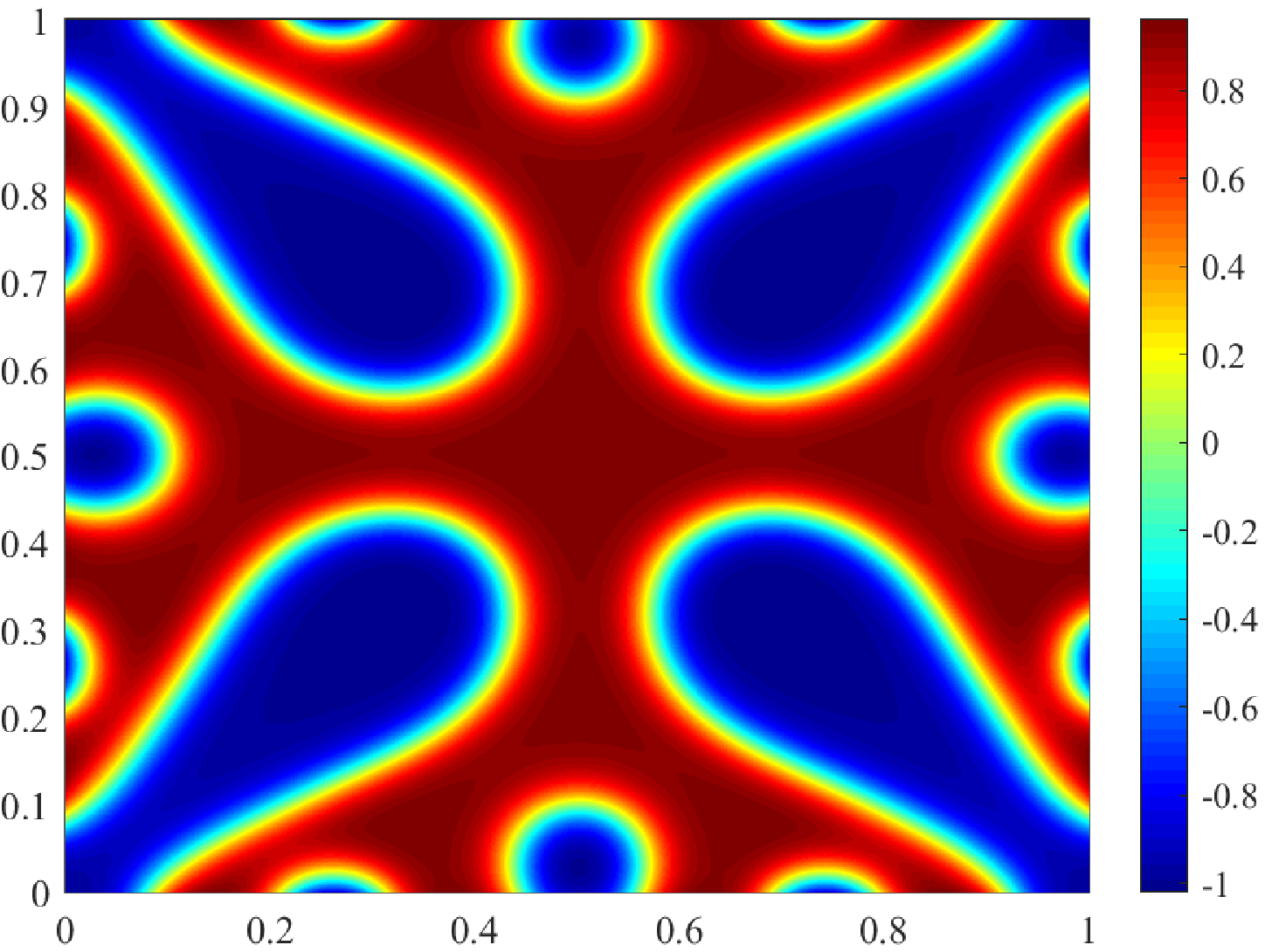}
	\includegraphics[height=0.23\textwidth,width=0.23\textwidth]{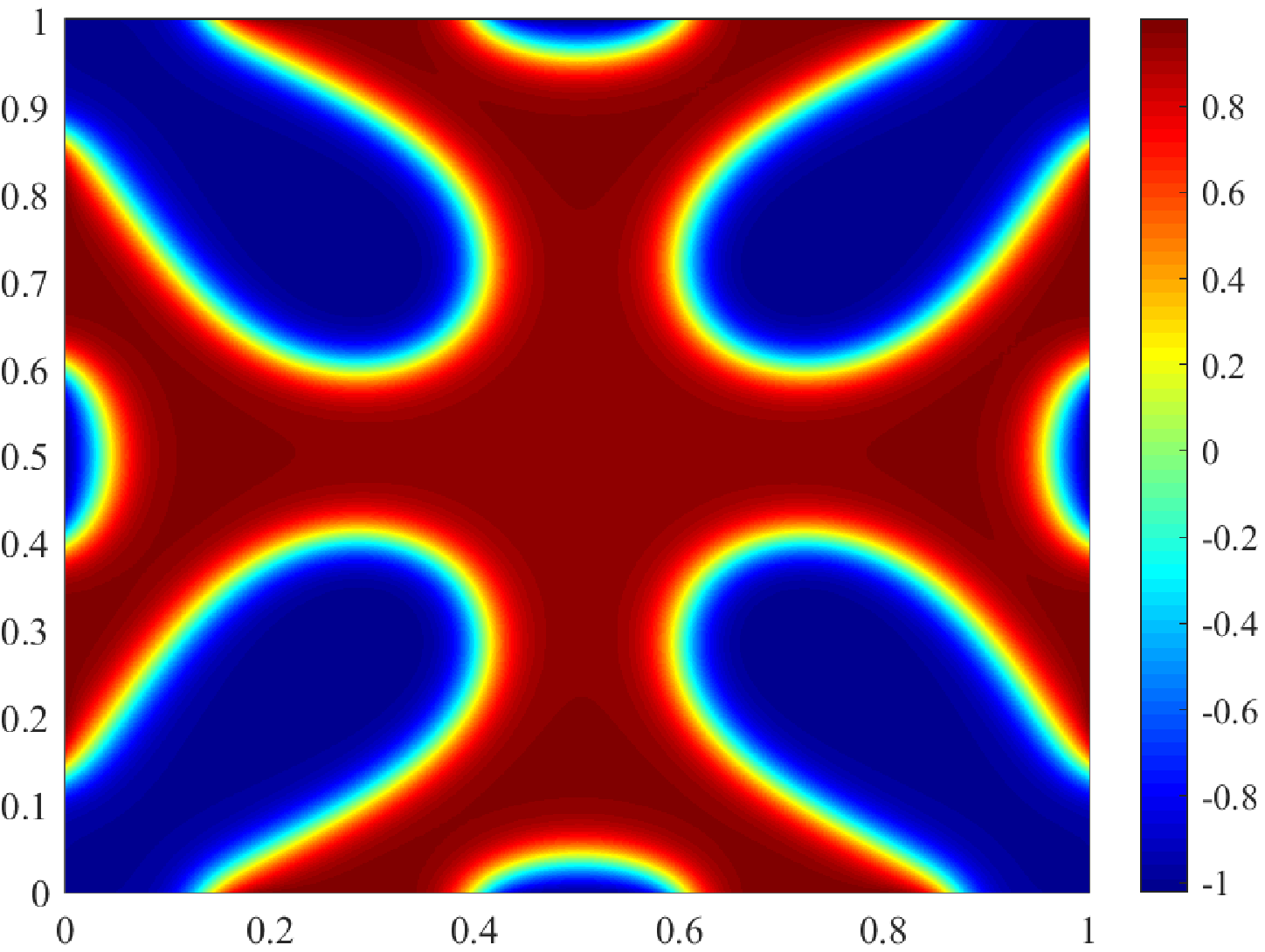}
	\includegraphics[height=0.23\textwidth,width=0.23\textwidth]{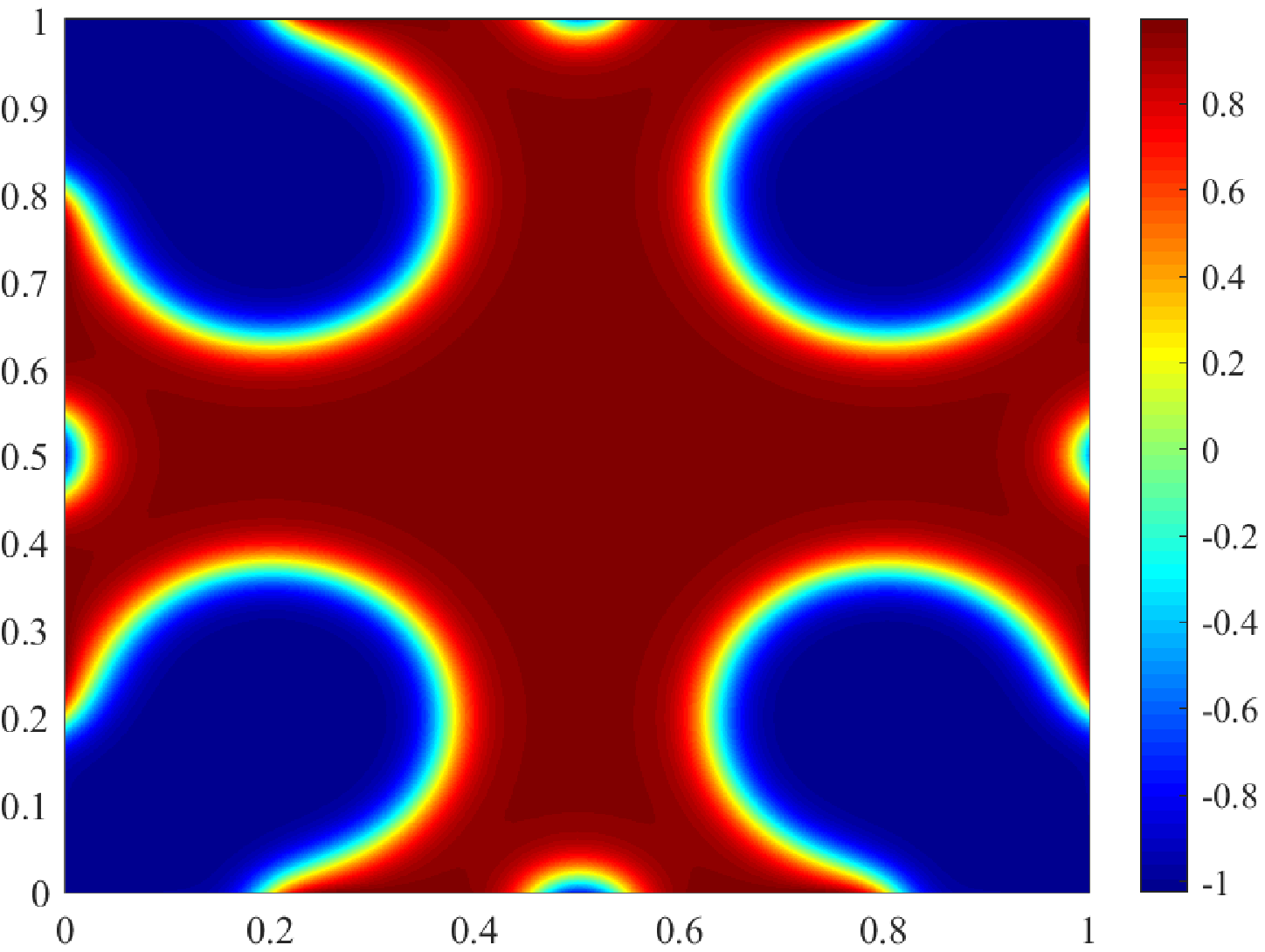}
	\includegraphics[height=0.23\textwidth,width=0.23\textwidth]{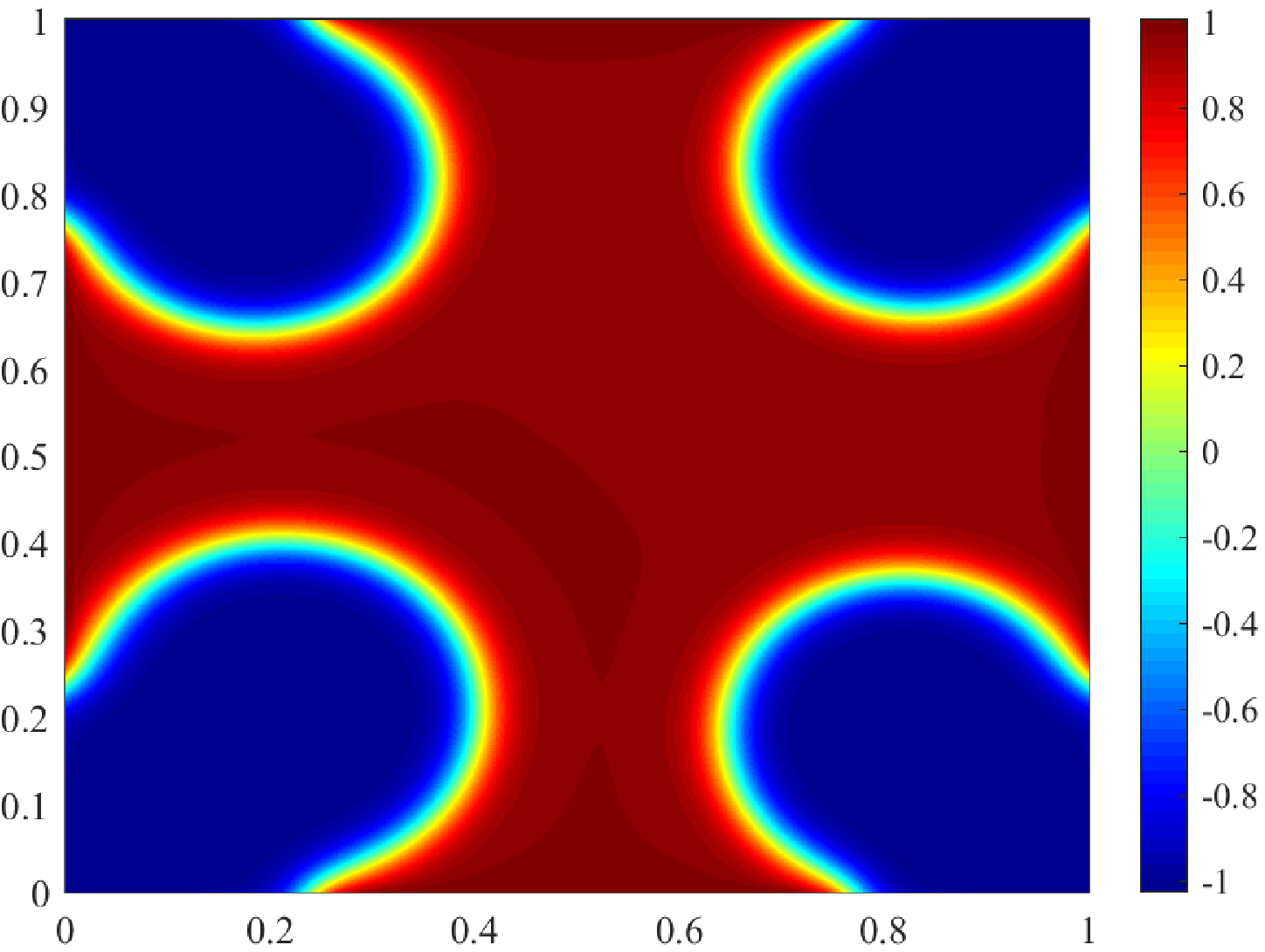}
	\caption{Snapshots of the phase variable $\phi$ at $t=0,\,0.0004,\,0.0006,\,0.0024,\,0.006,\,0.0327,\,0.2.$}		
	\label{fig7}
\end{figure}
Due to the unstable initial configuration, the two phases will be separated into different regions,
where the value of $\phi$  is close to constants $\pm 1$. The solution evolution is shown in Figure \ref{fig7}.
The red color represents the phase $\phi=1$ and blue one indicates phase $\phi=-1$. To visualize the
initial conditions, the figure of initial data is rescaled so that the red color represents $\phi=0.1$,
blue one corresponding $\phi=0$. Since the initial data is symmetric in both $x$- and $y$-direction, the
phase evolution is always developed in a symmetric way until it reaches the steady state with four patterns
arranged symmetrically. The evolution of energy and mass with time are also shown in
the Figure \ref{fig8}, which again indicates the energy stability and the conservation of mass in the region
and the boundary.
	\begin{figure}[H]
		\centering
		\includegraphics[height=0.286\textwidth,width=0.28\textwidth]{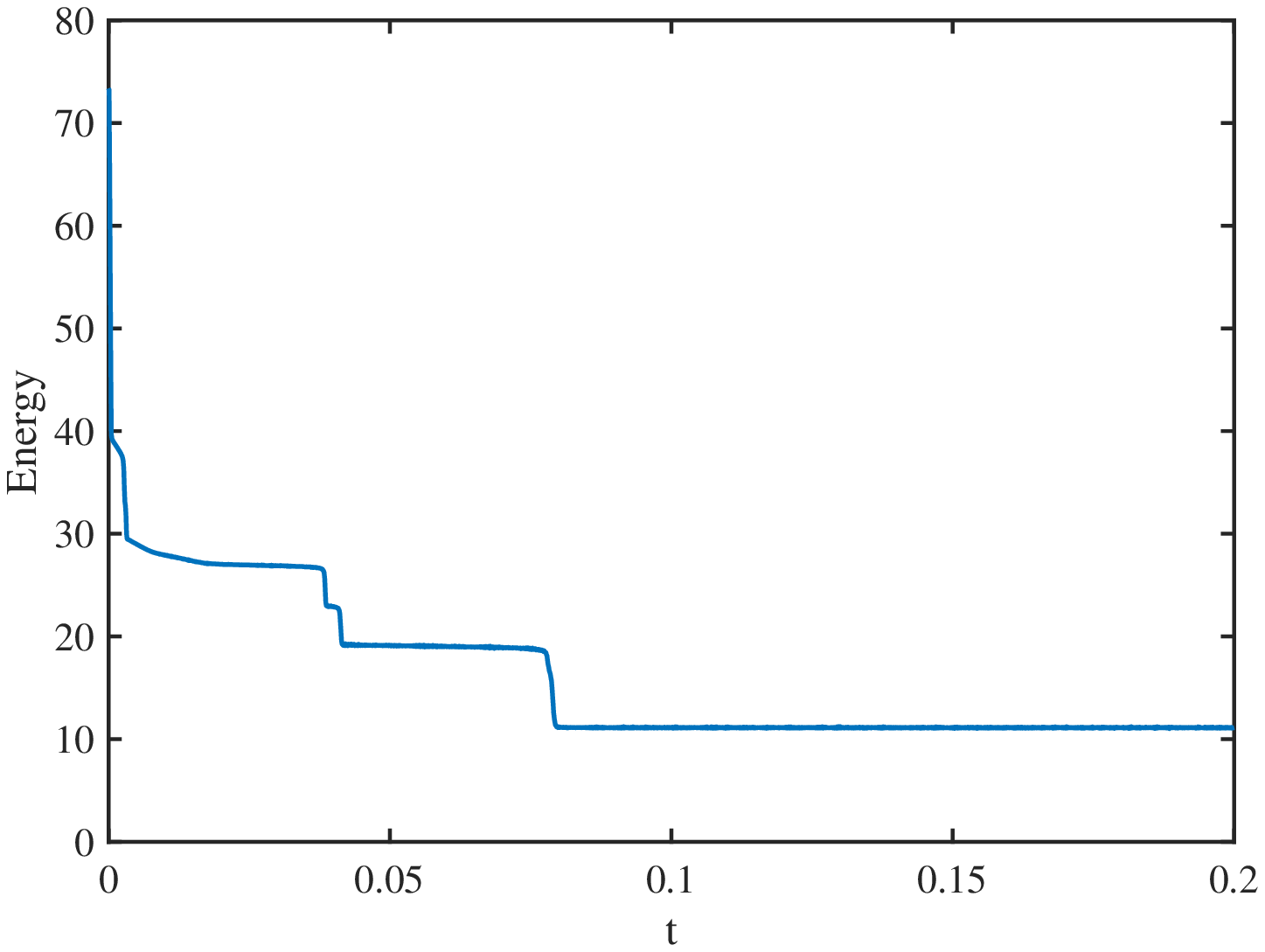}
		\includegraphics[height=0.28\textwidth,width=0.28\textwidth]{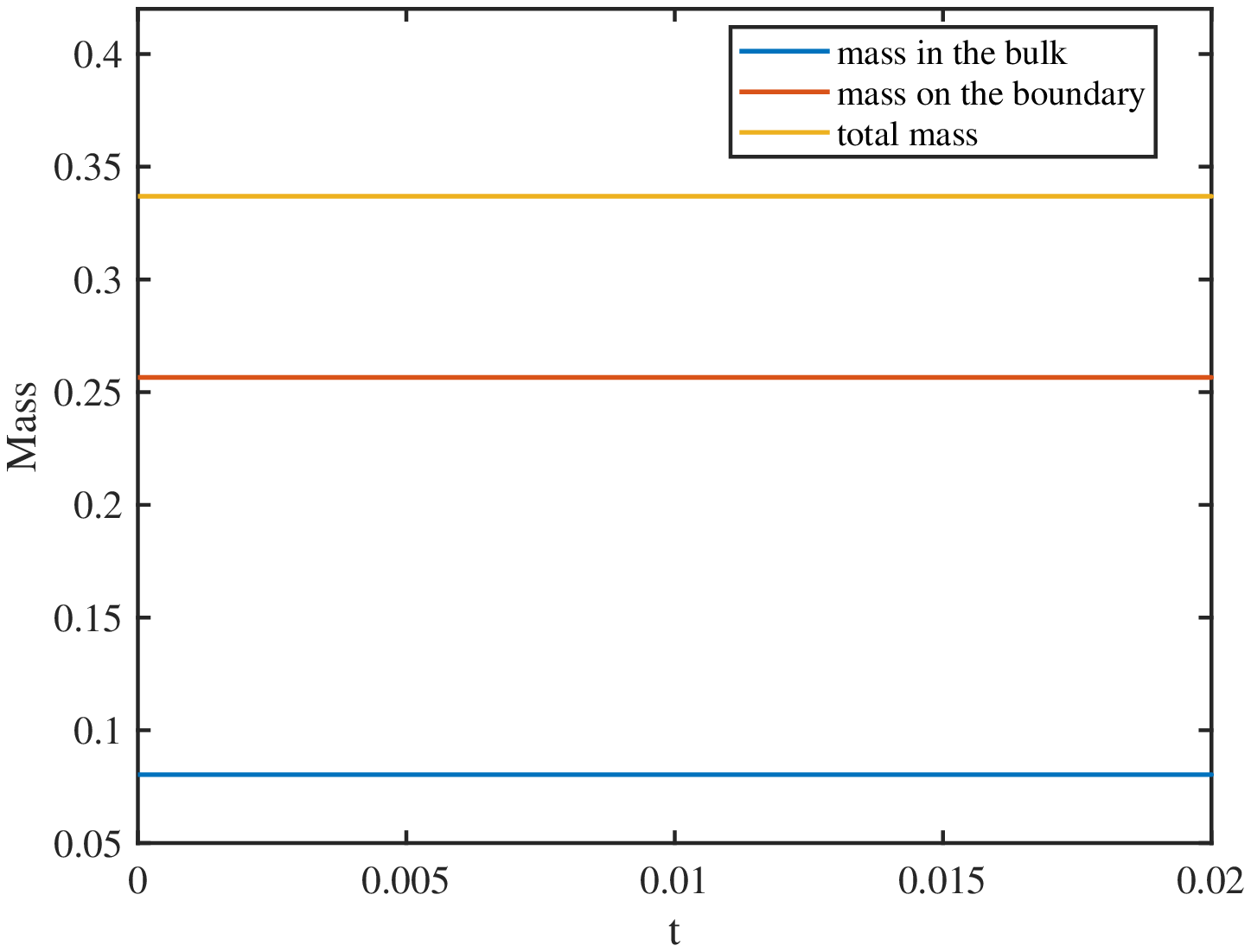}
		\caption{Energy evolution (left); Mass evolution (right) in the case of vanishing adsorption rates.}		\label{fig8}
	\end{figure}

\medskip

\noindent \textbf{Case\, 5.}\quad Here, we consider the shape deformation of a droplet. A square droplet is placed in the
area $[0,1]^2$ centered at $(0.5,0.25)$ and the length of each side
is 0.5 (as shown in Figure \ref{fig9}). The internal phase of the droplet is set to $1$ and the external phase
is set to $-1$. The forms of $F$ and $G$ are taken as regular double well potential functions (\ref{FGdwell}). The parameters
are set as $\varepsilon=\delta=0.02,\, \kappa=0.02$. The stabilized parameters are chosen as $A_1=A_2=5,\, B_1=B_2=100$.
We use the time step $\tau=2\times 10^{-4}$ and the spacial size $h=0.01$ to simulate the deformation of droplets
from $t=0$ to $t=0.5$.
\begin{figure}[h]
	\centering
	\includegraphics[height=0.28\textwidth,width=0.28\textwidth]{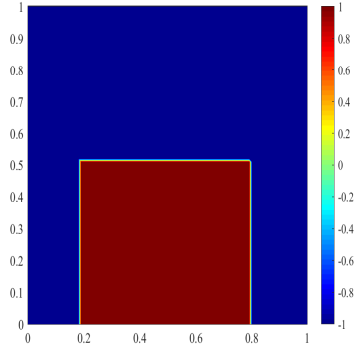}
	\caption{The initial data of the square shaped droplet.}
	\label{fig9}
\end{figure}

The deformation of droplets at time $t= 0.002,\, 0.01, \,0.02,\, 0.1,\, 0.2$ and $0.5$ are shown in Figure \ref{fig10}.
It is seen that the square droplet is  smoothed around the two up corners of the initial structure.
Then it gradually evolves into circular droplets with equal average
curvature. In addition, under the constraint of mass conservation, the contact area between the droplet and the
boundary almost keeps unchanged with time, which is consistent with the previous work \cite{knopf2020phase}. The development of
energy and mass are shown in Figure \ref{fig11}. It can be observed that the energy decreases quickly at the initial stage,
  which corresponds to the quick deformation of the square to the smoothed structure. Also we provided the energy curve from
  $t=0$ to $t=0.1$, again revealing the conservation of mass on the region and boundary respectively.

\begin{figure}[H]
	\centering
	\includegraphics[height=0.28\textwidth,width=0.28\textwidth]{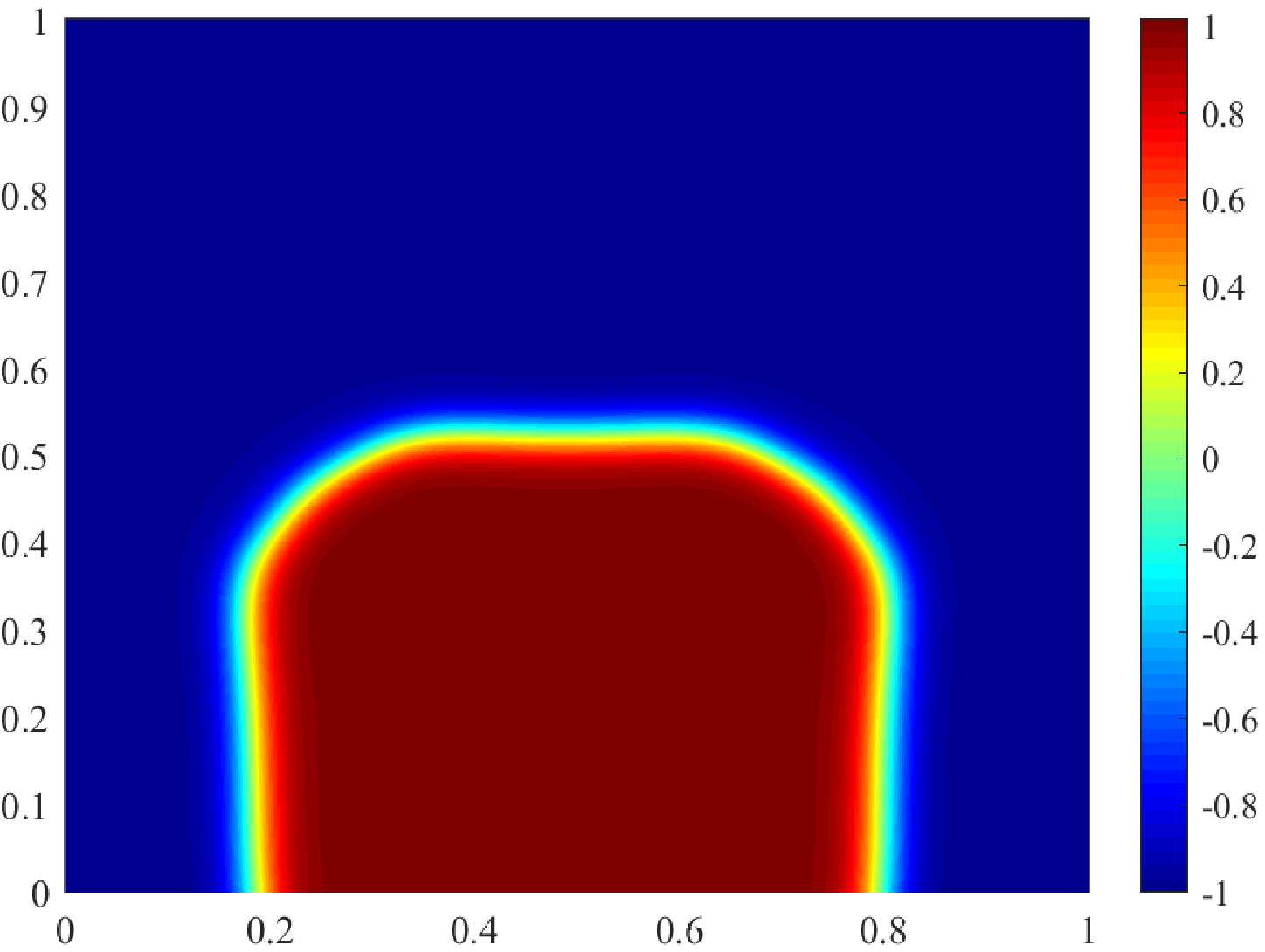}
	\includegraphics[height=0.28\textwidth,width=0.28\textwidth]{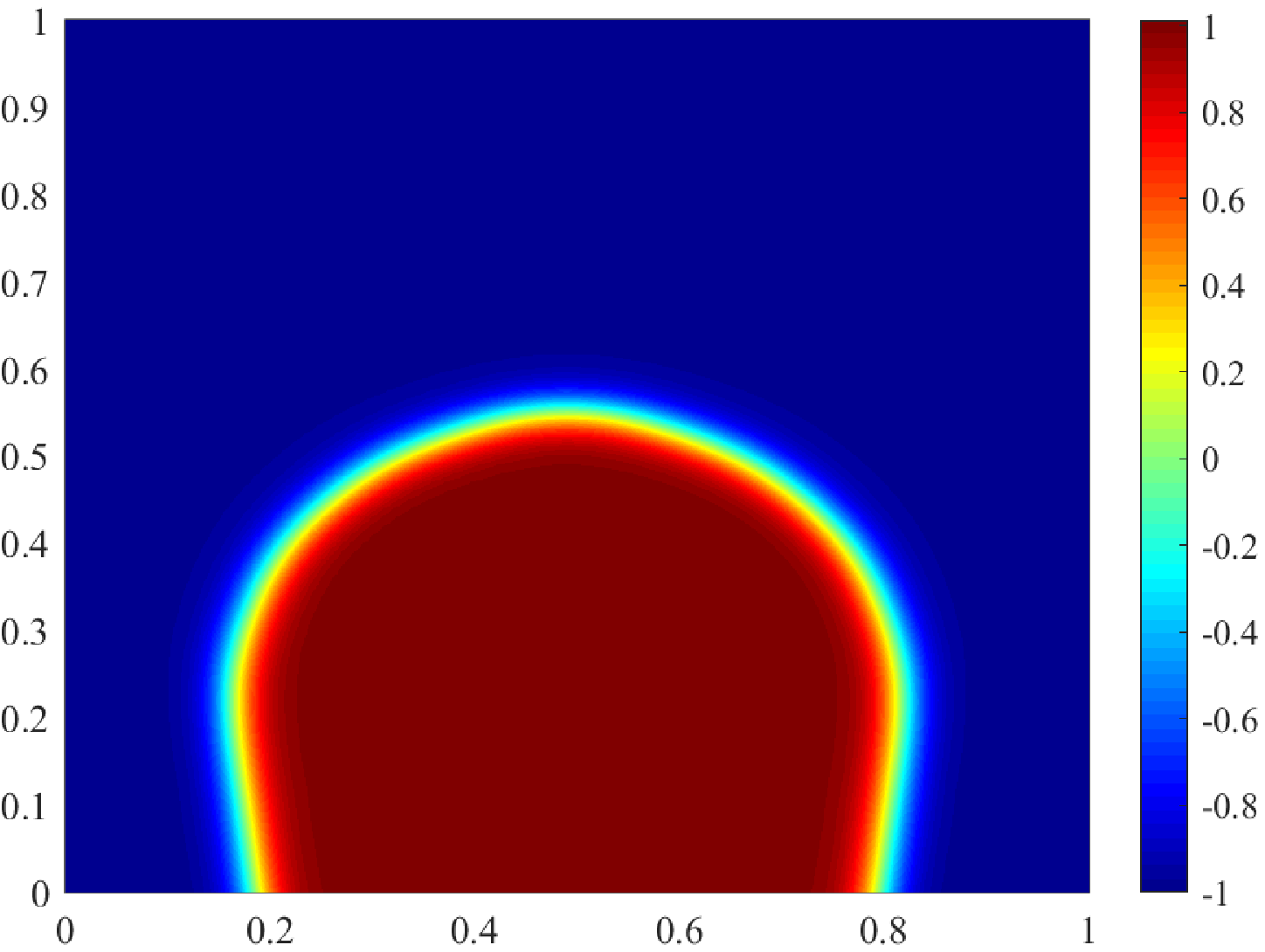}
	\includegraphics[height=0.28\textwidth,width=0.28\textwidth]{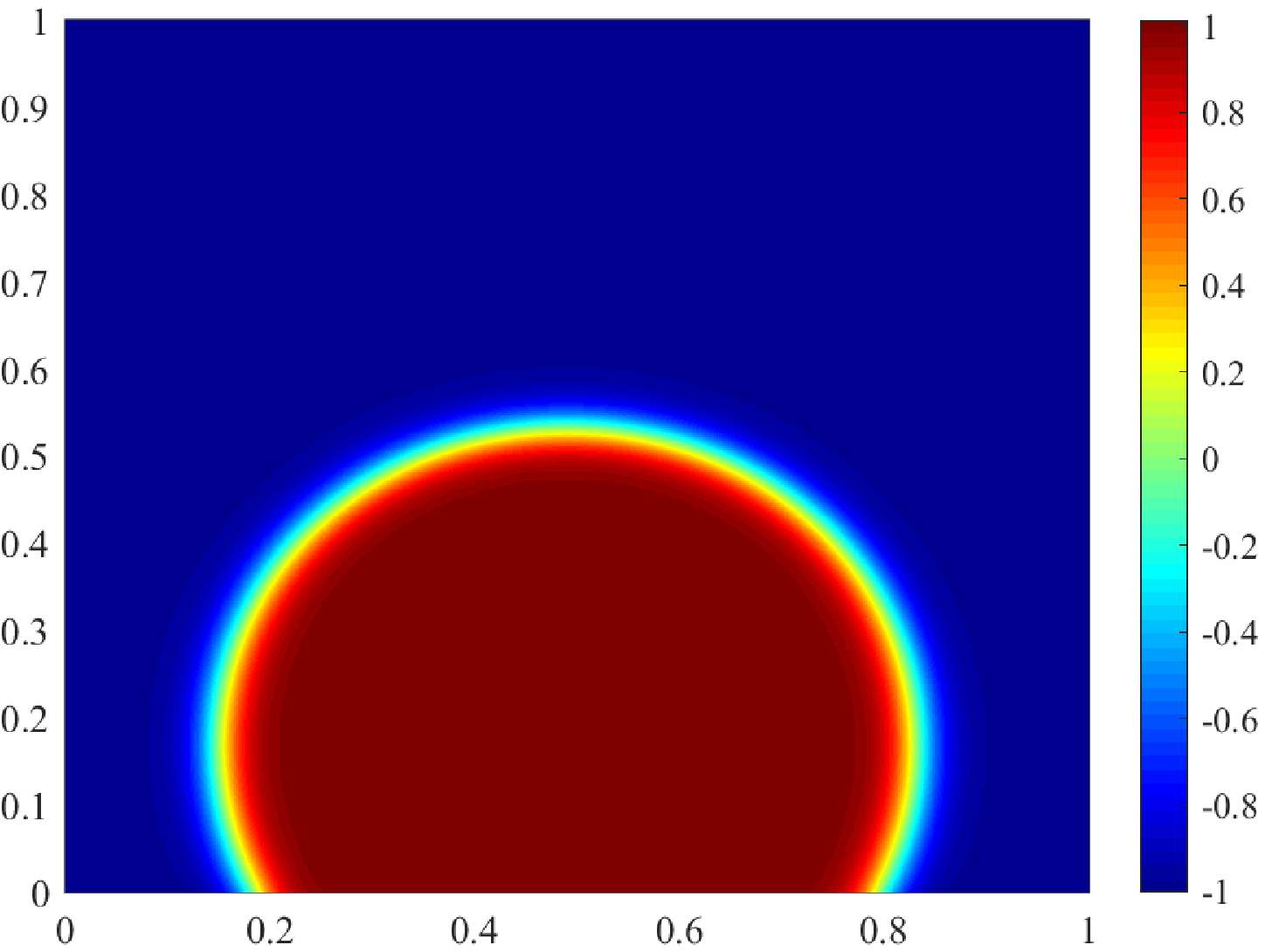}
	\includegraphics[height=0.28\textwidth,width=0.28\textwidth]{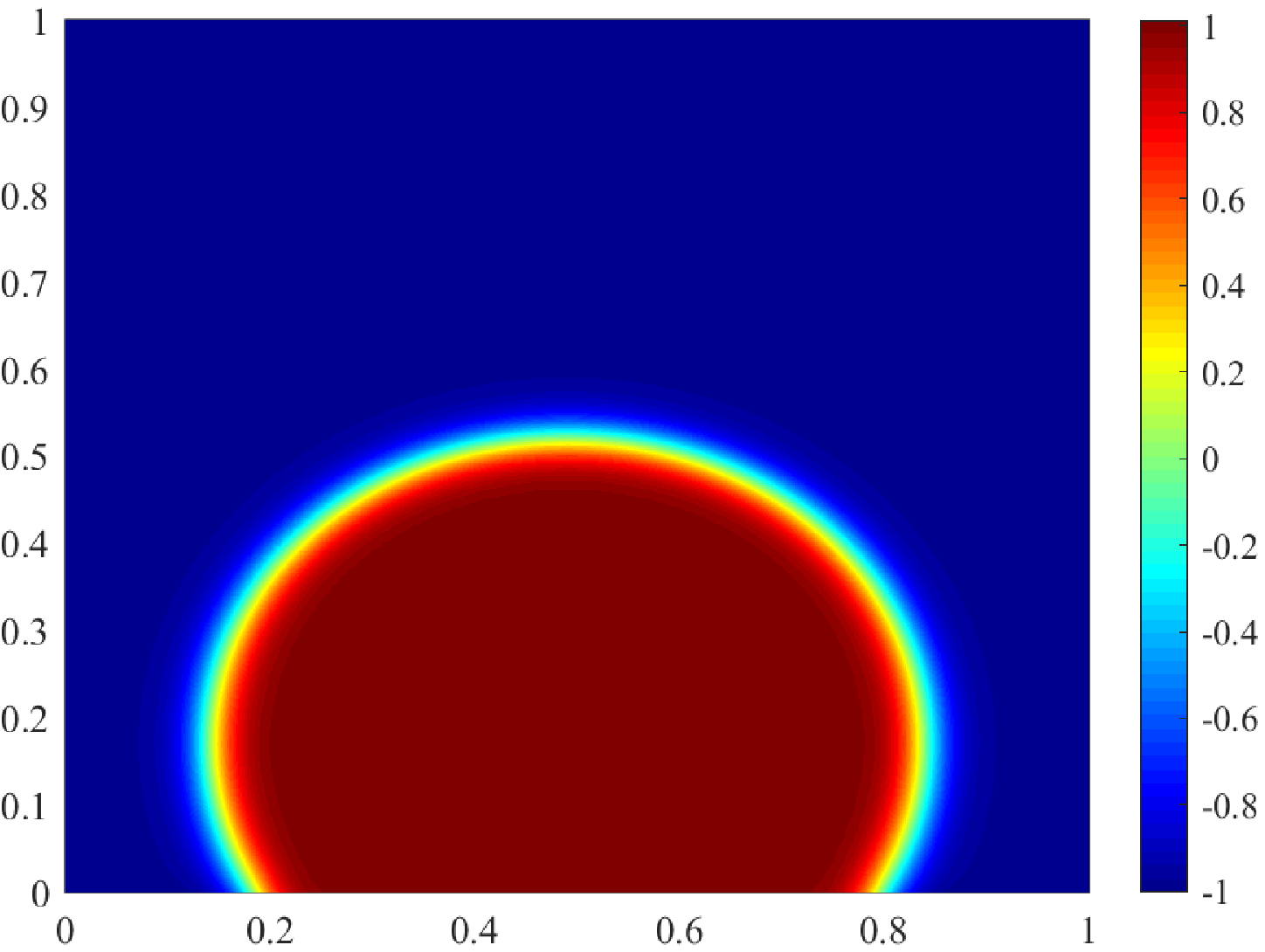}
	\includegraphics[height=0.28\textwidth,width=0.28\textwidth]{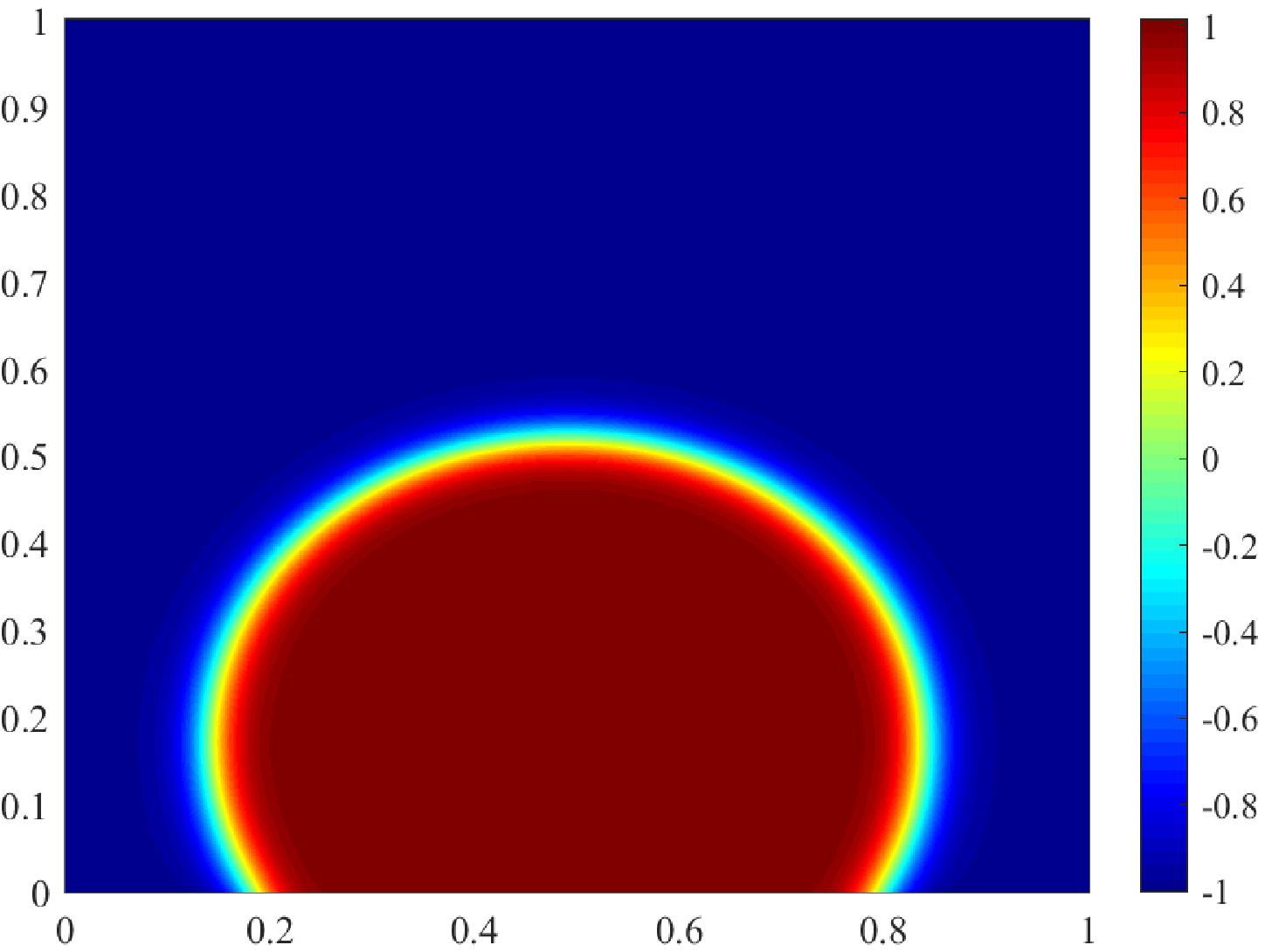}
	\includegraphics[height=0.28\textwidth,width=0.28\textwidth]{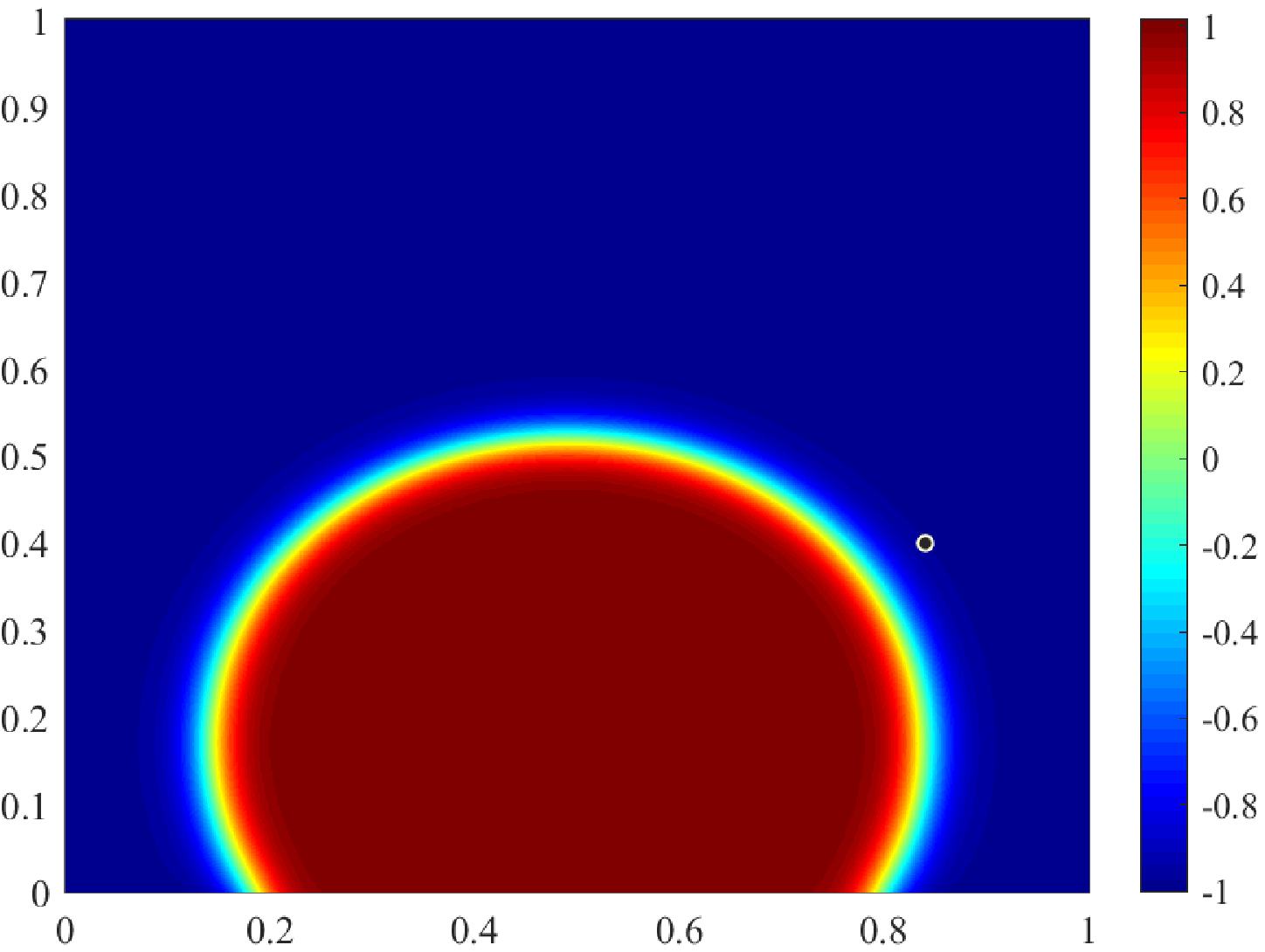}
	\caption{Snapshots of the phase variable $\phi$ at time $t= 0.002,\, 0.01, \,0.02,\, 0.1,\, 0.2,\, 0.5$ with
double well potential functions.}
	\label{fig10}
\end{figure}

	\begin{figure}[H]
		\centering
		\includegraphics[height=0.282\textwidth,width=0.28\textwidth]{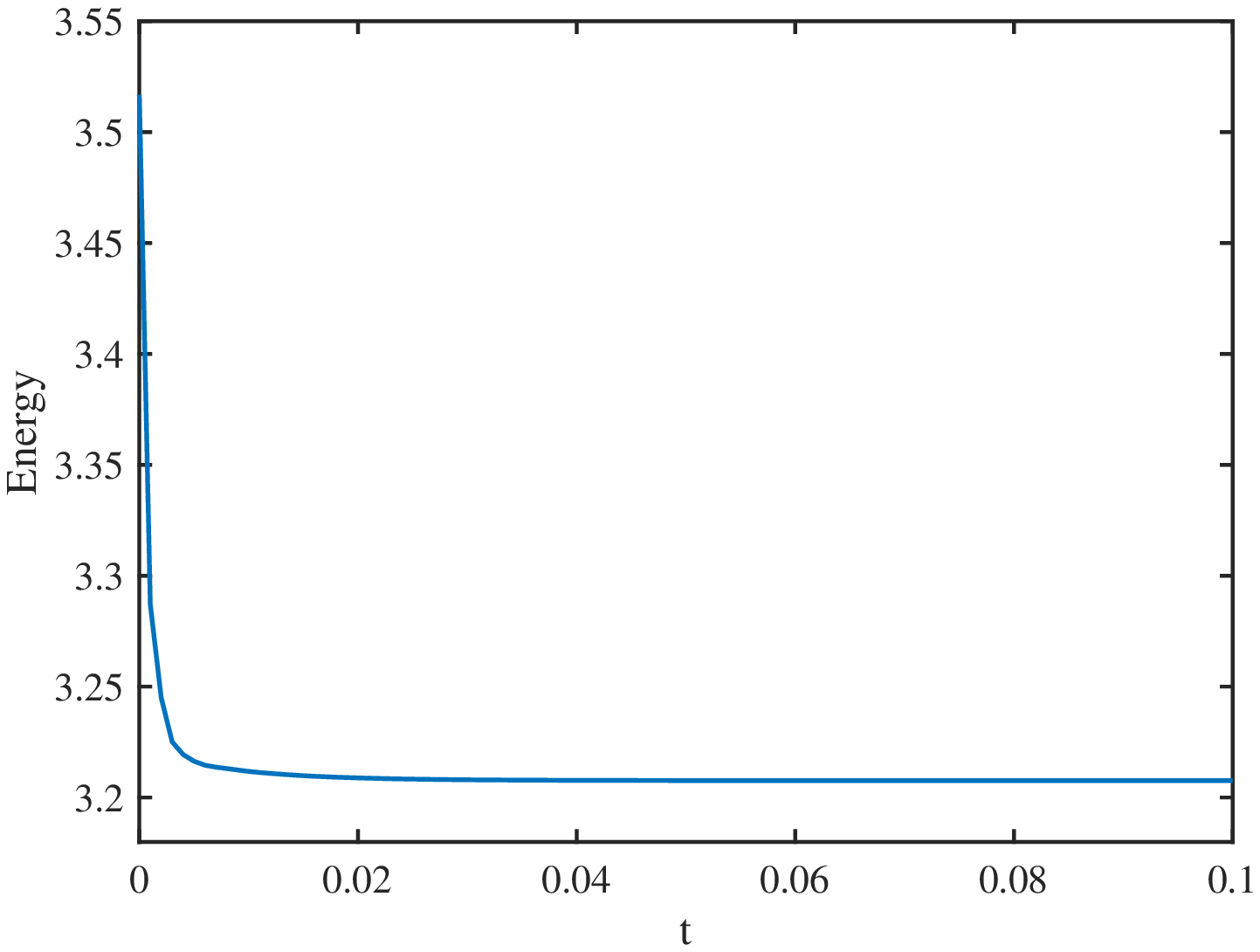}
		\includegraphics[height=0.28\textwidth,width=0.28\textwidth]{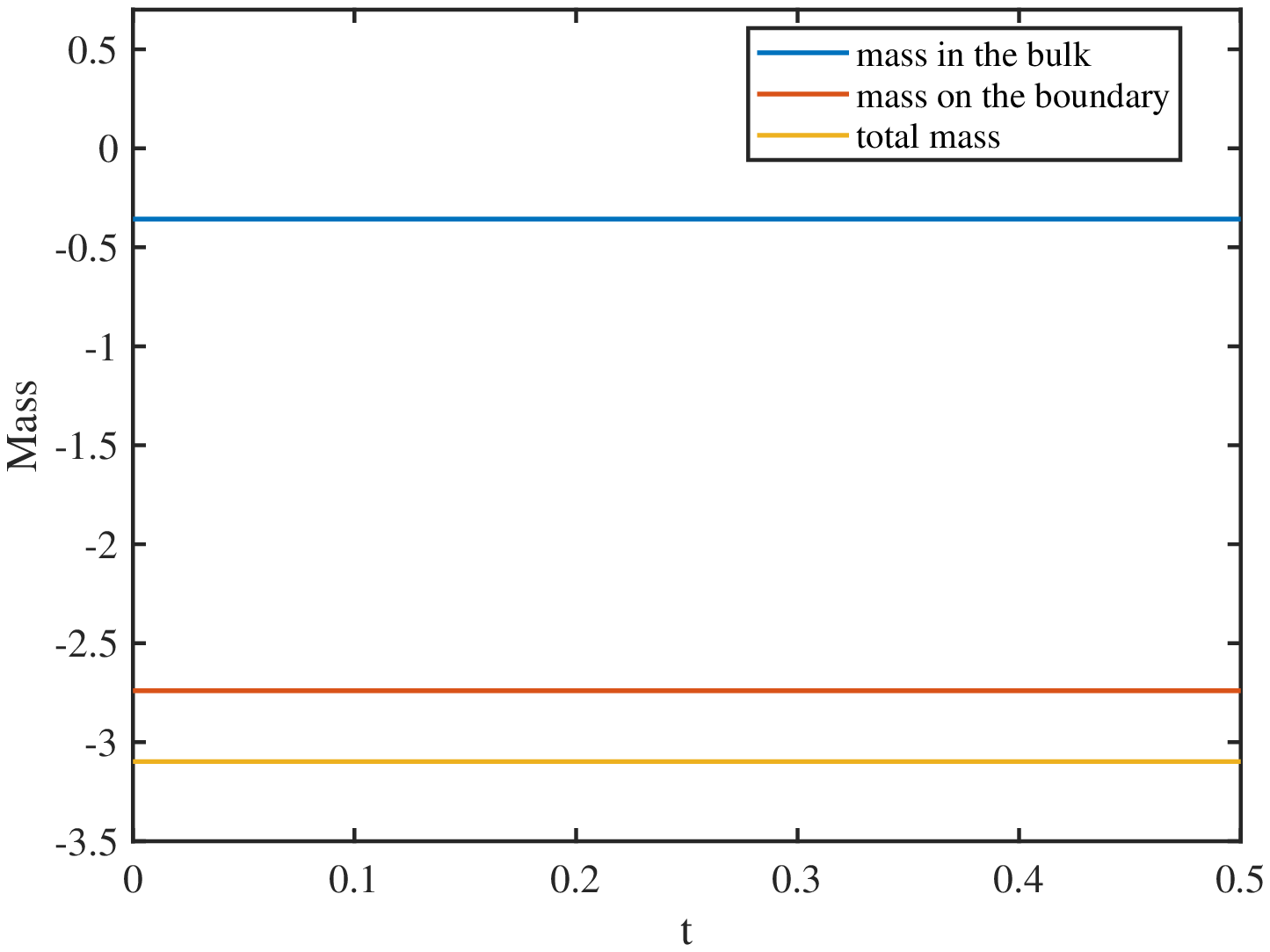}
		\caption{Energy evolution (left); Mass evolution (right) with the initial data of the square shaped droplet.}		
		\label{fig11}
	\end{figure}

\subsection{Cases with different potential functions}
In the previous numerical experiments, the surface potential function $G$ takes the form of polynomial. Here, we
consider different forms.

\medskip

\noindent \textbf{Case 1.}\quad We consider the typical moving contact line problem
\begin{eqnarray}
\label{G2}
 G(\phi)=\frac{\gamma}{2}\cos(\theta_s)\sin(\frac{\pi}{2}\phi),
\end{eqnarray}
where $\gamma=\frac{2\sqrt{2}}{2}$, $\theta_s$ is the static contact angle ($\cos\theta_s=\pm \frac12$ below).
$\tau=10^{-5},\,h=0.01$ and other parameters are the same as those in the previous samples.

 We show the energy curve and mass curve for $0\leq t\leq 0.01$ in Figure \ref{fig12}. It is found that the case
   $\cos\theta_s=-\frac{1}{2}$ takes longer time to reach the steady state than the case of $\cos\theta_s=\frac{1}{2}$.
 For both cases, the mass in the bulk and on the boundary keep unchanged throughout the computation.
In Figures \ref{fig13} and and \ref{fig14}, we present the phase contours for at  $t=0.0003,\,0.0005,\,0.001,\,0.002,\,0.008$ and $0.01$
corresponding to $\cos\theta_s=\frac{1}{2}$ and $\cos\theta_s=-\frac{1}{2}$, respectively.
Driven by the surface potential function (\ref{G2}), the square droplet also tends to change into a circle
with time, see Figure \ref{fig13} and \ref{fig14}. The same phenomena occurs in the Case 5 in Section \ref{SubSecIC}. However,
it is noted that the contact area between the droplet and
the boundary will change, which is different from the case of double well potential (\ref{FGdwell}).
Therefore, due to the mass conservation on the region and boundary respectively, the value of $\phi$ and $\psi$ are not
limited to the interval $[-1,1] $.

\begin{figure}[H]
	\centering
	\includegraphics[height=0.28\textwidth,width=0.28\textwidth]{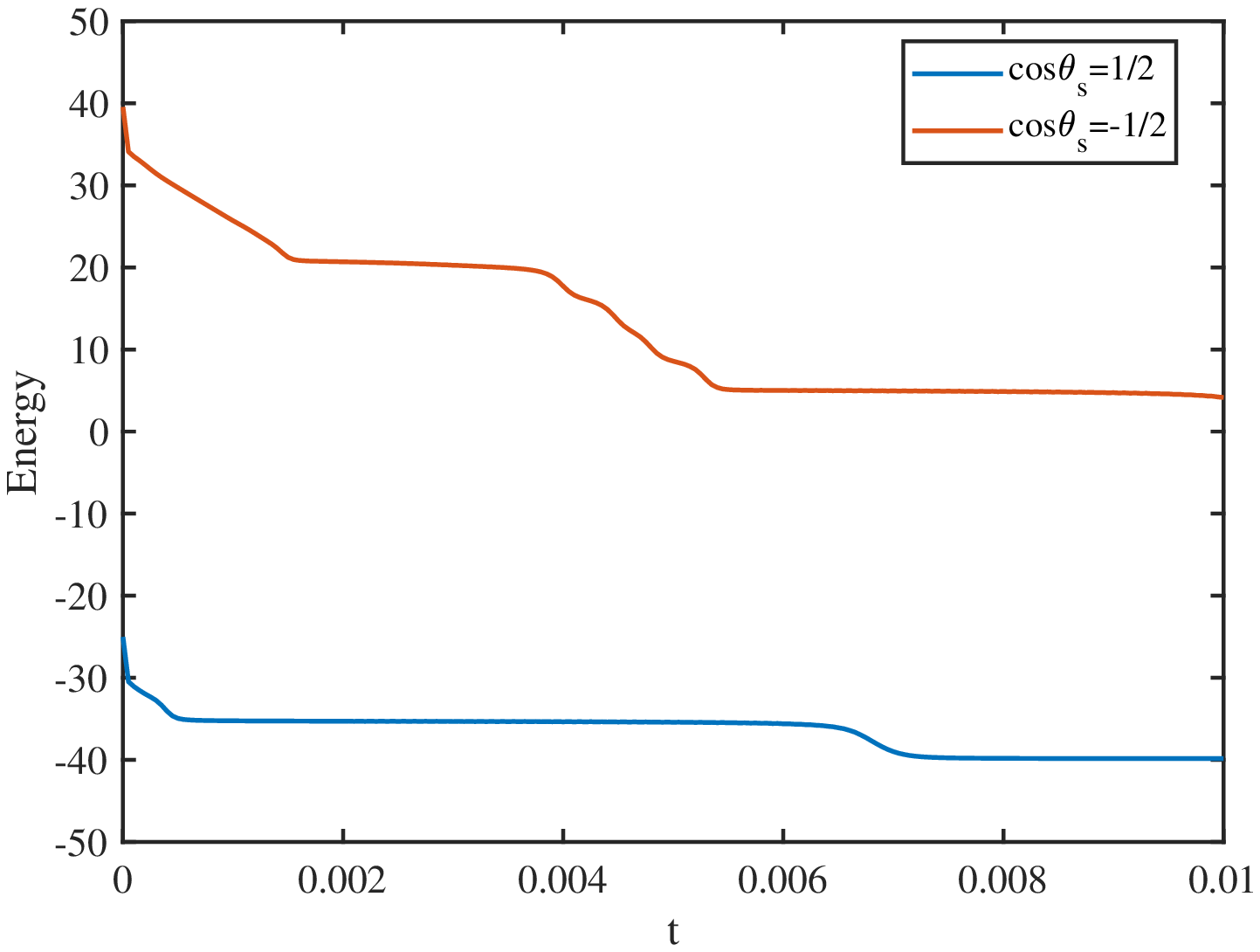}
	\includegraphics[height=0.28\textwidth,width=0.28\textwidth]{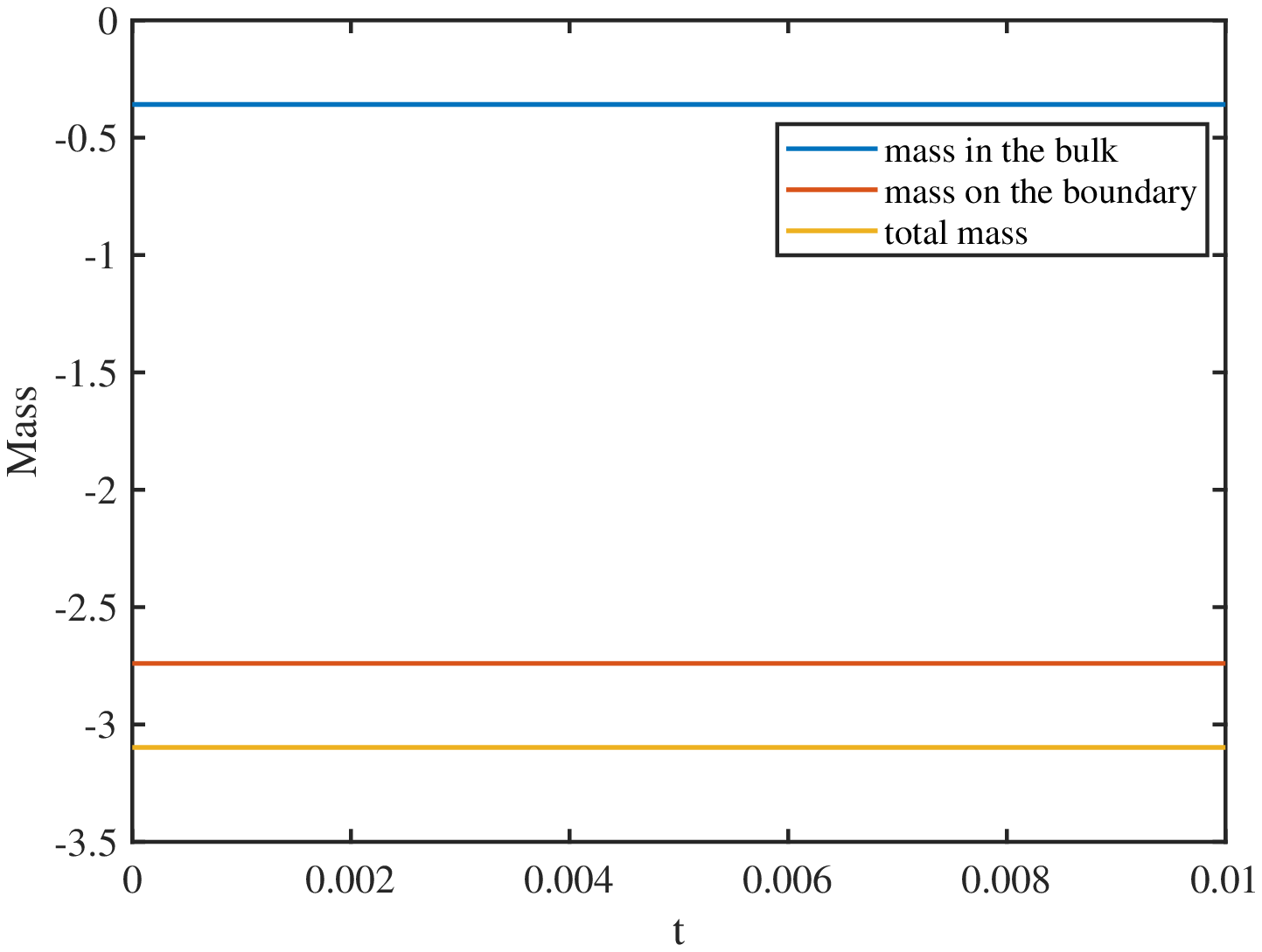}
	\includegraphics[height=0.28\textwidth,width=0.28\textwidth]{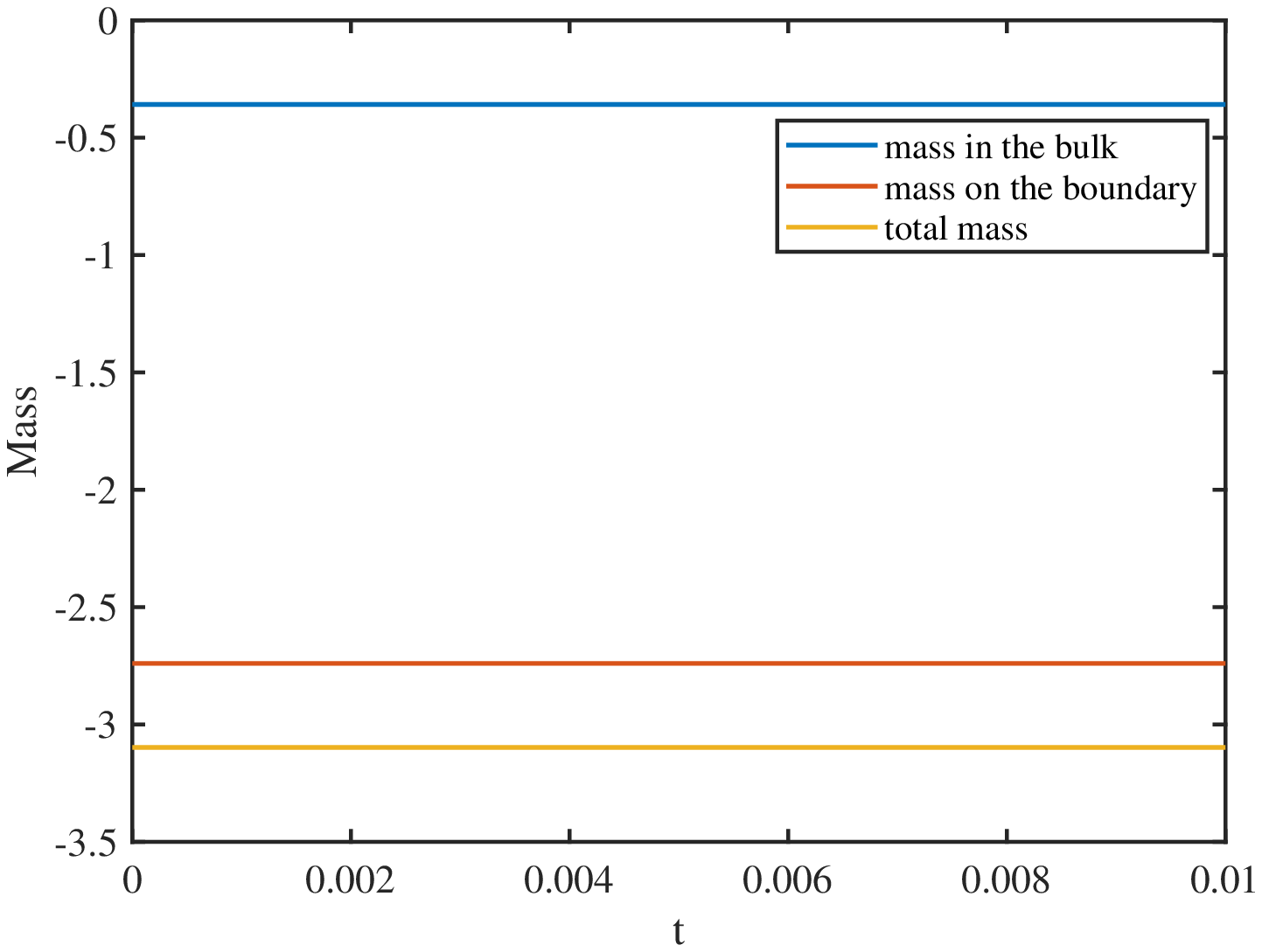}
	\caption{Energy evolution of Liu-Wu model with surface potential energy (\ref{G2}) (left);
Mass evolution of Liu-Wu model when the $\cos\theta_s=\frac12$ (middle) and $\cos\theta_s=-\frac12$ (right).}		
	\label{fig12}
\end{figure}

	\begin{figure}[H]
		\centering
		\includegraphics[height=0.28\textwidth,width=0.28\textwidth]{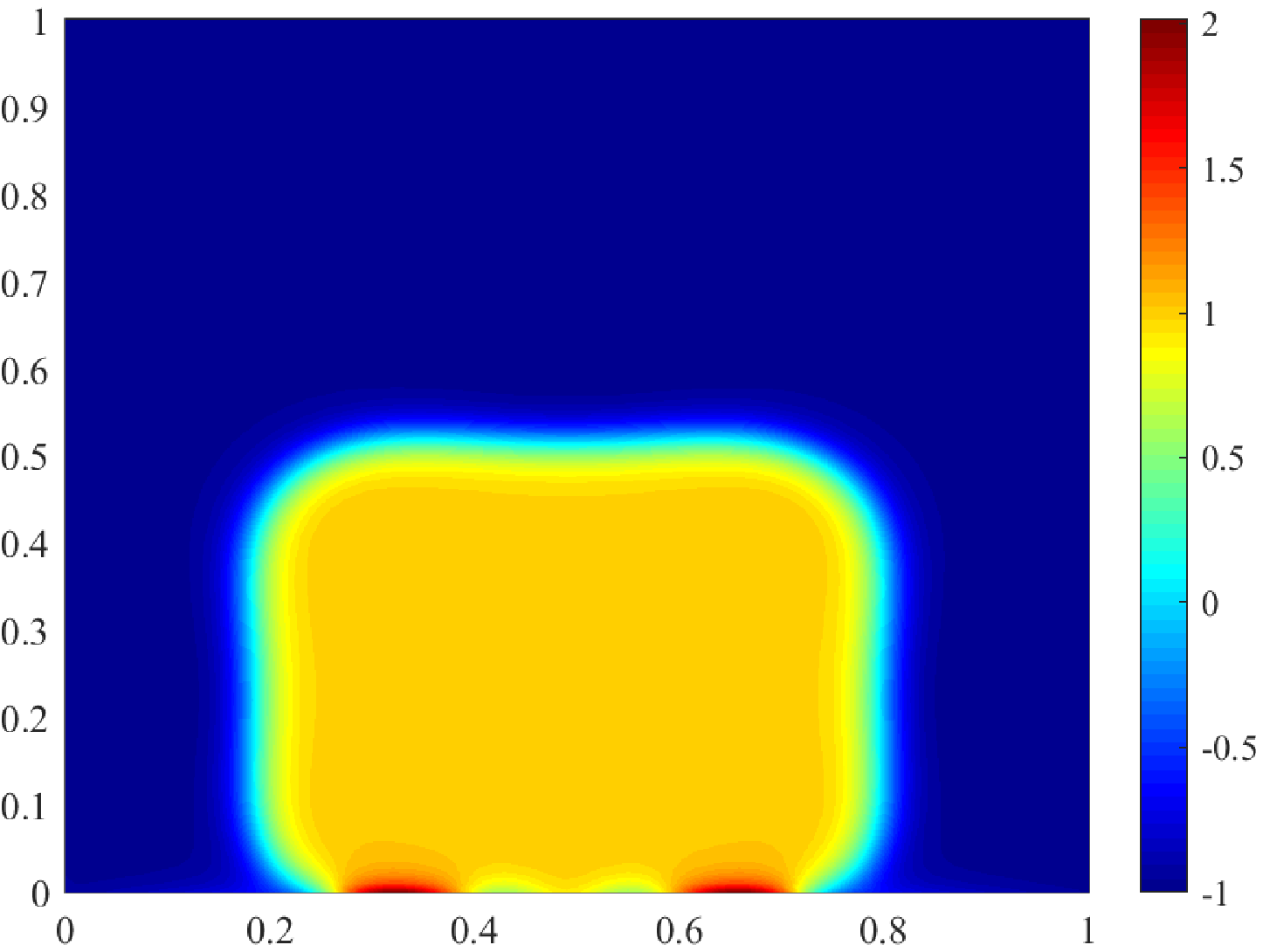}
		\includegraphics[height=0.28\textwidth,width=0.28\textwidth]{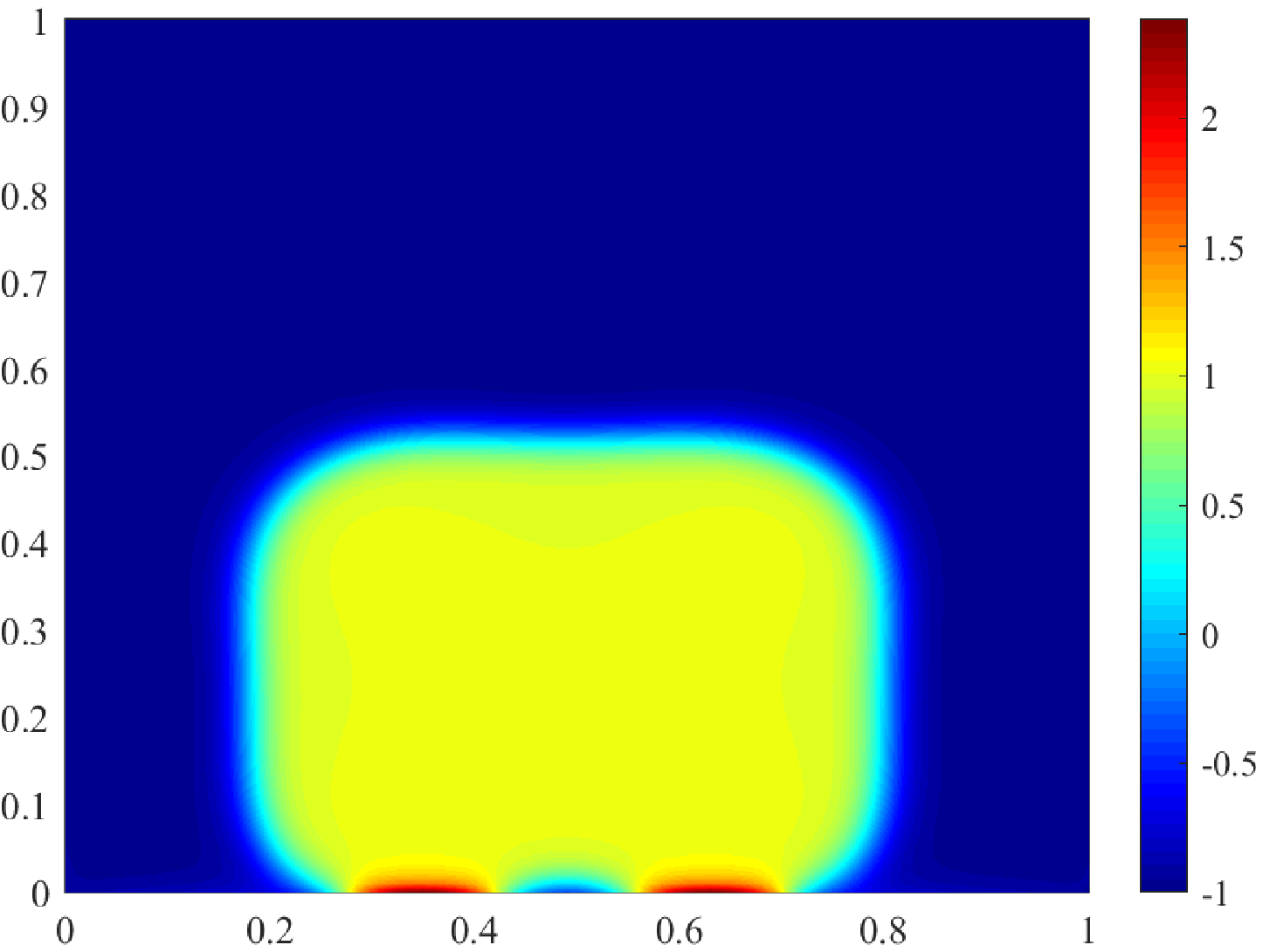}
		\includegraphics[height=0.28\textwidth,width=0.28\textwidth]{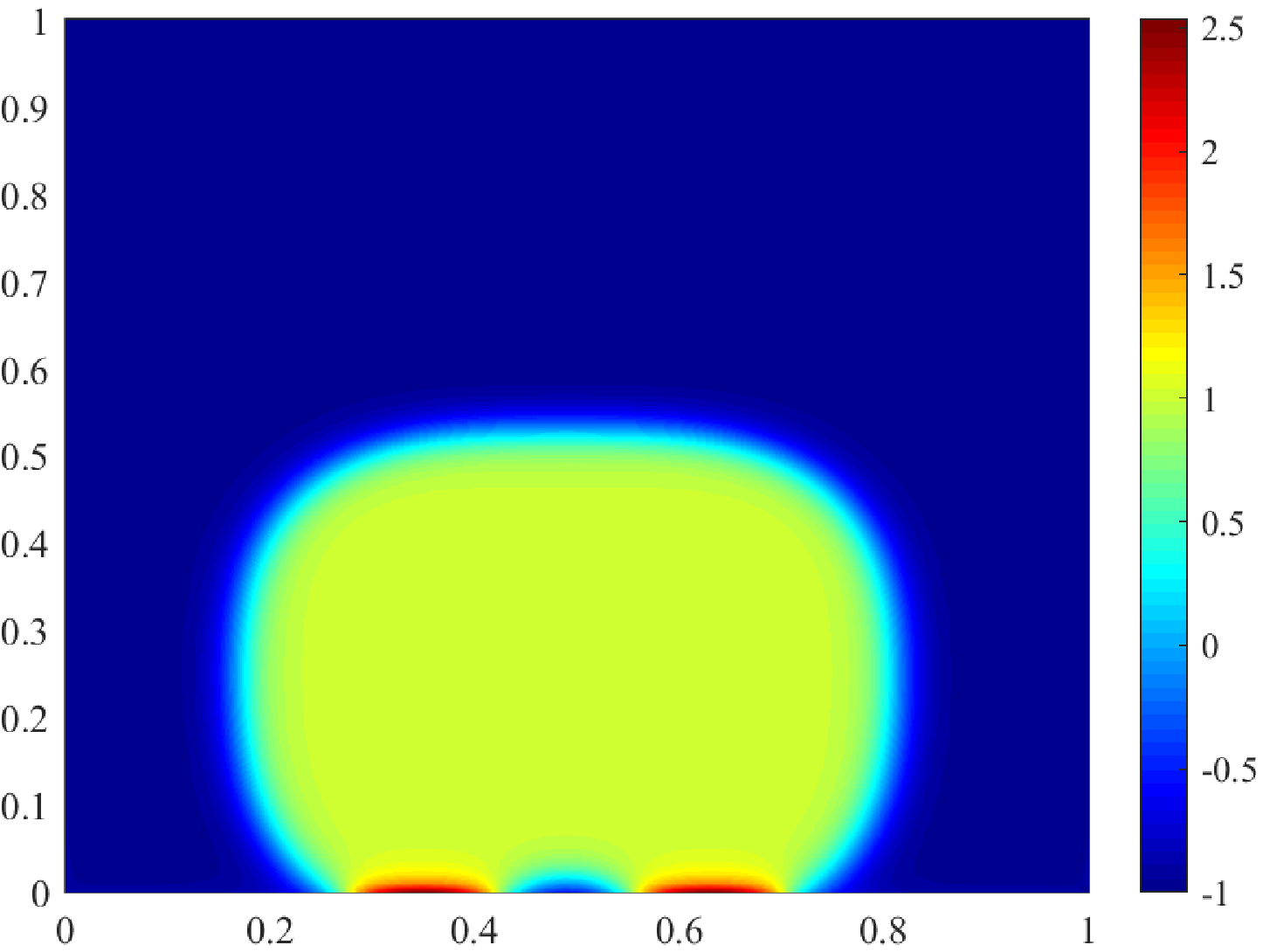}
		\includegraphics[height=0.28\textwidth,width=0.28\textwidth]{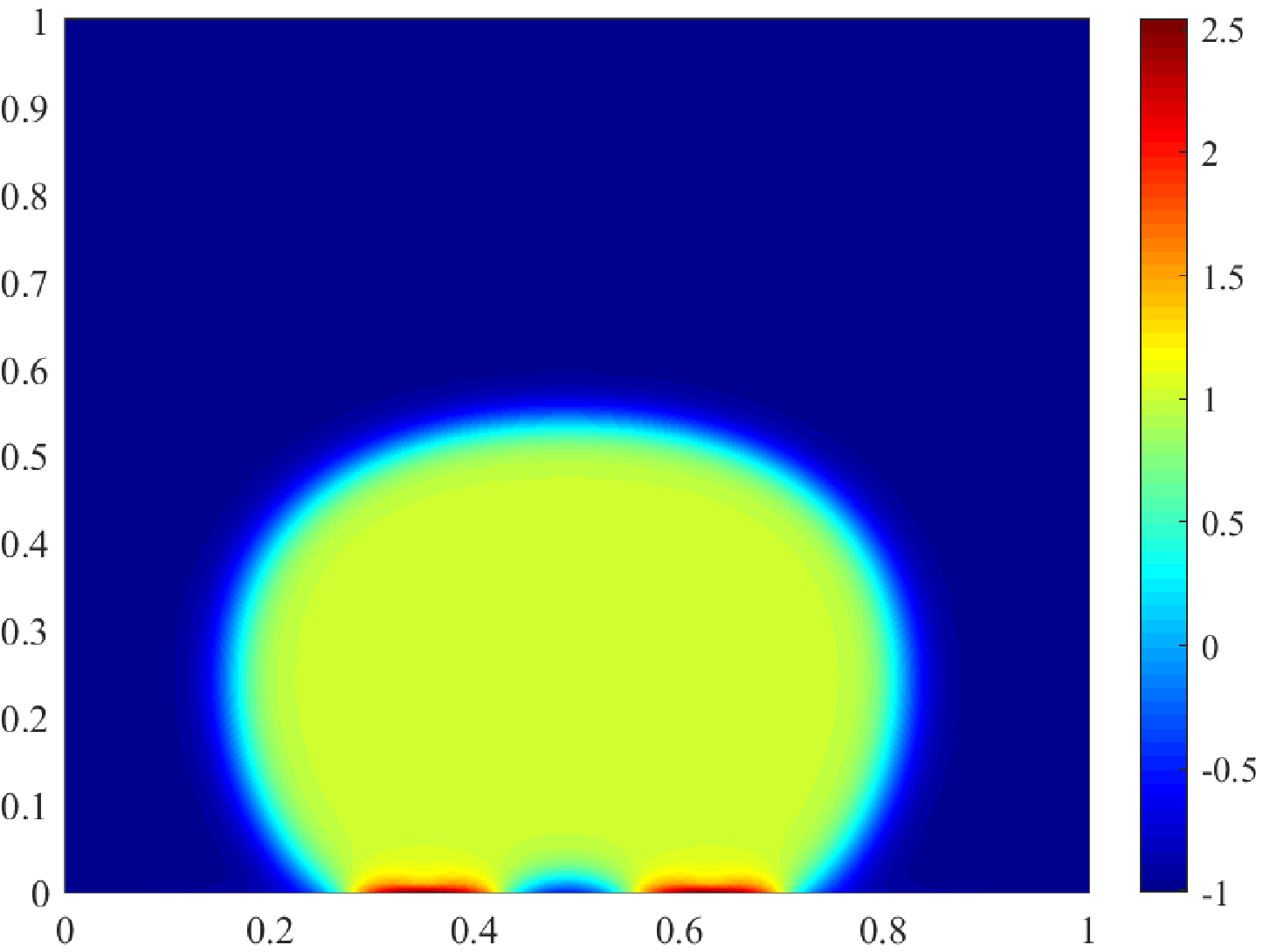}
		\includegraphics[height=0.28\textwidth,width=0.28\textwidth]{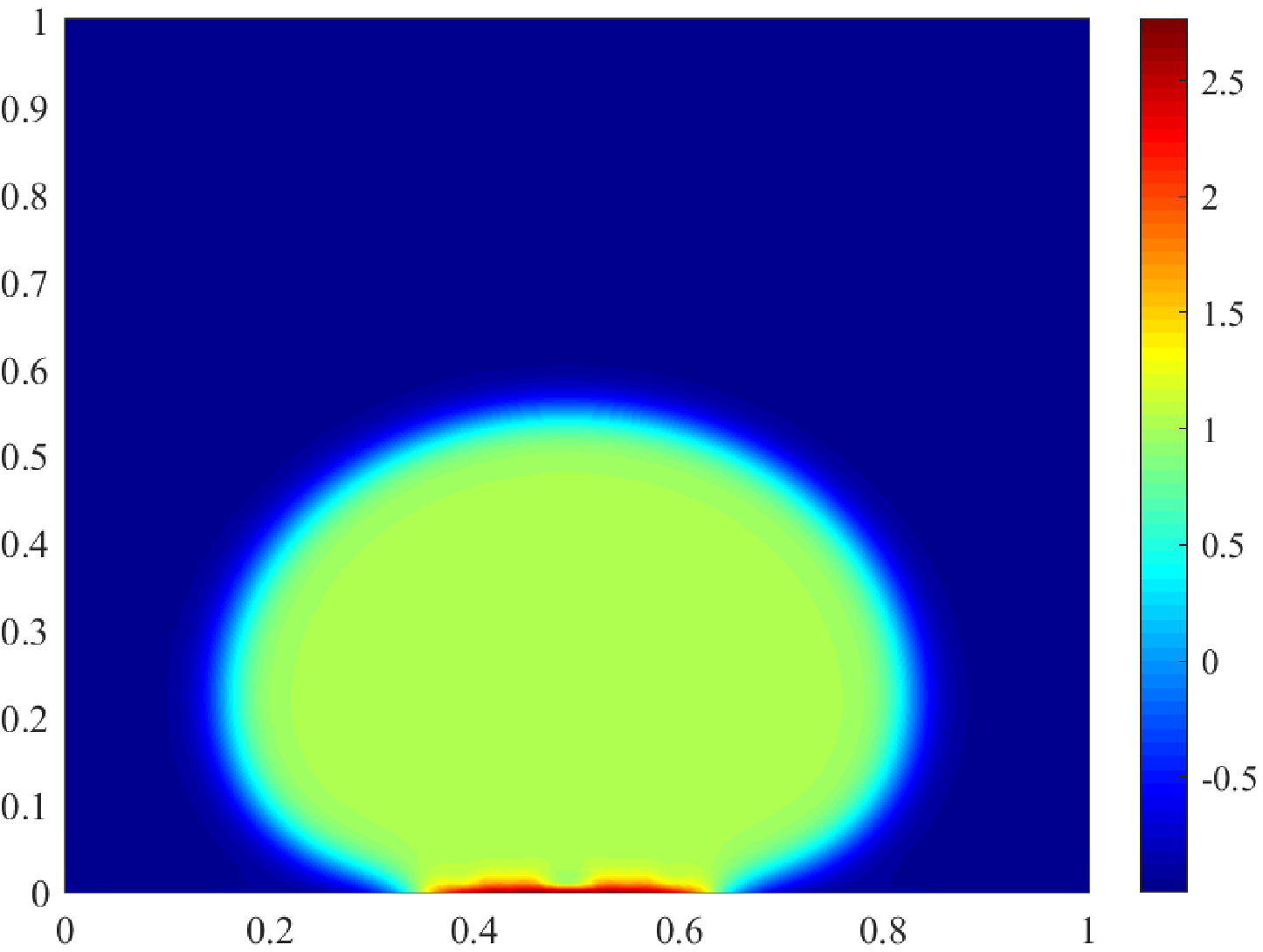}
		\includegraphics[height=0.28\textwidth,width=0.28\textwidth]{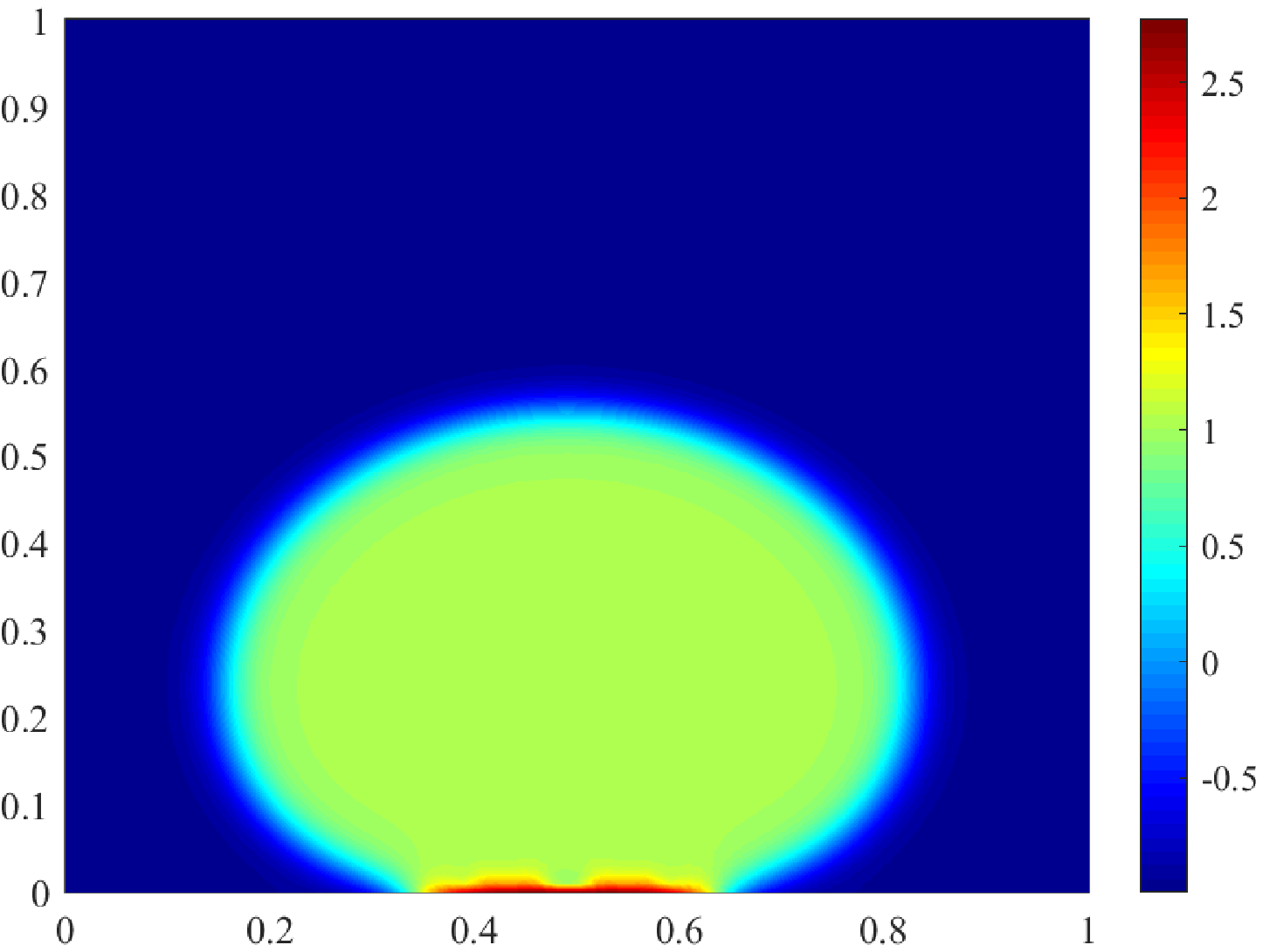}
		\caption{Snapshots of the phase variable $\phi$ at time $t=0.0003,\,0.0005,\,0.001,\,0.002,\,0.008,\,0.01$,\, $(\cos\theta_s=\frac12)$.}
		\label{fig13}
	\end{figure}

	\begin{figure}[H]
		\centering
		\includegraphics[height=0.28\textwidth,width=0.28\textwidth]{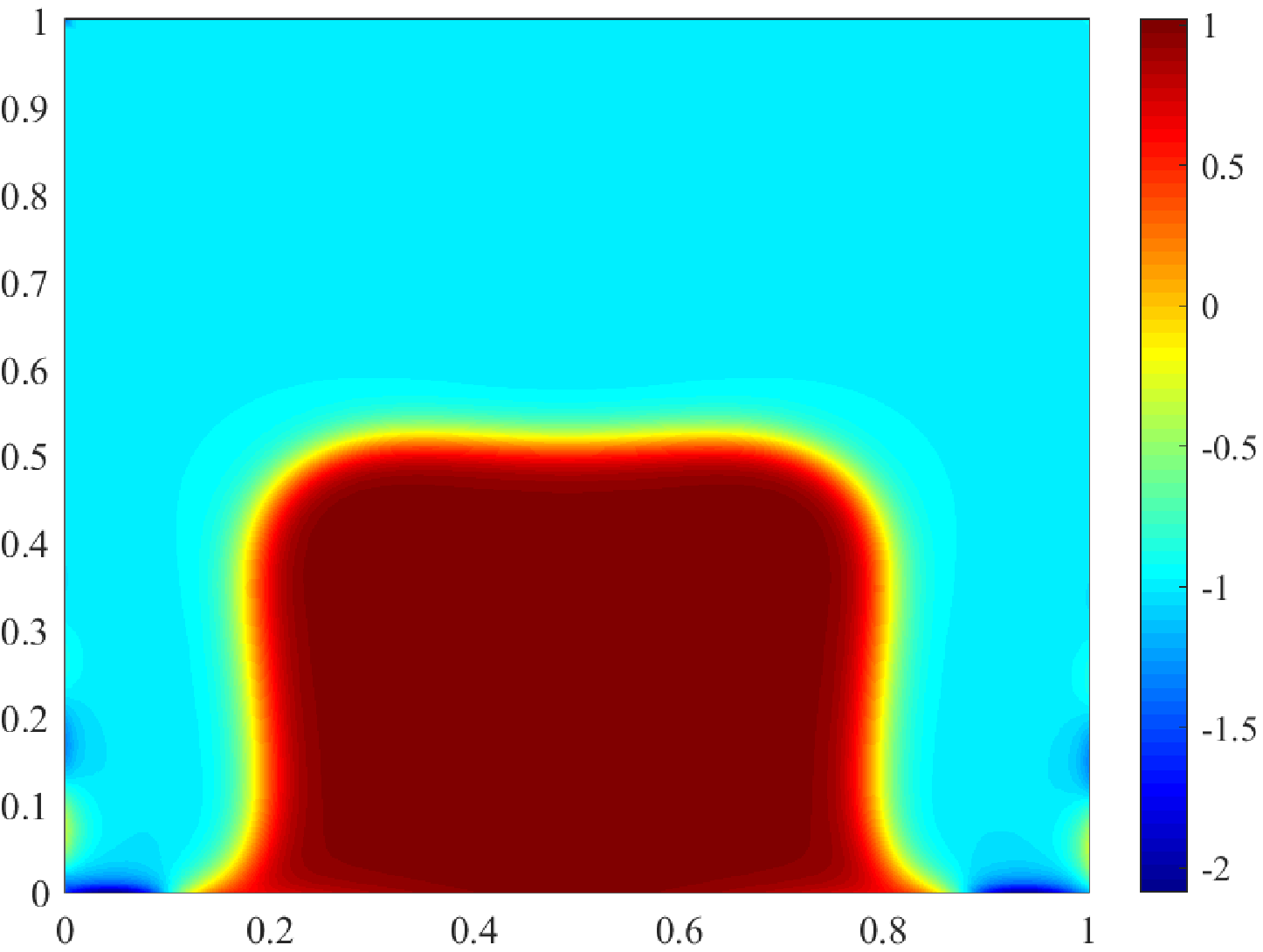}
		\includegraphics[height=0.28\textwidth,width=0.28\textwidth]{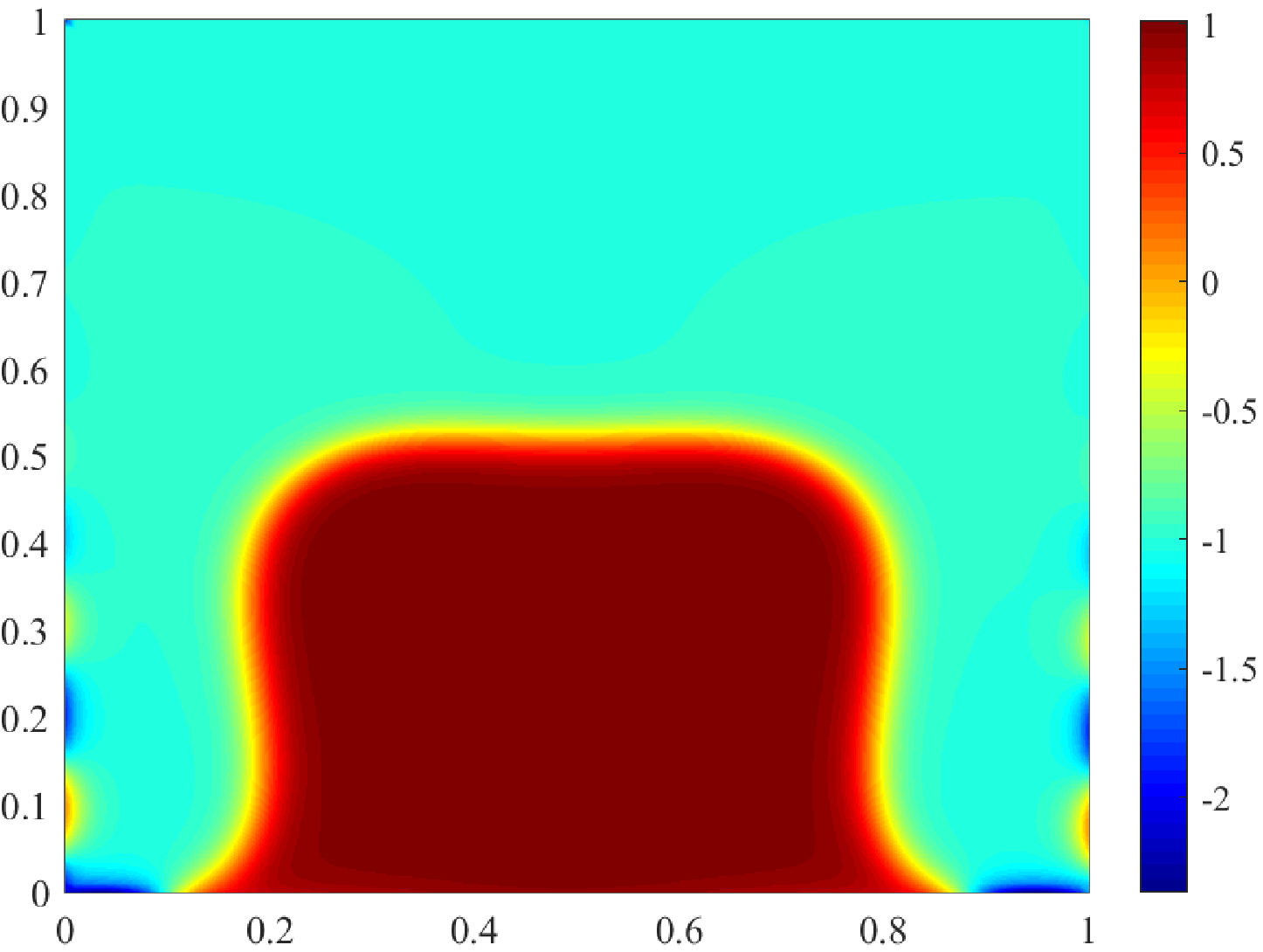}
		\includegraphics[height=0.28\textwidth,width=0.28\textwidth]{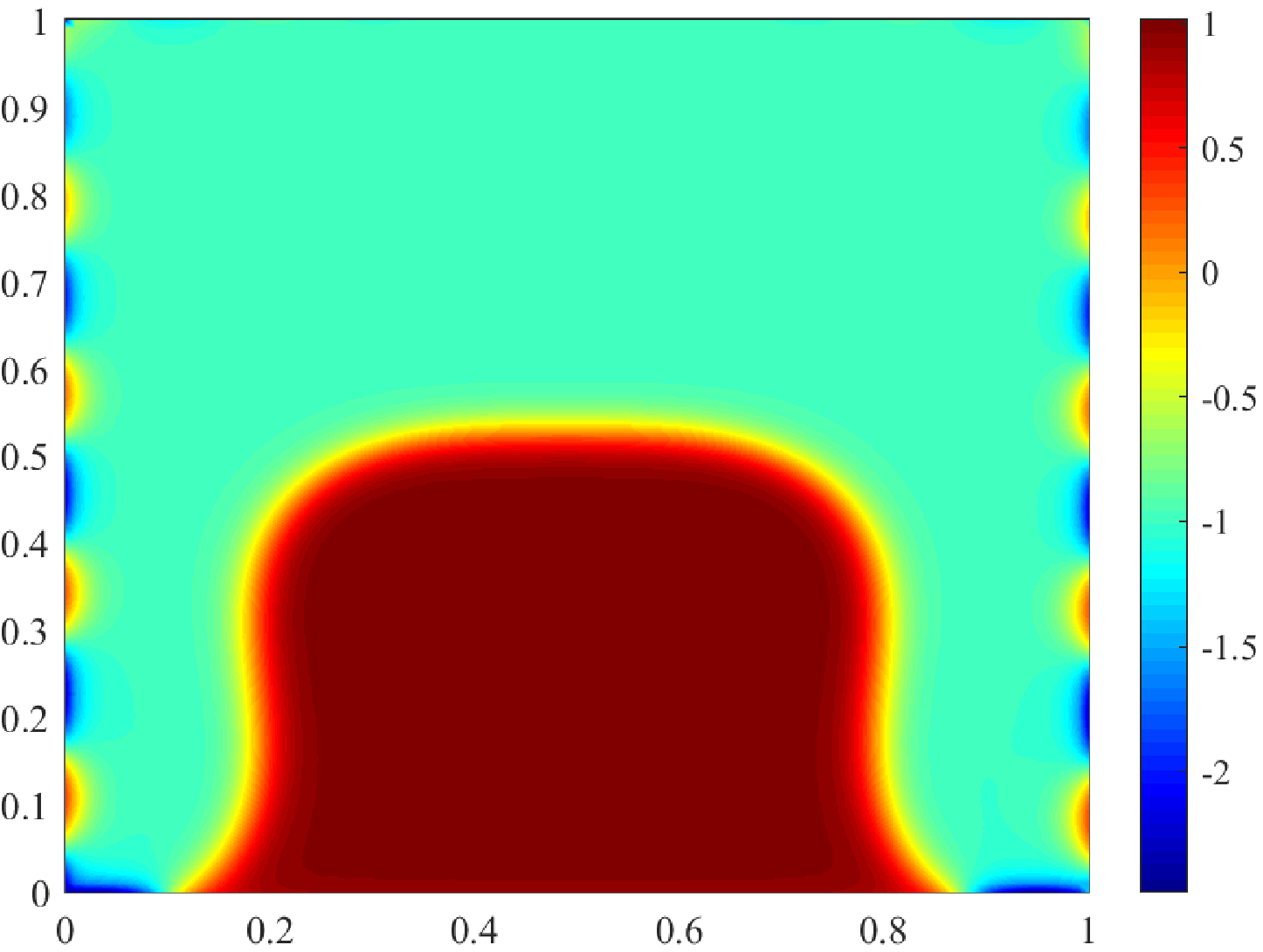}
		\includegraphics[height=0.28\textwidth,width=0.28\textwidth]{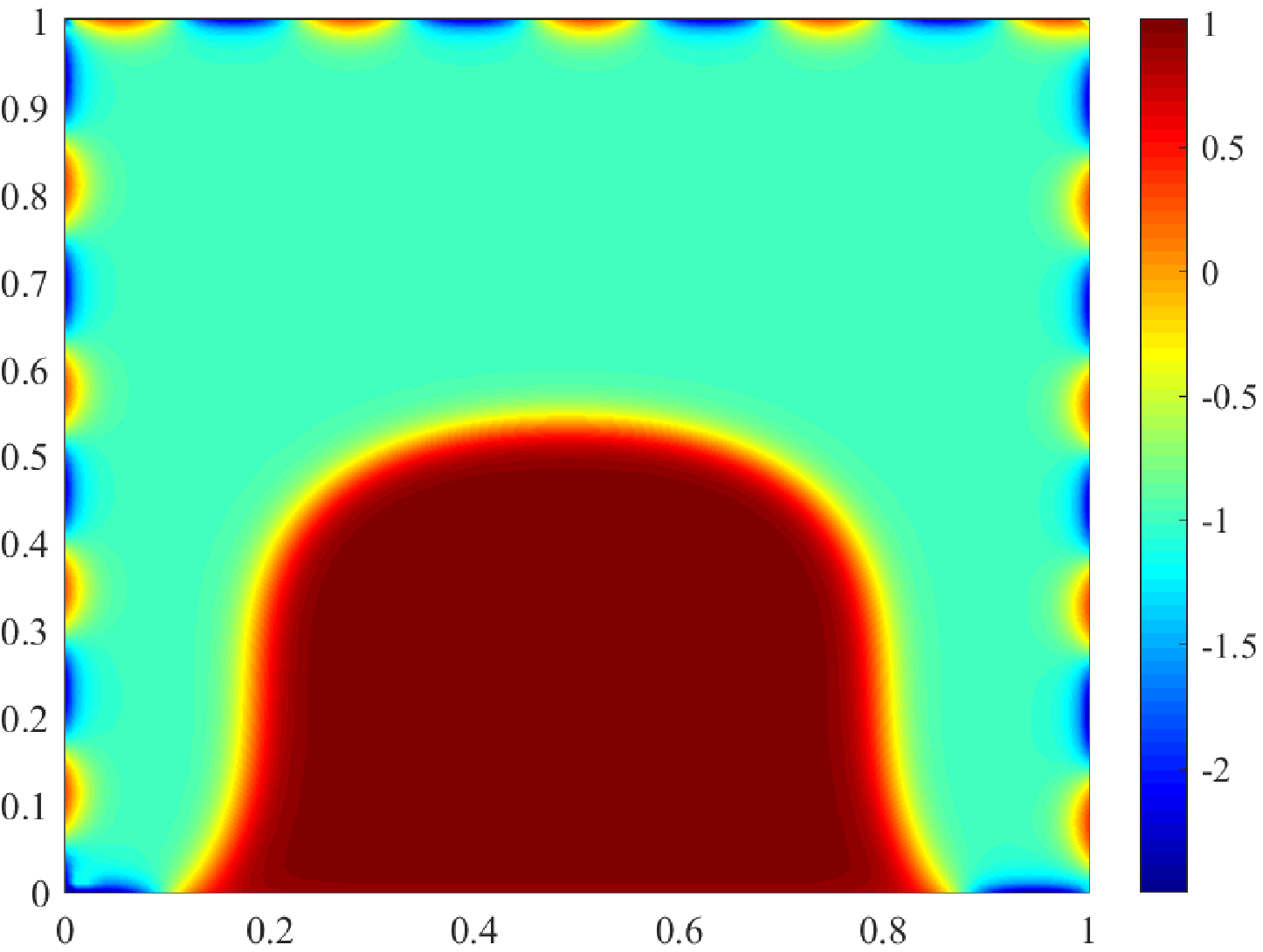}
		\includegraphics[height=0.28\textwidth,width=0.28\textwidth]{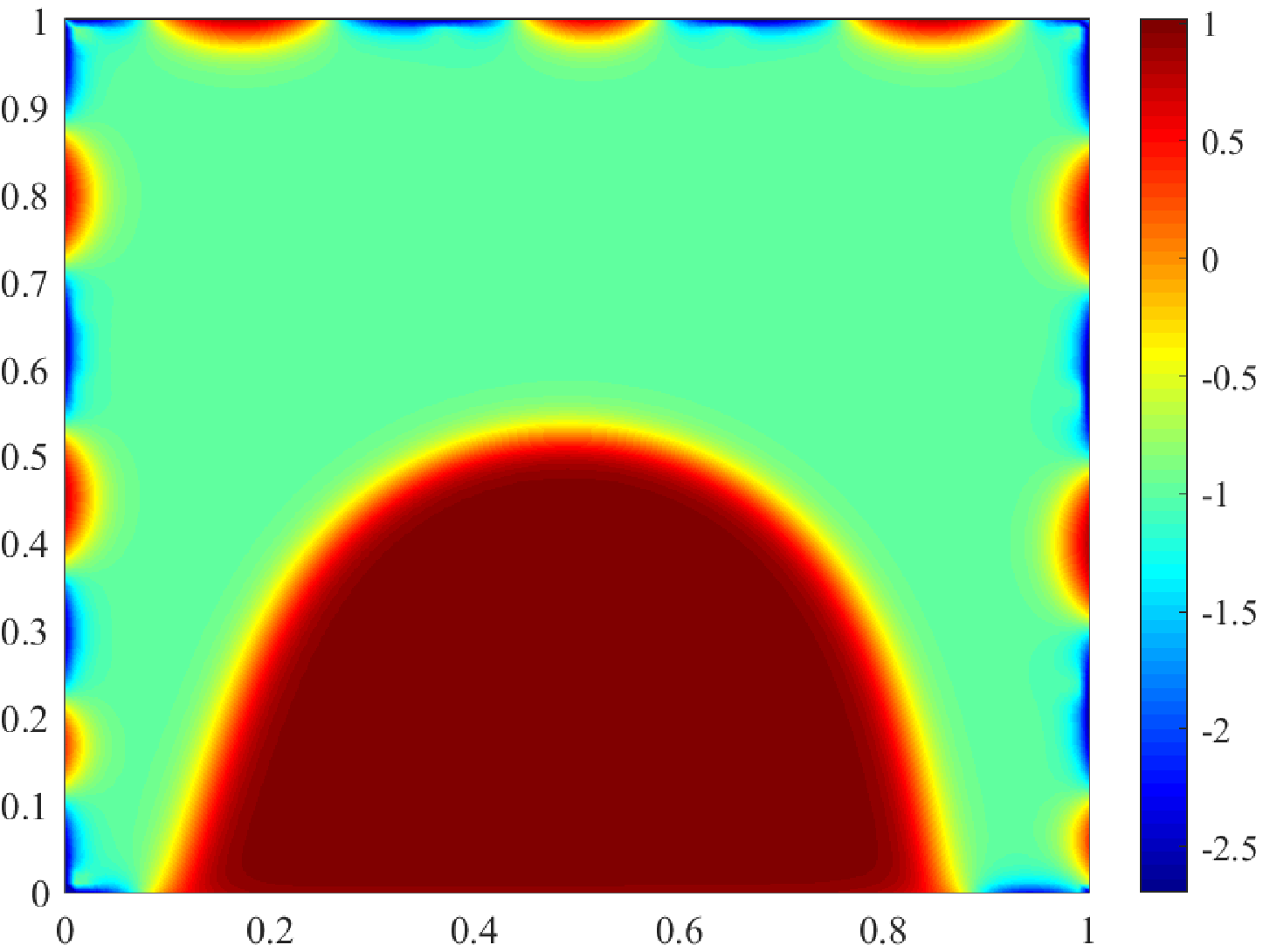}
		\includegraphics[height=0.28\textwidth,width=0.28\textwidth]{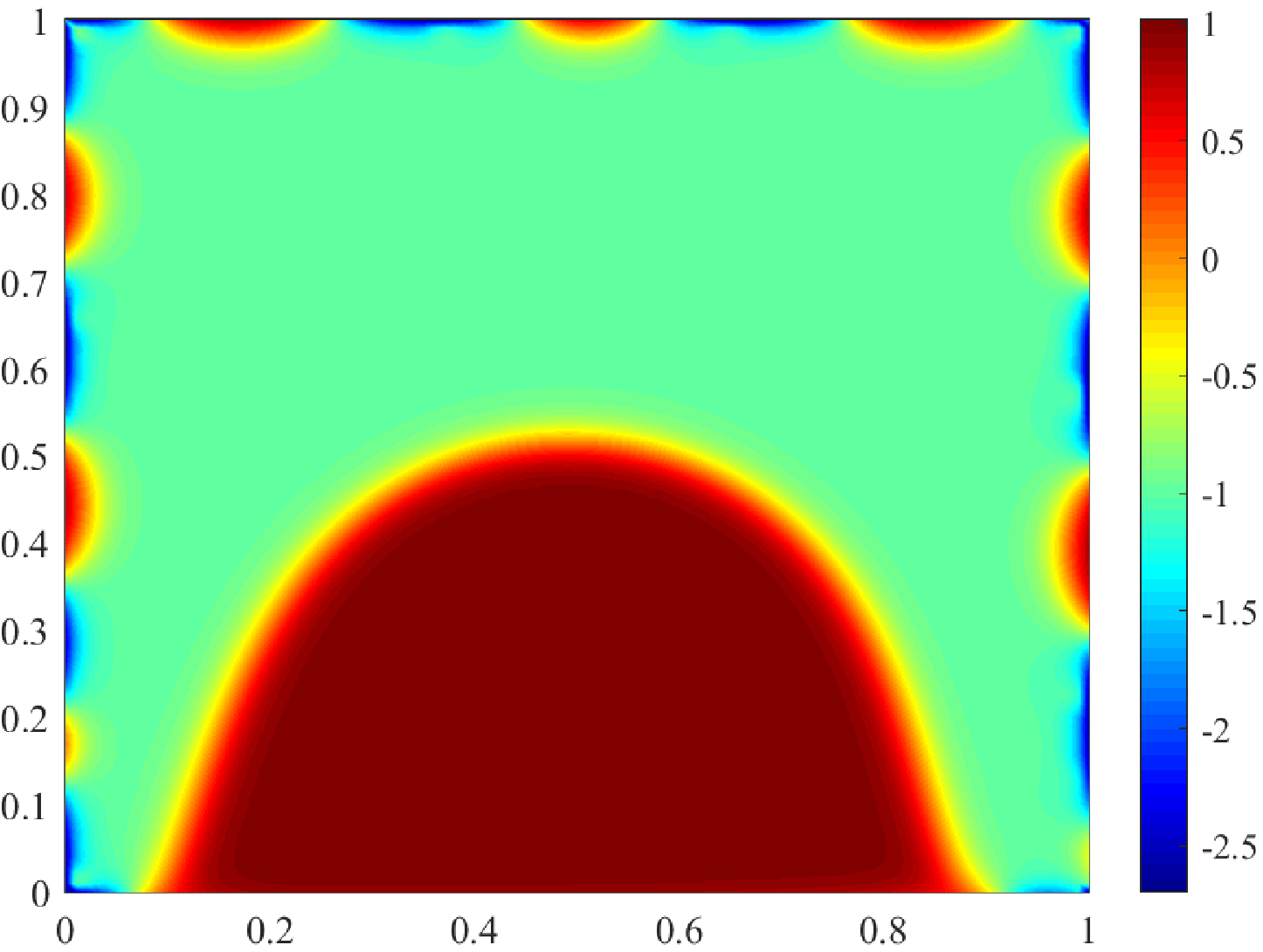}
		\caption{Snapshots of the phase variable $\phi$ at time $t=0.0003,\,0.0005,\,0.001,\,0.002,\,0.008,\,0.01$,\, $(\cos\theta_s=-\frac12)$.}
		\label{fig14}
	\end{figure}

\noindent\textbf{Case 2.}\quad  The Cahn-Hilliard equation with Flory-Huggins potential is widely used to describe the
spinodal decomposition and coarsening of binary mixtures. Namely, for the bulk and surface potential, we consider
the logarithmic Flory-Huggins potential as follows,
\[
F(\phi)=\phi\ln\phi+(1-\phi)\ln(1-\phi)+\theta\phi\ln(1-\phi),
\]
\[
G(\psi)=\psi\ln\psi+(1-\psi)\ln(1-\psi)+\theta\psi\ln(1-\psi),
\]	
where the constant $\theta\textgreater 0$.
 In this case, $\phi$ and $\psi$ represent the mass concentration of one component in the bulk and on the
 boundary, rather than $\phi$ and $\psi$ as the order parameters. Therefore, the concentrations of other
 components in the bulk and on the boundary are denoted by $1-\phi$ and $1-\psi$ respectively. Therefore,
 the corresponding physical correlation interval is $(0, 1)$.
According to the work in \cite{yang2017linear}, we need the regularized logarithmic potential as follows in
order to ensure the logarithmic potential smooth enough. Precisely, for $0 < \zeta \ll 1$,
\begin{numcases}{\hat{F}(\phi)=}{}
		\phi\ln\phi+\frac{(1-\phi)^2}{2\zeta}+(1-\phi)\ln\zeta-\frac{\zeta}{2}+\theta\phi(1-\phi),  &
$\phi \textgreater 1-\zeta, $\nonumber\\
		\phi\ln\phi+(1-\phi)\ln(1-\phi)+\theta\phi(1-\phi), & $\zeta \leq \phi \leq 1-\zeta, $\nonumber\\
		(1-\phi)\ln(1-\phi)+\frac{\phi^2}{2\zeta}+\phi\ln\zeta-\frac{\zeta}{2}+\theta\phi(1-\phi),
& $\phi \textless \zeta, $\nonumber
\end{numcases}
\begin{numcases}{\hat{G}(\psi)=}{}
		\psi\ln\psi+\frac{(1-\psi)^2}{2\zeta}+(1-\psi)\ln\zeta-\frac{\zeta}{2}+\theta\psi(1-\psi),
& $\psi \textgreater 1-\zeta, $\nonumber\\
		\psi\ln\psi+(1-\psi)\ln(1-\psi)+\theta\psi(1-\psi),  & $\zeta \leq \psi \leq 1-\zeta, $\nonumber\\
		(1-\psi)\ln(1-\psi)+\frac{\psi^2}{2\zeta}+\psi\ln\zeta-\frac{\zeta}{2}+\theta\psi(1-\psi),
& $\psi \textless \zeta. $\nonumber
\end{numcases}
Obviously, the advantage of using regularization potential is that the domain of $\hat{F}$ and $\hat{G}$
are $\mathbb{R}$, so we don't need to worry about the overflow caused by any small fluctuation near the
region boundary $(0,1)$ of the numerical solution.

	\begin{figure}[t]
		\centering
		\includegraphics[height=0.28\textwidth,width=0.28\textwidth]{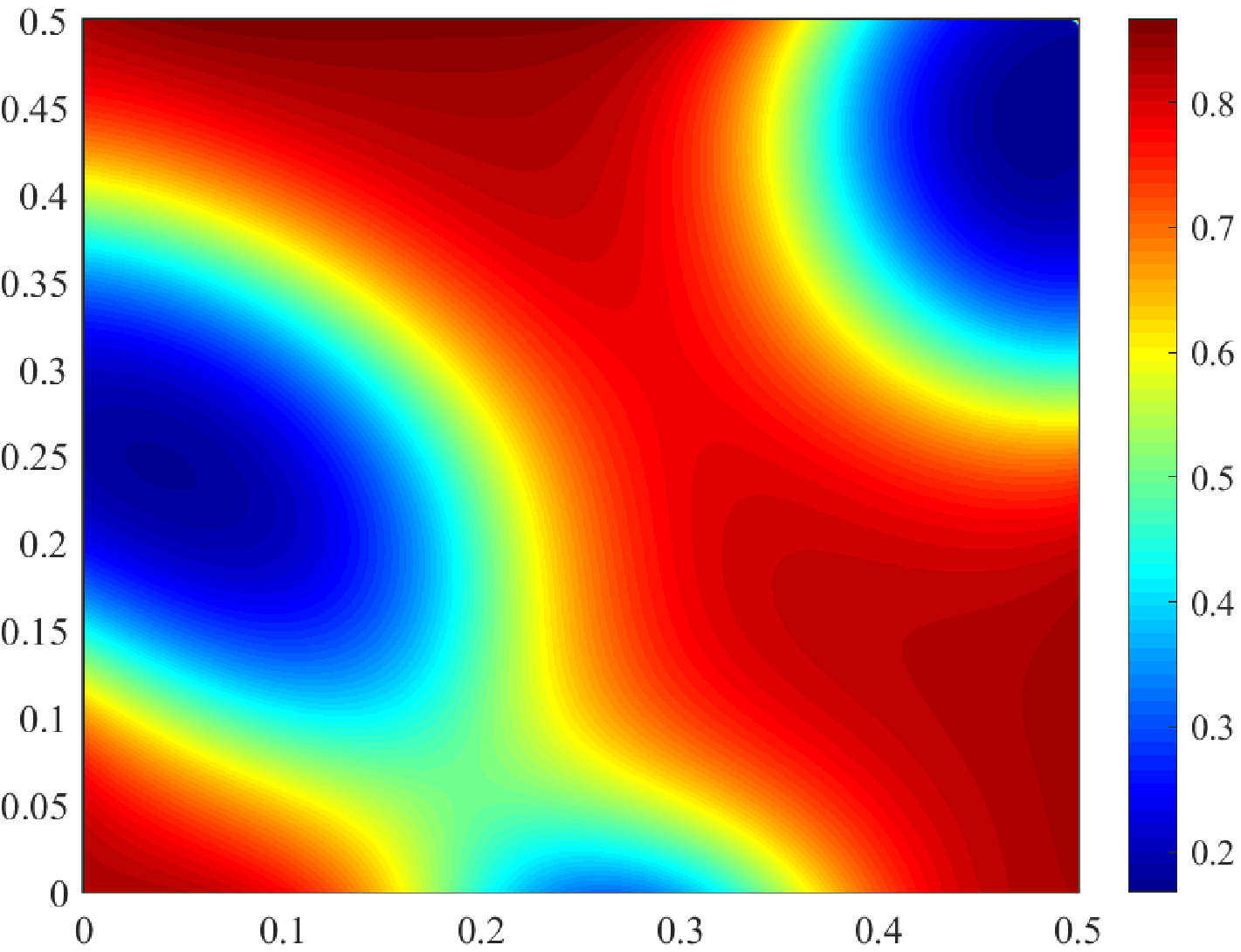}
		\includegraphics[height=0.28\textwidth,width=0.28\textwidth]{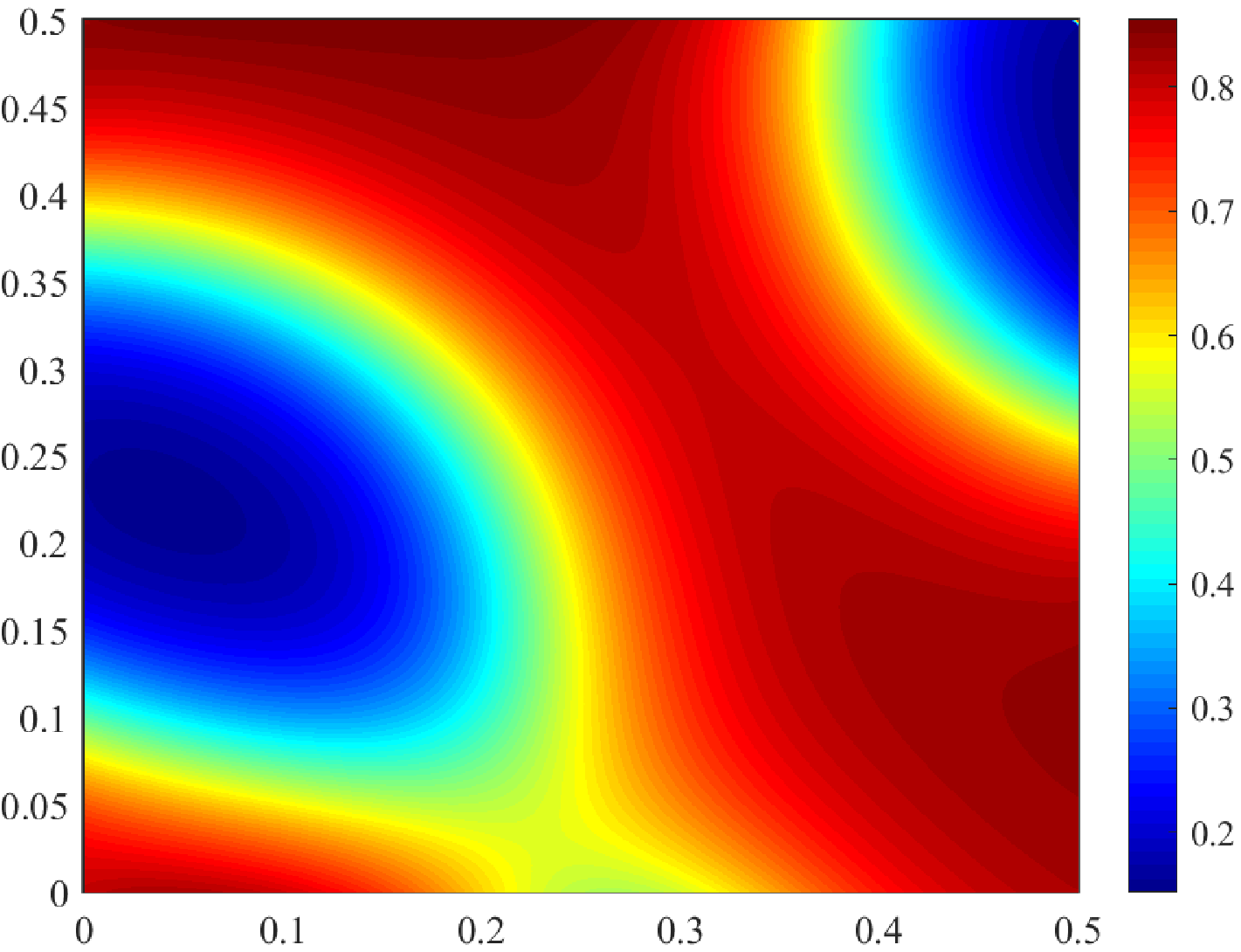}
		\includegraphics[height=0.28\textwidth,width=0.28\textwidth]{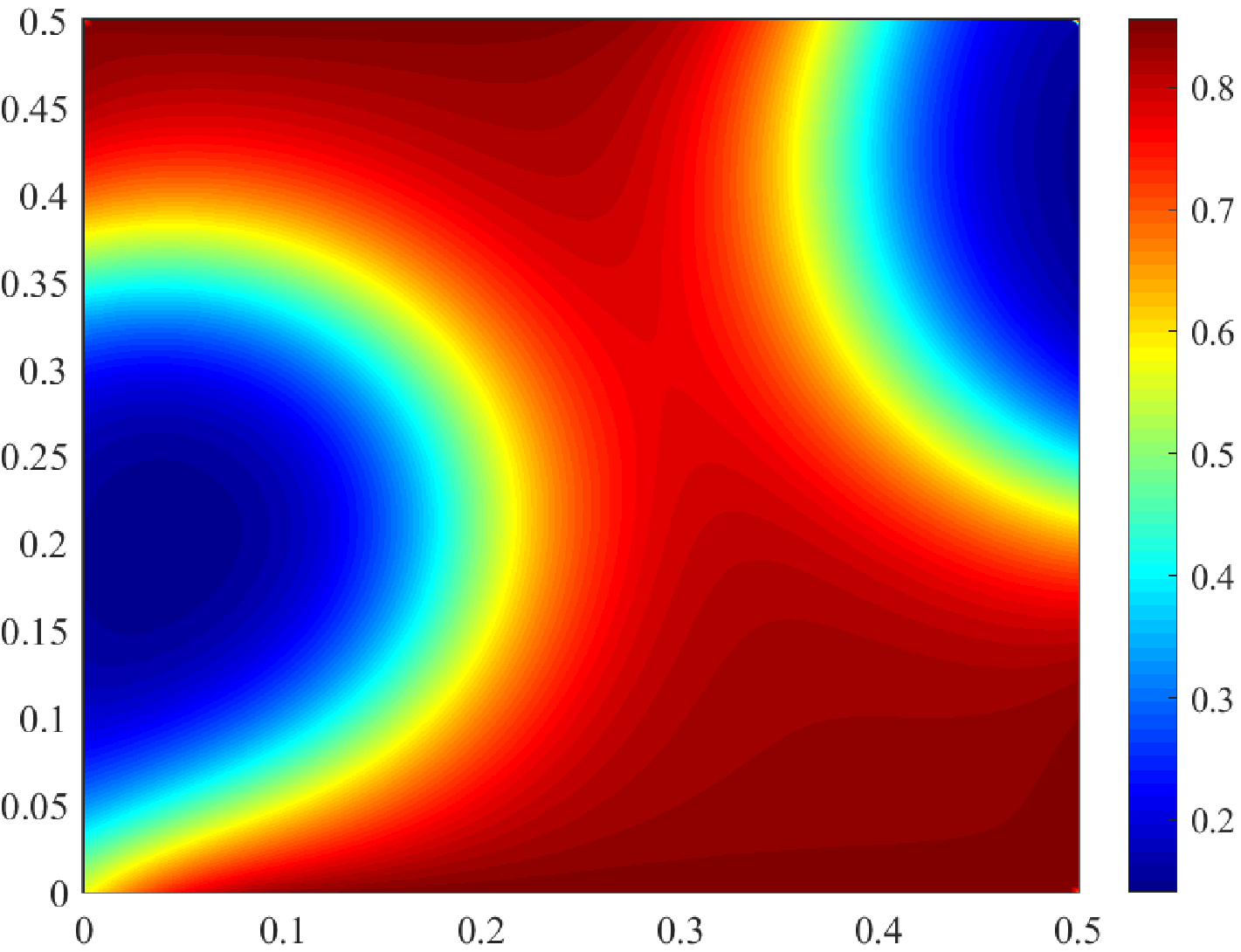}
		\includegraphics[height=0.28\textwidth,width=0.28\textwidth]{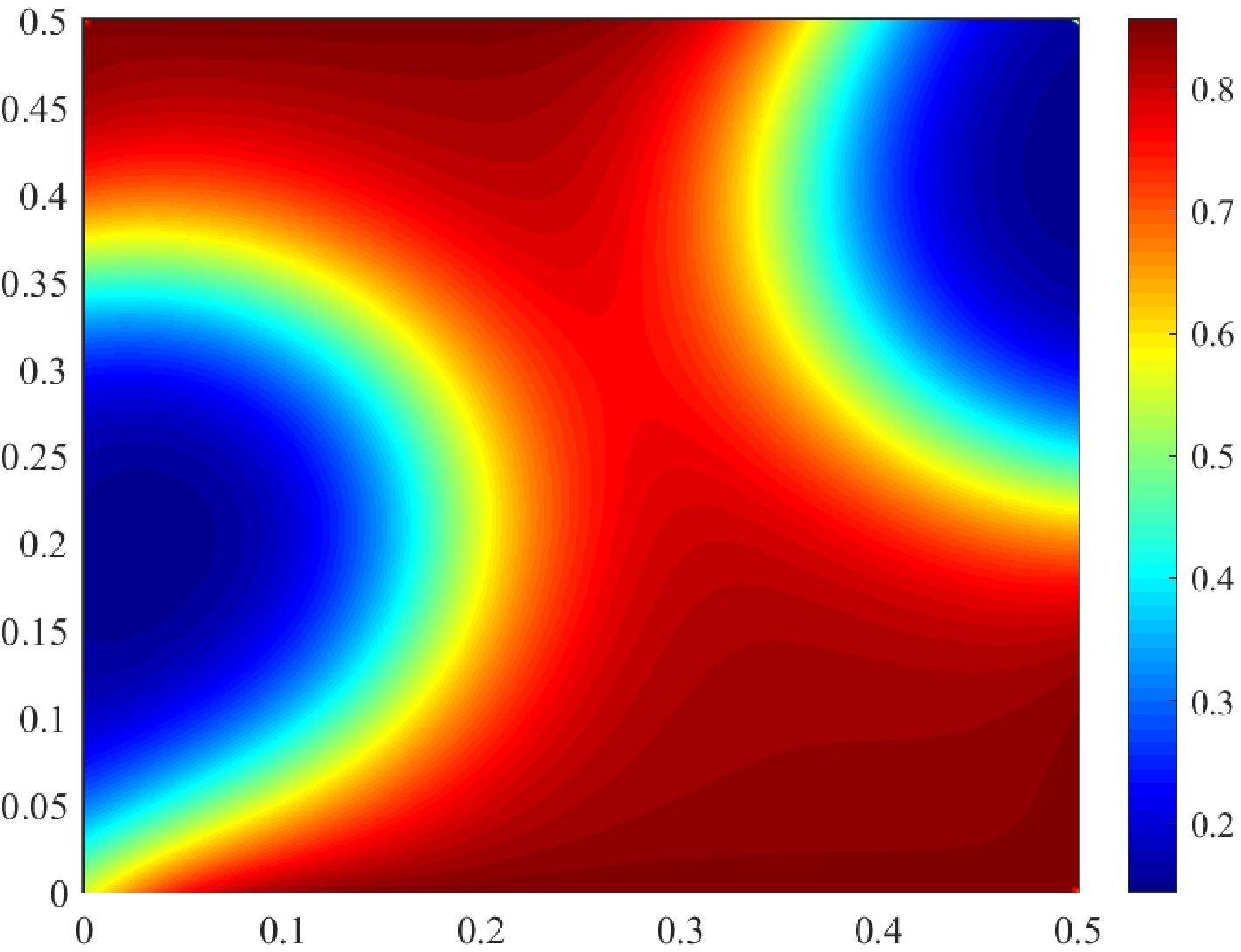}
		\includegraphics[height=0.28\textwidth,width=0.28\textwidth]{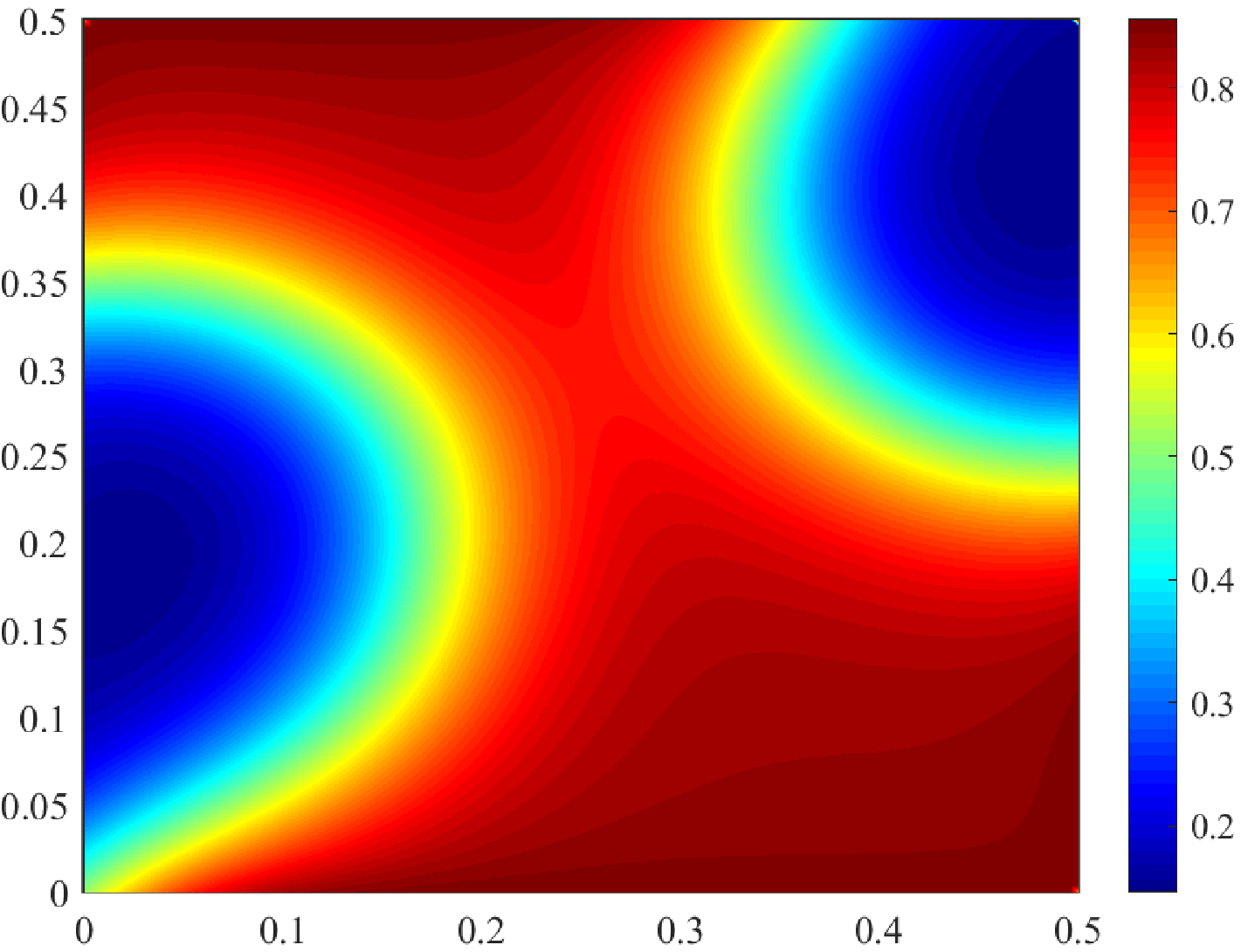}
		\includegraphics[height=0.28\textwidth,width=0.28\textwidth]{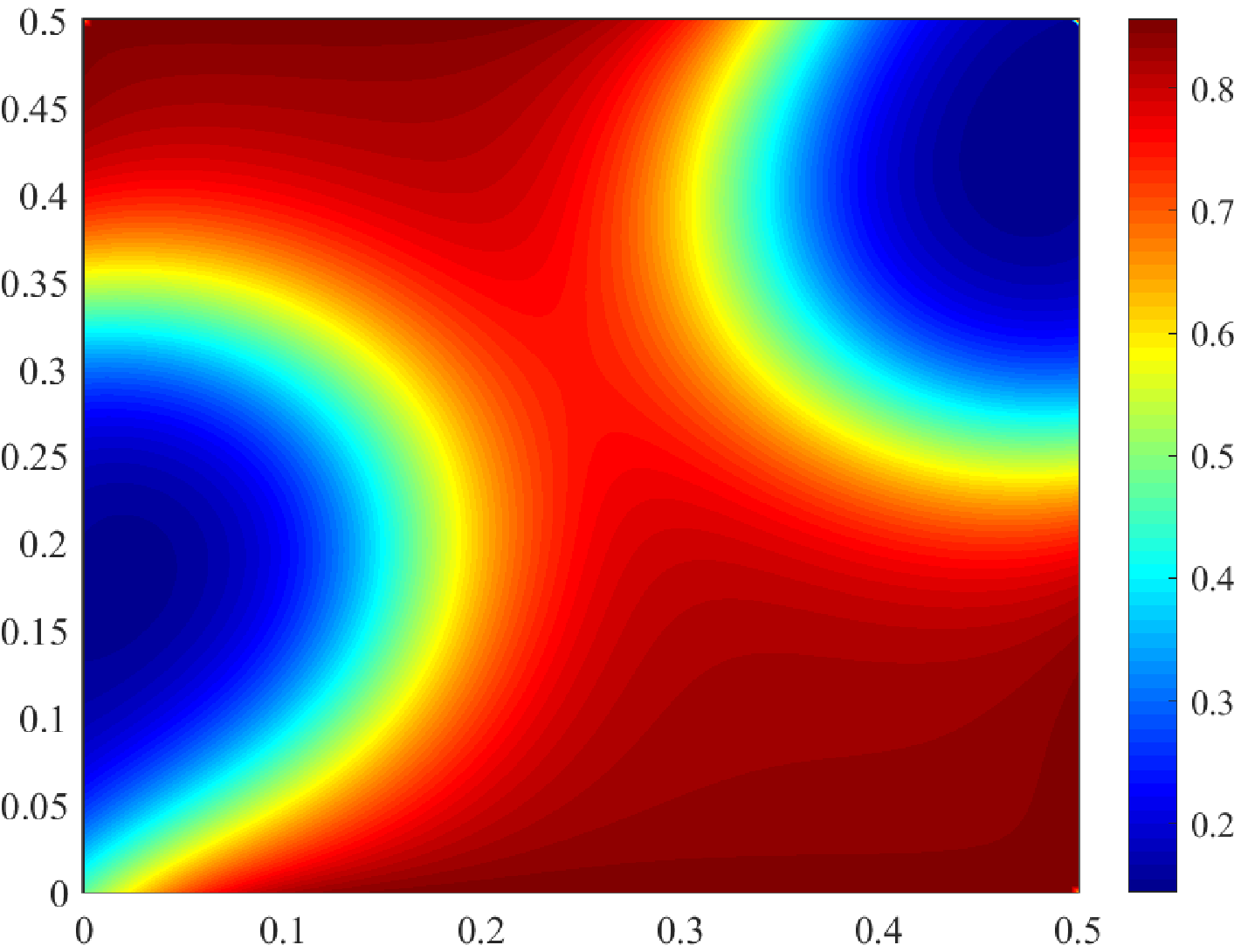}
		\caption{Snapshots of the phase variable $\phi$ at $t=0.005,0.01,0.015,0.02,0.035,0.05$
with the Flory-Huggins potential.}
		\label{fig15}
	\end{figure}
	
\begin{figure}[h]
		\centering
		\includegraphics[height=0.28\textwidth,width=0.28\textwidth]{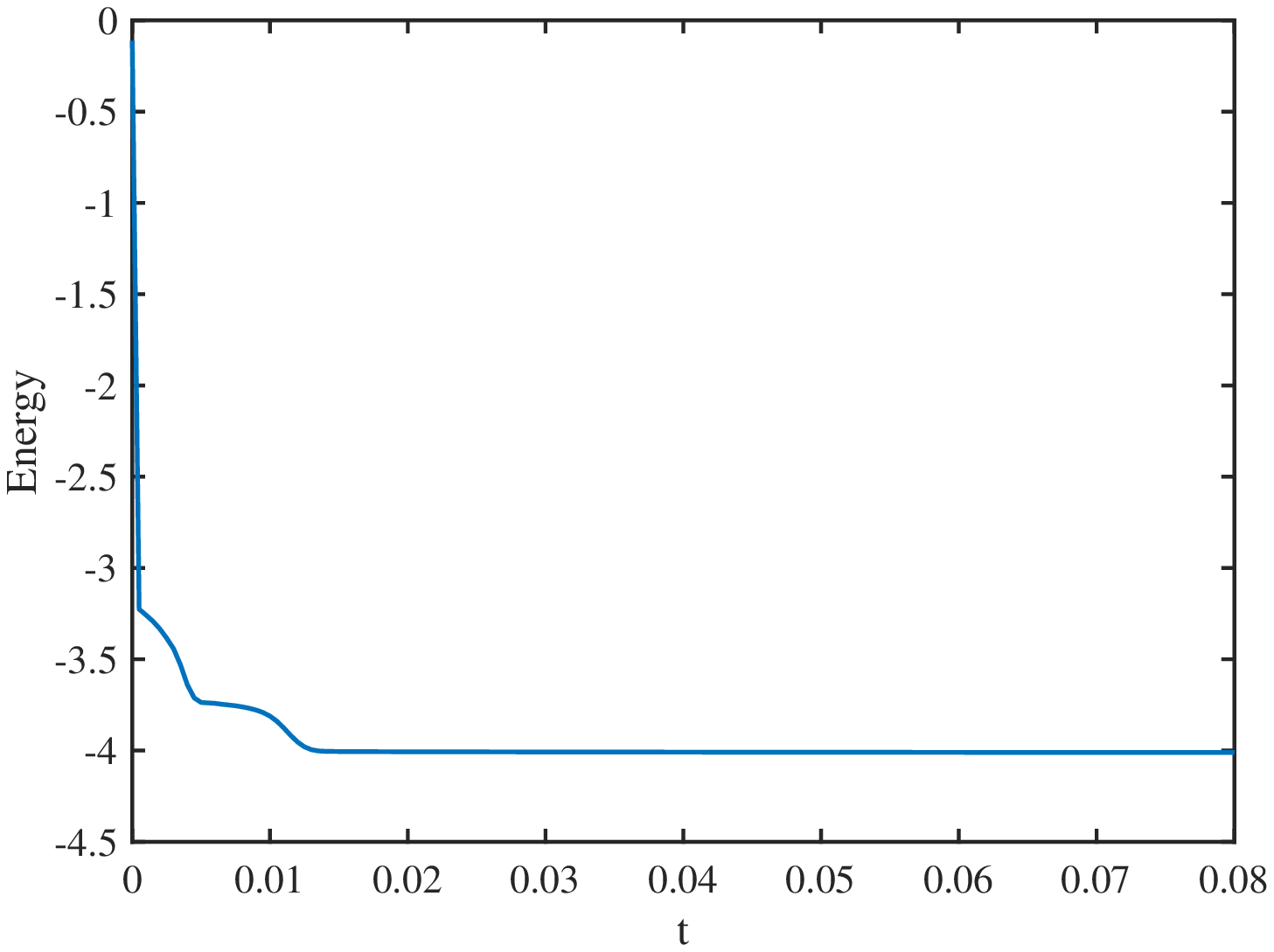}
		\includegraphics[height=0.28\textwidth,width=0.28\textwidth]{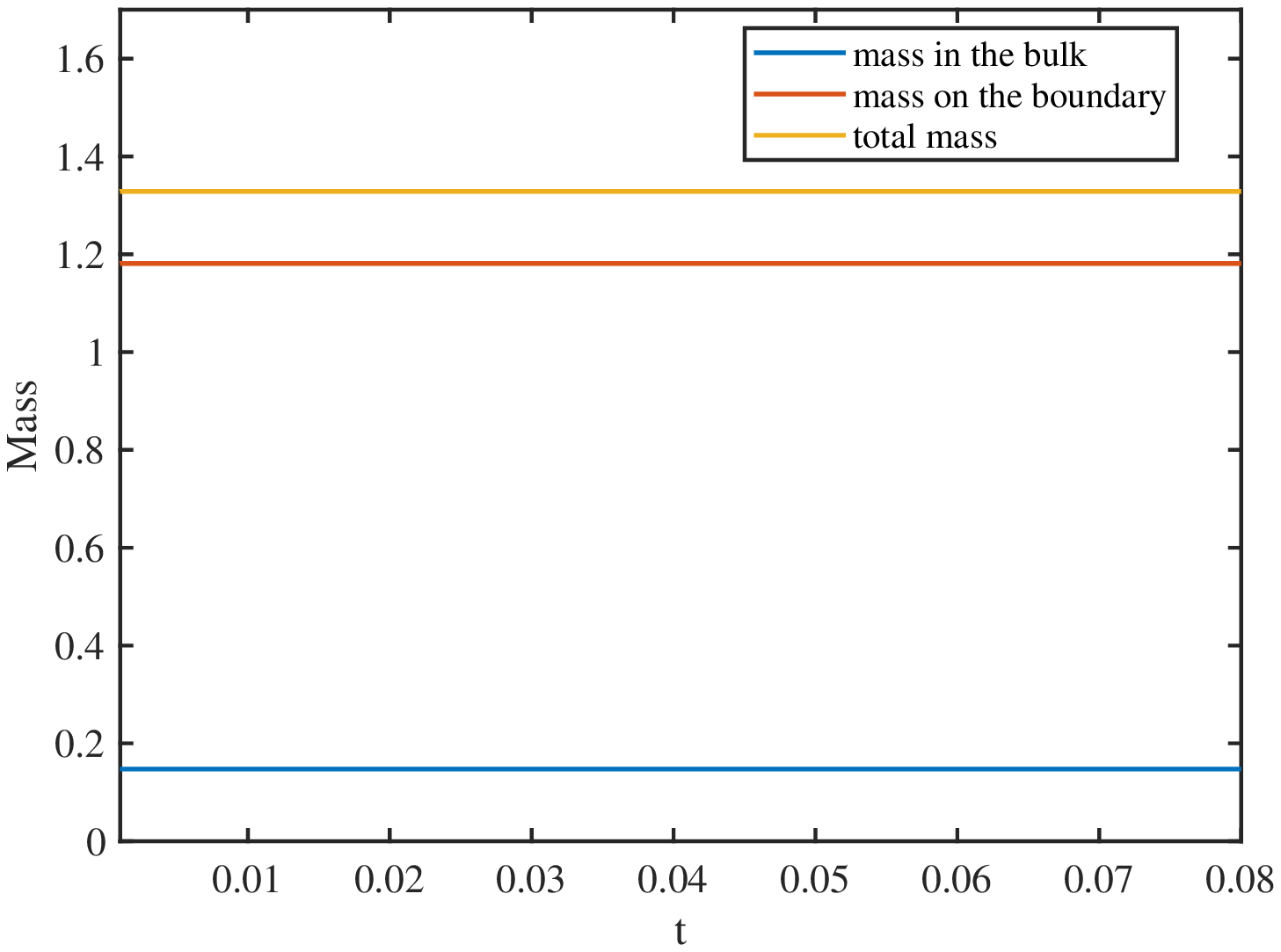}
		\caption{Energy evolution (left); Mass evolution (right) with  the Flory-Huggins potential.}		
		\label{fig16}
	\end{figure}	
Here, we conduct the numerical simulations on the domain $\Omega=[0,0.5]^2 \subset \mathbb{R}^2$. The domain
size and the region is evenly divided into $128\times 128$ grids. The time step $\tau=10^{-4}$. The parameters are set
as $\varepsilon=\delta=0.05,\, \kappa=1,\,  \theta=2.5,\,\zeta=0.005$.
The artificial parameters $A_1=A_2=10,\, B_1=B_2=500$ are used to ensure that the scheme
is stable. The initial data is set as random numbers between $0.4$ and $0.6$.
The numerical results at $t=0.005,\, 0.01,\, 0.015,\, 0.02,\, 0.035,\, 0.05$ are plotted in Figure \ref{fig15}.
It is seen that the numerical solution roughly lies in the interval $(0.1,0.9)$, which makes the Flory-Huggins energy
potential well-defined. The phase field along the boundary is dynamically developed, see Figure \ref{fig15}. However,
both the total mass in the interior domain and on the boundary remain unchanged, see Figure \ref{fig16}. The energy
development is also displayed in Figure \ref{fig16}, again indicating the energy decreasing throughout the computation
and a quick decay at early stage.
\begin{rem}
In the actual numerical computations, we find the proposed BDF2 scheme displays the property of stability energy with the
 stabilizers $A_1,\, A_2,\,B_1$ and $B_2$ much smaller than the theoretical ones \eqref{StaCon1}-\eqref{StaCon3}.
\end{rem}

\section{Conclusions}
\label{sec6}
To the best of our knowledge, we are the first to propose the second-order stabilized semi-implicit
linear scheme for the Cahn-Hilliard equation with dynamic boundary conditions. The nonlinear bulk
forces are treated explicitly with four additional linear stabilization terms:
$A_1\tau\Delta\left(\phi^{n+1}-\phi^{n}\right)$,
$B_1\left(\phi^{n+1}-2\phi^n+\phi^{n-1}\right)$,
$A_2\tau\Delta_\Gamma\left(\psi^{n+1}-\psi^{n}\right)$ and $B_2\left(\psi^{n+1}-2\psi^n+\psi^{n-1}\right)$.
By a serial of estimates both in the bulk and on the boundary, we find the modified total energy decays throughout the time.
We also present a rigorous analysis to obtain an optimal error estimate for the proposed BDF2-type scheme, which is a more challenging
work than the numerical analysis with classical boundary conditions. Numerical experiments with various of initial conditions and
potential functions are presented to verify the stability and accuracy of the scheme, also we find  many interesting phenomena caused by dynamic boundary condition.

\section*{Acknowledgement}
Z.R. Zhang is partially supported by the NSFC No.11871105. The authors would like to thank Prof. Cheng Wang for the helpful discussions.

\bibliographystyle{plain}
\bibliography{ref1}

\vspace{2mm}

\end{document}